\documentclass[a4paper,12pt]{article}
\usepackage{amssymb,amsmath,stmaryrd,mathrsfs,amsthm,xspace}
\usepackage{color}
\definecolor{darkgreen}{rgb}{0,0.45,0} 
\usepackage[pagebackref,colorlinks,citecolor=darkgreen,linkcolor=darkgreen]{hyperref}
\usepackage{mathtools}
\usepackage[all,2cell]{xy}
\UseAllTwocells

\renewcommand{\phi}{\varphi}
\newcommand{\bA}{\ensuremath{\mathbf{A}}\xspace}
\newcommand{\bB}{\ensuremath{\mathbf{B}}\xspace}
\newcommand{\bL}{\ensuremath{\mathbf{L}}\xspace}
\newcommand{\bLL}{\ensuremath{\mathbf{L'}}\xspace}
\newcommand{\beff}{\ensuremath{\mathbf{f}}\xspace}
\newcommand{\sA}{\ensuremath{\mathscr{A}}\xspace}
\newcommand{\sAt}{\ensuremath{\mathscr{A}_\tight}\xspace}
\newcommand{\sAl}{\ensuremath{\mathscr{A}_\loose}\xspace}
\newcommand{\sB}{\ensuremath{\mathscr{B}}\xspace}
\newcommand{\sC}{\ensuremath{\mathscr{C}}\xspace}
\newcommand{\sD}{\ensuremath{\mathscr{D}}\xspace}
\newcommand{\sDt}{\ensuremath{\mathscr{D}_\tight}\xspace}
\newcommand{\sDl}{\ensuremath{\mathscr{D}_\loose}\xspace}

\newcommand{\sF}{\ensuremath{\mathscr{F}}\xspace}
\newcommand{\sFF}{\ensuremath{\mathscr{F}'}\xspace}

\newcommand{\sK}{\ensuremath{\mathscr{K}}\xspace}
\newcommand{\sKt}{\ensuremath{\mathscr{K}_\tight}\xspace}
\newcommand{\sKl}{\ensuremath{\mathscr{K}_\loose}\xspace}
\newcommand{\sV}{\ensuremath{\mathscr{V}}\xspace}
\newcommand{\sX}{\ensuremath{\mathscr{X}}\xspace}
\newcommand{\sY}{\ensuremath{\mathscr{Y}}\xspace}
\newcommand{\cB}{\ensuremath{\mathcal{B}}\xspace}
\newcommand{\Set}{\ensuremath{\mathbf{Set}}\xspace}
\newcommand{\Cat}{\ensuremath{\mathbf{Cat}}\xspace}
\newcommand{\Catg}{\ensuremath{\mathbf{Cat}_g}\xspace}
\newcommand{\Prof}{\ensuremath{\mathbf{Prof}}\xspace}
\newcommand{\bbA}{\ensuremath{\mathbb{A}}\xspace}
\newcommand{\bbB}{\ensuremath{\mathbb{B}}\xspace}
\newcommand{\bbC}{\ensuremath{\mathbb{C}}\xspace}
\newcommand{\bbD}{\ensuremath{\mathbb{D}}\xspace}
\newcommand{\bbE}{\ensuremath{\mathbb{E}}\xspace}
\newcommand{\bbF}{\ensuremath{\mathbb{F}}\xspace}
\newcommand{\bbK}{\ensuremath{\mathbb{K}}\xspace}
\newcommand{\bbX}{\ensuremath{\mathbb{X}}\xspace}

\newcommand{\fbar}{\ensuremath{\overline{f}}\xspace}
\newcommand{\gbar}{\ensuremath{\overline{g}}\xspace}

\newcommand{\wbar}{{\bar{w}}}
\newcommand{\ybar}{\ensuremath{\overline{y}}\xspace}
\renewcommand{\hbar}{\ensuremath{\overline{h}}\xspace}
\newcommand{\kbar}{\ensuremath{\overline{k}}\xspace}
\newcommand{\ph}{\ensuremath{\varphi}}
\newcommand{\phbar}{\ensuremath{\overline{\varphi}}\xspace}

\newcommand{\Psitil}{\ensuremath{\widetilde{\Psi}}}
\DeclareMathOperator\Lan{Lan}
\DeclareMathOperator\Ran{Ran}
\let\iso\cong
\DeclareMathOperator\ob{ob}
\newcommand{\op}{^{\mathrm{op}}}
\newcommand{\co}{^{\mathrm{co}}}
\newcommand{\twocat}{\ensuremath{\mathbf{2Cat}}\xspace}
\newcommand{\Alg}{\text{-}\mathrm{Alg}}
\newcommand{\Coalg}{\text{-}\mathrm{Coalg}}
\newcommand{\bbAlg}{\text{-}\mathbb{A}\mathrm{lg}}
\newcommand{\bbCoalg}{\text{-}\mathbb{C}\mathrm{oalg}}
\newcommand{\bbProf}{\ensuremath{\mathbb{P}\mathrm{rof}}\xspace}

\newcommand{\Ps}{\mathrm{Ps}}
\newcommand{\Oplax}{\mathrm{Oplax}}
\newcommand{\Nat}{\mathrm{Nat}}
\newcommand{\bbNat}{\mathbb{N}\mathrm{at}}

\newcommand{\bbLax}{\mathbb{L}\mathrm{ax}}
\newcommand{\bbPs}{\mathbb{P}\mathrm{s}}
\newcommand{\bbOplax}{\mathbb{O}\mathrm{plax}}
\newcommand{\maps}{\colon}
\newcommand{\lto}{\rightsquigarrow}
\renewcommand{\op}{^{\mathrm{op}}}
\newcommand{\tight}{\tau}
\newcommand{\loose}{\lambda}
\DeclareSymbolFont{bbold}{U}{bbold}{m}{n}
\DeclareSymbolFontAlphabet{\mathbbb}{bbold}
\newcommand{\bbDelta}{\ensuremath{\mathbbb{\Delta}}}
\newcommand{\bbone}{\ensuremath{\mathbbb{1}}\xspace}
\newcommand{\bbtwo}{\ensuremath{\mathbbb{2}}\xspace}
\newcommand{\bbthree}{\ensuremath{\mathbbb{3}}\xspace}
\newcommand{\dom}{\textnormal{dom}}
\newcommand{\cod}{\textnormal{cod}}

\newcommand{\Mnd}{\mathrm{Mnd}}

\newcommand{\QF}{\ensuremath{\mathcal{Q}}\xspace}
\newcommand{\RF}{\ensuremath{\mathcal{R}}\xspace}

\newcommand{\pto}{\relbar\joinrel\mapstochar\joinrel\rightarrow}

\newcommand{\ls}[1]{\Lambda(#1)}

\newcommand{\llim}[1]{\ensuremath{\overline{\{#1\}}}}

\makeatletter
\def\defthm#1#2{%
  \newtheorem{#1}{#2}[section]%
  \expandafter\def\csname #1autorefname\endcsname{#2}%
  \expandafter\let\csname c@#1\endcsname\c@thm}
\newtheorem{thm}{Theorem}[section]

\defthm{cor}{Corollary}
\defthm{lem}{Lemma}
\defthm{prop}{Proposition}
\theoremstyle{definition}
\defthm{eg}{Example} 
\defthm{defn}{Definition}
\defthm{rmk}{Remark}
\makeatother

\newdir{ >}{{}*!/-10pt/@{>}}    
\title{Enhanced 2-categories and limits for lax morphisms}
\author{Stephen Lack \\
Mathematics Department \\
Macquarie University NSW 2109 \\
Australia \\
{\tt steve.lack@mq.edu.au}
\and
Michael Shulman \\
Department of Mathematics \\
University of California, San Diego \\
9500 Gilman Dr. \#0112 \\
San Diego, CA 92093-0112 \\
U.S.A.\\
{\tt mshulman@ucsd.edu}}
\date{}

\begin{document}
\maketitle

\begin{abstract}
 We study limits in 2-categories whose objects are categories with extra structure and whose morphisms are functors preserving the structure only up to a coherent comparison map, which may or may not be required to be invertible. This is done using the framework of 2-monads. In order to characterize the limits which exist in this context, we need to consider also the functors which do strictly preserve the extra structure. We show how such a 2-category of weak morphisms which is ``enhanced'', by specifying which of these weak morphisms are actually strict, can be thought of as category enriched over a particular base cartesian closed category \sF. We give a complete characterization, in terms of \sF-enriched category theory, of the limits which exist in such 2-categories of categories with extra structure.
\end{abstract}
\enlargethispage{\baselineskip}
\enlargethispage{\baselineskip}
\enlargethispage{\baselineskip}

\section{Introduction}
\label{sec:introduction}

Just as sets with algebraic structure are often conveniently described
as the algebras for a monad, \emph{categories} with algebraic
structure are often conveniently described as the algebras for a
\emph{2-monad} (see~\cite{bkp:2dmonads}).  By a 2-monad we mean a
\emph{strict} 2-monad, i.e.\ a \Cat-enriched monad, and likewise its
algebras satisfy their laws strictly.  Experience shows that even when
the ``algebraic structure'' borne by a category or family of categories
satisfies laws only up to
specified isomorphisms, it is always possible, and usually more
convenient, to describe it using strict algebras for a strict 2-monad.

Thus, for example, there are 2-monads on the 2-category \Cat whose algebras are monoidal
categories, strict monoidal categories, symmetric monoidal categories,
categories with finite products, categories with finite products and finite coproducts connected by a distributive law, categories with finite limits,
categories with countable limits, and so on. A structure such as that of 
cartesian closed category is more delicate, since the internal hom is 
contra\-variant in the first variable; to deal with it, one can work 
not over \Cat itself, but over the 2-category \Catg of categories, 
functors, and natural isomorphisms. There is a 2-monad on \Catg whose
algebras are cartesian closed categories; similarly there are 
2-monads on \Catg for monoidal closed categories, symmetric monoidal 
closed categories, and elementary toposes.

Moreover, weak algebras can often be reduced directly to strict
algebras.  In good situations, such as when the base 2-category \sK is
locally presentable and the 2-monad $T$ has a rank, there is another
2-monad $T'$ whose strict algebras are the weak $T$-algebras.  (This
follows from the general theory of ``weak morphism classifiers'',
which we will recall in \S\ref{sec:weak-morph-class}, using an
auxiliary 2-monad whose algebras are 2-monads.)

However, even though we can usually consider only strict algebras for
strict 2-monads, no such simplification is possible for \emph{morphisms}
between such algebras; the strict morphisms are generally too strict
and we must consider weaker notions.  Thus, for any 2-monad $T$ on a
2-category \sK, in addition to the 2-category $T\Alg_s$ of
$T$-algebras and strict $T$-morphisms (this is the \Cat-enriched
Eilenberg-Moore category), we have the 2-categories $T\Alg$,
$T\Alg_l$, and $T\Alg_c$, whose objects are (strict) $T$-algebras and
whose morphisms are \emph{pseudo}, \emph{lax}, and \emph{colax}
$T$-morphisms, respectively.  Pseudo $T$-morphisms are defined to
preserve $T$-algebra structure up to a (suitably coherent) isomorphism, while lax and
oplax ones preserve it only up to a transformation in one direction or
the other.  Lax monoidal functors, for instance, are ubiquitous in
mathematics, pseudo ones are also common, and strict ones are quite
rare.  The properties of the 2-categories $T\Alg$, $T\Alg_l$, and
$T\Alg_c$ are therefore of interest; our present concern is with the
limits that they admit, in the 2-categorical sense
(cf.~\cite{kelly:2limits}).  (Of course, $T\Alg_s$ admits all limits
that \sK\ does, by general enriched category theory.)

In the case of $T\Alg$, this question was answered
in~\cite{bkp:2dmonads}.  For any 2-monad $T$ on a complete 2-category
\sK, the 2-category $T\Alg$ admits PIE-limits; that is, all limits
constructible from products, inserters, and equifiers
(see~\cite{pr:pie-limits}).  In particular, this includes all lax
limits and pseudo limits, and therefore all bilimits; thus from the
``fully weak'' point of view of bicategories, $T\Alg$ has all the
limits one might ask for.  Moreover, the PIE-limits in $T\Alg$ also
satisfy an additional strictness property: for each of 
products, inserters, and equifiers, there is a specified set of limit
projections each of which is a strict $T$-morphism, and which jointly
``detect strictness'' of $T$-morphisms. 

For
$T\Alg_l$ and $T\Alg_c$ the question is more difficult, and the existing answers
less complete.  It was shown in~\cite{lack:lim-lax} that $T\Alg_l$
admits the following limits whenever \sK does:
\begin{itemize}
\item All \emph{oplax} limits.
\item All limits of diagrams consisting of strict $T$-morphisms (that
  is, the inclusion functor $T\Alg_s \to T\Alg_l$ preserves limits).
\item Equifiers of pairs of 2-cells $\alpha,\beta\maps
  g\rightrightarrows f$ where $g$ is a strict $T$-morphism.
\item Inserters of pairs of morphisms $g,f\maps A\rightrightarrows B$
  where $g$ is a strict $T$-morphism.
\item Comma objects $(g/f)$ where $g$ is a strict $T$-morphism.
\end{itemize}
Once again, each of these limits has the property that some or
all of the limit projections are strict $T$-morphisms, and that they
jointly ``detect strictness.''  

The main result of this paper is a complete characterization of those
limits which lift to $T\Alg_l$ for any 2-monad $T$.
As is evident
from the examples above, such a characterization must involve, not
only $T\Alg_l$ itself, but its relationship to $T\Alg_s$.  The obvious
relationship is the existence of the inclusion functor $T\Alg_s \to
T\Alg_l$, which is the identity on objects, faithful (on 1-morphisms), and locally fully faithful.
A fundamental observation is that the following notions are essentially
equivalent:
\begin{enumerate}
\item A 2-functor which is the identity on objects, faithful, and
  locally fully faithful.
\item A category enriched over the cartesian closed category \sF\
  whose objects are functors that are fully faithful and injective on objects.
  We sometimes call such functors \emph{full embeddings}.
\end{enumerate}
Therefore, rather than viewing $T\Alg_s$ and $T\Alg_l$ as two
2-categories related by a functor, we can combine them together into a
\emph{single} structure $T\bbAlg_l$, which happens to be a category
enriched over \sF, i.e.\ an \textbf{\sF-category}.  Intuitively, an
\sF-category has objects, two types of morphism of which one is a
special ``stricter'' case of the other, and 2-cells between these
morphisms.

In working with \sF-categories, we of course need words for the two
types of morphism.  In the \sF-category $T\bbAlg_l$ they are called ``strict'' and
``lax,'' but there are also other important \sF-categories one might
consider, such as $T\Alg_s\to T\Alg$ (where they are ``strict'' and
``pseudo'') and $T\Alg\to T\Alg_l$ (where they are ``pseudo'' and
``lax'').  Thus, in order to avoid favoring one of these cases in our
terminology, we introduce new words for the two types of morphism in
a general \sF-category: we call them \textbf{tight} and \textbf{loose}.

\begin{rmk}
  Since a full embedding is injective on objects, a loose morphism can
  ``be tight'' in at most one way.  More generally, we could consider
  \sFF-categories, where \sFF is the cartesian closed category whose
  objects are fully faithful functors.  \sFF-categories correspond to
  2-functors which are merely the identity on objects and locally
  fully faithful.  In an \sFF-category, it may be possible to make a
  given loose morphism into a tight morphism in more than one way
  (although all such ``tightenings'' will be isomorphic).

  Since ``tightenings are unique'' in the fundamental examples such as
  $T\Alg_s\to T\Alg_l$, and since it is slightly easier to say that
  such-and-such a morphism ``is tight'' than to say that it ``can be
  made into a tight morphism'' or ``is equipped with a tight morphism
  structure'', we have chosen to work in the slightly more restrictive
  setting of \sF-categories.  However, it is not hard to generalize
  all of our results to \sFF-categories.
\end{rmk}

The introduction of \sF-categories allows us to make use of the
familiar and powerful language of enriched category theory when
discussing the limits that lift to $T\Alg_l$.  In this language, these
limits will be characterized by certain \emph{weights} $\Phi\maps
\bbD\to\bbF$, where \bbD\ is a small \sF-category and \bbF\ denotes
\sF\ regarded as an \sF-category.  All the limits mentioned above that
exist in $T\Alg_l$, together with strictness and strictness-detection
of some of their projections, can now be described precisely as
certain \sF-weighted limits in the \sF-category $T\bbAlg_l$.

Of course we then have to actually give the characterization of the
\sF-weighted limits that lift.  This turns out to be a refinement of
the notion of \emph{flexible} limit from~\cite{bkps:flexible}, which
we now describe.  Recall that for any \Cat-weight $\Phi\maps \sD\to\Cat$,
there is another \Cat-weight $Q\Phi\maps \sD\to\Cat$ such that
\emph{pseudo} natural transformations $\Phi\lto\Psi$ are
in bijection with \emph{strict} natural transformations
$Q\Phi\to\Psi$.
A weight is said to be
\emph{flexible} if the projection $q\maps Q \Phi\to \Phi$ has a
strictly natural section (it always has a pseudonatural section).  Every
PIE-weight is flexible, but the converse is false: the splitting of
idempotents is a flexible weight that is not a PIE-weight (and in a
certain sense it is the ``only'' such, since together with PIE-limits
it generates all flexible limits).

Now $Q$ defines a comonad on the 2-category of \Cat-weights, which is in fact
pseudo-idempotent in the sense of~\cite{kl:property-like}---so that in particular, a weight admits at most one $Q$-coalgebra structure, up to unique isomorphism.
In fact we shall show that 
\emph{the PIE-weights are precisely the $Q$-coalgebras}.
(This has independently and separately been observed by John Bourke
and by Richard Garner.) Note that being flexible is
``half'' of being a $Q$-coalgebra.  What is missing is an
associativity axiom for the section $s\maps \Phi\to Q \Phi$, and it
turns out to be exactly this additional axiom which guarantees that
$\Phi$-weighted limits lift from \sK\ to $T\Alg$ for any $T$.

To generalize this statement to $T\Alg_l$, we first have to replace
2-categories by \sF-categories and then pseudo morphisms by lax ones.
Roughly, a \emph{pseudo \sF-transformation} consists of a pseudo
transformation on the loose morphisms which becomes strictly natural
when restricted to the tight ones, and likewise for lax and oplax
\sF-transformations (see \S\ref{sec:wf-transf}).  As is the case for
2-categories, such transformations are classified by \sF-comonads
$\QF_p$, $\QF_l$, and $\QF_c$, respectively.

It turns out that the \sF-limits which lift to $T\bbAlg_w$, where $w$
is one of $p,l,c$, are always $\QF_\wbar$-coalgebras, where $\wbar$ denotes
$w$ with sense reversed: $\bar{p}=p$, $\bar{l}=c$, and $\bar{c}=l$. But this
$\QF_\wbar$-coalgebra structure is not quite enough for the \sF-limits to lift.
There is also an additional ``tightness'' condition; this is what ensures
that the projections detect tightness, as is necessary for an
\sF-limit.
We call these ($w$-)\textbf{rigged} weights; they provide our promised
characterization of the limits which lift to $T\bbAlg_w$.

By the phrase \emph{enhanced 2-category theory} in the title, we mean
to indicate the study of structures akin to \sF-categories, which
combine a 2-category with additional data, to be studied as a unit.
Two existing notions that can also be viewed as enhanced
2-categories are \emph{proarrow equipments} and \emph{double
  categories}, both of which are in fact quite closely related to
\sF-categories.
For instance, from an \sF-category we can construct a double category
whose horizontal arrows are its tight arrows, whose vertical
arrows are its loose arrows, and whose squares are 2-cells
  \[\xymatrix{ \ar@{~>}[d]_h\ar[r]^{f} \drtwocell\omit & \ar@{~>}[d]^h \\
    \ar[r]_{g} & }\]
This double category comes with a \emph{connection} in the sense of~\cite{bm:dbl-thin-conn}
(although there the focus was on edge-symmetric double categories), and in fact \sF-categories are essentially
equivalent to double categories with connections.  The perspective
of \sF-categories, however, has the advantage that we can deploy all
the tools of enriched category theory.  In fact, the framework of double
categories allows a clear explanation of which limits should lift \cite{GrandisPare:LimitsDblCats}, 
as well as 
their universal property with respect to strict maps, but does not seem easily
to capture the universal property with respect to weak maps.

Then again, the {\em proarrow equipments} of \cite{wood:proarrows-i}
can be identified (modulo questions of 2-categorical strictness) with
\sF-categories in which every tight morphism has a loose right
adjoint.  This condition is motivated by the example where tight
arrows are functors and loose arrows are profunctors (also called
``modules'' or ``distributors''), but it is not satisfied in the
examples we are interested in such as $T\bbAlg_l$.  (Some authors,
such as \cite{Verity-thesis}, have also used the term \emph{equipment}
without this extra condition.)  In terms of double categories with
a connection, the condition that tight maps have loose right adjoints
corresponds to also having an ``op-connection;''
see~\cite{shulman:frbi}.

The title of the paper conveys our belief that the introduction of \sF-categories is a contribution of equal importance to the actual characterization of the limits that lift to \sF-categories of weak morphisms.
In fact, there is also a version of our main result using only 2-categories
(see \S\ref{sec:pie}), which in the case $w=p$ recovers the result that PIE-limits
lift to $T\Alg$.  However, while PIE-limits, that is to say
$Q_p$-coalgebras, are plentiful and useful in the 2-categorical
context, there seem to be fairly few $Q_l$- or $Q_c$-coalgebras until
we pass to the \sF-categorical context.  Thus, in the lax case, the
passage to \sF-categories significantly enlarges the class of limits
possessed by $T\Alg_l$ that we can describe.

There are many possible variations on the themes considered here. For example,
one could consider \sF-categories with lax morphisms as the loose maps
and pseudo morphisms as the tight ones. This gives rise to a different notion
of rigged weight. Then again, one could extend the very notion of \sF-category
to allow strict, pseudo, and lax morphisms to be encoded into the structure;
or, more radically, to combine both lax and colax morphisms. As a final example,
one could consider \sF-bicategories, in which composition is only associative 
and unital up to isomorphism. The obvious example is  \bbProf, in
 which the tight morphisms are functors and the loose ones are
profunctors.  This is known to admit many bicolimits 
(see~\cite{street:cauchy-enr}). We hope to address some of these issues
in a future paper.

\paragraph*{Acknowledgements.} Most of the results here were presented
at the conference CT2008 where the first-named author was an invited speaker.
He thanks the organizers for the invitation and for the extremely enjoyable and
rewarding conference. Both authors acknowledge with gratitude
partial support from the Australian Research Council, project DP0771252.
The second-named author gratefully acknowledges support from the United
States National Science Foundation.

\section{2-categorical preliminaries}
\label{sect:2-cats}

We begin with some background material from 2-category theory.
Most of this is standard, but \autoref{lem:w-idempotent} and
the terminological conventions of \S\ref{sect:comonad} do not appear
to be in the literature.

\subsection{Monads in 2-categories}
\label{sec:monads-in-2cat}

A \textbf{monad} in a 2-category \sK\ on an object $A$ is a monoid in
the monoidal category $\sK(A,A)$; thus it consists of a morphism
$t\maps A\to A$ and 2-cells $\mu\maps t t \to t$ and $\eta\maps 1\to
t$ satisfying the usual identities.  We write $\bbDelta_+$ for the
algebraic simplex category (a skeleton of the category of finite
totally ordered sets).  Since $\bbDelta_+$ is the free strict monoidal
category containing a monoid, a monad in \sK\ is equivalently a strict
monoidal functor $\bbDelta_+ \to \sK(A,A)$, or a strict 2-functor
$\cB\bbDelta_+ \to \sK$, where \cB\ indicates that we regard a strict
monoidal category as a 2-category with one object.

An \textbf{object of algebras} or \textbf{Eilenberg-Moore object} for
a monad in \sK\ is an object $A^t$ with a forgetful morphism $u\maps
A^t\to A$ such that for every object $X$, composing with $u$ exhibits
an isomorphism
\[\sK(X,A^t) \iso \sK(X,A)^{\sK(X,t)}\]
between $\sK(X,A^t)$ and the usual Eilenberg-Moore category of the
ordinary monad $\sK(X,t)$ on the ordinary category $\sK(X,A)$ induced
by whiskering with $t$.  Of course, an object $A^t$ with this property
may or may not exist for given \sK\ and $t$.  Making this universal
property explicit, it says that $u$ is a ``$t$-algebra'' in the sense
that we have a 2-cell $\alpha\maps t u \to u$ such that the usual
diagrams for an algebra commute:
\[\vcenter{\xymatrix{ttu \ar[r]^{t\alpha}\ar[d]_{\mu u} & tu \ar[d]^\alpha\\
    tu \ar[r]_\alpha & u}} \qquad\text{and}\qquad
\vcenter{\xymatrix{ u \ar[r]^{\eta u} \ar[dr]_{1_u} & tu \ar[d]^\alpha \\ & u}}
\]
and $u$ is the universal $t$-algebra, in a suitable 2-dimensional
sense.
It was shown in~\cite[I,7.12.4]{gray:formal-ct} that $A^t$ can be
described as a \emph{lax limit} of the diagram $\cB\bbDelta_+ \to \sK$.
The lax \emph{colimit} of this diagram turns out to be the
\textbf{Kleisli object} $A_t$, while if we consider $\bbDelta_+\op$
instead we obtain Eilenberg-Moore and Kleisli objects for
\emph{comonads}.

If $(A,t)$ and $(B,s)$ are monads in \sK, a \textbf{lax monad
  morphism} is a morphism $f\maps A\to B$ together with a 2-cell
$\fbar\maps sf\to ft$ satisfying suitable axioms.
(In~\cite{street:ftm} these were called \emph{monad functors}.)  These
are the lax morphisms of algebras for a suitable 2-monad or 2-comonad,
and also the lax natural transformations between 2-functors
$\cB\bbDelta_+\to\sK$.   A \textbf{monad 2-cell} $\alpha\colon
(f,\fbar) \to (g,\gbar)$ is a 2-cell $\alpha\colon f\to g$ in \sK such
that $s\alpha . \fbar = \gbar . \alpha t$.

We write $\Mnd_l(\sK)$ for the 2-category of monads, lax monad
morphisms, and monad 2-cells.  There is a functor $\sK\to \Mnd_l(\sK)$
assigning to each object its identity monad, and \sK has
Eilenberg-Moore objects if and only if this functor has a right
adjoint.  In particular, any lax monad morphism $(A,t) \to (B,s)$
induces a morphism $A^t \to B^s$ in a functorial way.

Dually, \textbf{colax monad morphisms} (also known as \emph{monad
  opfunctors}) come with a 2-cell $ft\to sf$ and induce morphisms
between Kleisli objects.  There is a 2-category $\Mnd_c(\sK)$ of
monads and colax monad morphisms, and \sK has Kleisli objects if and
only if the inclusion $\sK\to\Mnd_c(\sK)$ has a left adjoint.

A \textbf{distributive law} between monads $t$ and $s$ on the same
object $A$ consists of a 2-cell $st\to ts$ satisfying suitable axioms.
This is equivalent to giving a compatible monad structure on the
composite $ts$, and to giving a lifting of $t$ to the Eilenberg-Moore
object $A^s$, and also to giving an extension of $s$ to the Kleisli
object $A_t$.
It is also equivalent to giving a monad in $\Mnd_l(\sK)$ on the object
$(A,s)$, and to giving a monad in $\Mnd_c(\sK)$ on the object $(A,t)$.
When \sK has Eilenberg-Moore objects, the EM-object-assigning functor
$\Mnd_l(\sK)\to\sK$ takes each distributive law to the above-mentioned
lifting, and similarly for $\Mnd_c(\sK)$ and the extensions to Kleisli
objects.

It follows that there are four different 2-categories whose objects
are distributive laws in \sK.  We will need the following description
of one of them, which is easily verified by writing out the axioms.

\begin{lem}\label{lem:monad-stuff}
  The 2-category $\Mnd_l(\Mnd_c(\sK))$ can be described as follows.
  \begin{itemize}
  \item Its objects are distributive laws $k\colon S R\to R S$ on an
    object \sA in \sK.
  \item Its morphisms from $k_1$ to $k_2$ are morphisms $F\colon \sA_1
    \to \sA_2$ in \sK equipped with 2-cells $\psi\colon S_2 F\to F
    S_1$ making it a lax morphism of monads from $S_1$ to $S_2$ and
    $\chi\colon F R_1\to R_2 F$ making it a colax morphism of monads
    from $R_1$ to $R_2$, and such that the following diagram commutes:
    \begin{equation}
      \xymatrix{
        S_2 F R_1 \ar[d]_{S_2 \chi} \ar[r]^{\psi R_1} & F S_1 R_1 \ar[r]^{F k_1} & F R_1 S_1 \ar[d]^{\chi S_1} \\
        S_2 R_2 F \ar[r]_-{k_2 F} & R_2 S_2 F \ar[r]_{R_2 \psi} & R_2
        F S_1 }\label{eq:monad-stuff}
    \end{equation}
  \item Its 2-cells from $F$ to $G$ are 2-cells $\alpha\colon F\to G$
    in \sK which are both colax monad 2-cells and lax monad 2-cells.
  \end{itemize}
\end{lem}

\begin{cor}\label{lem:monad-stuff-2}
  If \sK has Kleisli objects and $(F,\psi,\chi)$ is a morphism in
  $\Mnd_l(\Mnd_c(\sK))$ as in \autoref{lem:monad-stuff}, then its
  extension $\overline F\colon (\sA_1)_{R_1} \to (\sA_2)_{R_2}$ to
  Kleisli objects is naturally a lax monad morphism from
  $\overline{S_1}$ to $\overline{S_2}$.
\end{cor}
\begin{proof}
  The Kleisli-object-assigning 2-functor $\Mnd_c(\sK)\to\sK$ induces a
  2-functor $\Mnd_l(\Mnd_c(\sK))\to\Mnd_l(\sK)$.
\end{proof}

\subsection{2-monads}
\label{sec:2-monads}

A \textbf{2-monad} is a monad in the 2-category \twocat of
2-categories, 2-functors, and 2-natural transformations.  For a
2-monad $T$ on a 2-category \sK, we write $T\Alg_s = \sK^T$ for its
2-category of algebras.  Explicitly, an object of $T\Alg_s$ is a
(strict) $T$-algebra, consisting of an object $A\in\sK$ and a morphism
$a\maps TA\to A$ such that
\[\vcenter{\xymatrix{T^2A \ar[r]^{Ta}\ar[d]_{\mu A} & TA \ar[d]^a\\
    TA\ar[r]_a & A}}
\qquad\text{and}\qquad
\vcenter{\xymatrix{A \ar[r]^{\eta A} \ar[dr]_{1_A} & TA \ar[d]^a \\ & A }}
\]
commute (strictly).  A morphism in $T\Alg_s$ from $(A,a)$ to $(B,b)$
is called a \textbf{strict $T$-morphism}; it consists of a morphism
$f\maps A\to B$ in \sK\ such that
\[\xymatrix{TA \ar[r]^{Tf}\ar[d]_a & TB \ar[d]^b\\
  A \ar[r]_f & B}\]
commutes (strictly).  Finally, a 2-cell in $T\Alg_s$ from $f$ to $g$
is called a \textbf{$T$-transformation}, and consists of a 2-cell
$\alpha\maps f\to g$ in \sK\ such that
\[\xymatrix{
TA \rtwocell^{Tf}_{Tg}{~T\alpha} & TB\ar[r]^b & B &=& 
TA \ar[r]^{a} & A \rtwocell^{f}_{g}{\alpha} & B.} \]

However, we also have various weaker notions of morphism between
$T$-algebras.  A \textbf{lax $T$-morphism} $(f,\fbar)\maps (A,a)\to
(B,b)$ consists of $f\maps A\to B$ and a 2-cell
\[\xymatrix{TA \ar[r]^{Tf}\ar[d]_a  \drtwocell\omit{\fbar} & TB \ar[d]^b\\
  A \ar[r]_f & B}\]
such that certain diagrams of 2-cells commute \cite{bkp:2dmonads}.  
It is a \textbf{colax
  $T$-morphism} if \fbar\ goes in the other direction, and a
\textbf{pseudo $T$-morphism} if \fbar\ is an isomorphism.  We write
$T\Alg_l$, $T\Alg_c$, and $T\Alg=T\Alg_p$ for the 2-categories of
(strict) $T$-algebras and lax, colax, and pseudo $T$-morphisms,
respectively (each with an appropriate notion of $T$-transformation).

\begin{eg}
When $\sK=[\ob\sD,\Cat]$ for a
small 2-category \sD\ and $T$ is the 2-monad whose algebras are
2-functors $\sD\to\Cat$, then lax, oplax, and pseudo $T$-morphisms
coincide with lax, oplax, and pseudo natural transformations.  
\end{eg}

\begin{rmk}
  We will use the generic word \emph{weak} to refer to pseudo, lax, or colax without prejudice.
  We use the letter $w$ as a decoration or subscript to stand for one of $p$ (pseudo), $l$ (lax), or $c$ (colax).
  Thus, for instance, for any $w$ and any 2-monad $T$, we have a 2-category $T\Alg_w$.
  We write $\wbar$ to denote $w$ with sense reversed, i.e.\ $\bar{p}=p$, $\bar{l}=c$, and $\bar{c}=l$.
\end{rmk}

If $T$ and $S$ are 2-monads on 2-categories \sA and \sB, and
$(F,\psi)\colon (\sA,T) \to (\sB,S)$ is a lax morphism of monads in
\twocat, then as well as a 2-functor $T\Alg_s\to S\Alg_s$, it also 
induces a 2-functor $T\Alg_w \to S\Alg_w$
in a straightforward way.

Moreover, each 2-category $T\Alg_w$, like $T\Alg_s$, comes equipped with a
forgetful 2-functor $U_w\maps T\Alg_w\to \sK$ and a transformation
$TU_w\to U_w$ which again makes $U_w$ into a strict $T$-algebra.  The
difference is that now the transformation $TU_w\to U_w$ is only
pseudo, oplax, or lax natural, respectively as $w=p$, $l$, or $c$
(note the inversion of lax and oplax).

It is shown in~\cite{lack:psmonads} that composing with $U_w$ induces an isomorphism
\begin{equation}
  \Nat_{\wbar}(\sX,T\Alg_w) \iso \Nat_{\wbar}(\sX,T)\Alg_w\label{eq:natalg}
\end{equation}
where $\Nat_{\wbar}(\sX,\sY)$ denotes the 2-category of 2-functors and
$\wbar$-natural transformations between 2-categories \sX and \sY, and
$\Nat_{\wbar}(\sX,T)$ is the 2-monad induced on
$\Nat_{\wbar}(\sX,\sK)$ by composition with $T$.
This should be compared with the universal property
of $T\Alg_s$, which asserts that
\[\mbox{Nat}_s(\sX,T\Alg_s) \iso \mbox{Nat}_s(\sX,T)\Alg_s.\]
where $\mathrm{Nat}_s(\sX,\sY) = [\sX,\sY]$ denotes the 2-category of
2-functors and strict 2-natural transformations.

Since an \sF-categorical version of~\eqref{eq:natalg} will be central
to our characterization theorem, we recall briefly the idea behind it.
Suppose for simplicity that $\sX =\bbtwo$ and $w=l$.
Then an object of $\Oplax(\bbtwo,T\Alg_l)$ is simply
a lax $T$-morphism $(f,\fbar)\colon (A,a) \to (B,b)$.  On the other
hand, an $\Oplax(\bbtwo,T)$-algebra consists of a morphism $f\colon
A\to B$ in \sK (that is, an object of $\Oplax(\bbtwo,\sK)$) together
with an oplax natural transformation from $Tf$ to $f$; this consists
of morphisms $a\colon TA\to A$ and $b\colon TB\to B$ and a 2-cell
$\fbar\colon b.Tf \to f.a$.  The algebra axioms then assert precisely
that $(A,a)$ and $(B,b)$ are $T$-algebras and $(f,\fbar)$ is a lax
$T$-morphism.  This shows the bijection on objects.

Now a morphism in $\Oplax(\bbtwo,T\Alg_l)$ from $(f,\fbar)$ to
$(g,\gbar)\colon (C,c) \to (D,d)$ consists of lax $T$-morphisms
$(h,\hbar)\colon (A,a)\to (C,c)$ and $(k,\kbar)\colon (B,b) \to
(D,d)$, together with a $T$-transformation
$\alpha\colon (k,\kbar)(f,\fbar) \to (g,\gbar)(h,\hbar)$.
On the other hand, a lax morphism of $\Oplax(\bbtwo,T)$-algebras
consists of an oplax transformation from $f\colon A\to B$ to $g\colon
C \to D$, hence morphisms $h\colon A\to C$ and $k\colon B\to D$ and a
2-cell $\alpha\colon k f \to g h$, together with a modification
consisting of 2-cells $\hbar\colon c.Th \to h.a$ and $\kbar\colon d.Tk
\to k.b$.  The requisite axioms then assert precisely that $(h,\hbar)$
and $(k,\kbar)$ are lax $T$-morphisms and $\alpha$ is a
$T$-transformation.

Finally, of particular importance are those 2-monads $T$ such that for
$T$-algebras $(A,a)$ and $(B,b)$, every morphism $f \maps A\to B$ in
\sK\ supports a unique structure of $w$-$T$-morphism.
Following~\cite{kl:property-like}, we call such a 2-monad
\textbf{$w$-idempotent} (also in use is \textbf{(co-)KZ}, since they
were first isolated by Kock and later Z\"oberlein,
cf.~\cite{kock:kzmonads1,zoberlein:kzmonads,kock:kzmonads}).  The following
conditions are known to be equivalent to lax-idempotence of $T$:
\begin{itemize}
\item For every $T$-algebra $(A,a)$, we have $a\dashv \eta_A$ with
  identity counit.
\item For every $A\in\sK$, we have $\mu_A\dashv \eta_{TA}$ with
  identity counit.
\item For every $A\in\sK$, we have $T\eta_{A}\dashv\mu_A$ with
  identity unit.
\end{itemize}
For colax-idempotence, the adjunctions go the other way, and for
pseudo-idempotence, they are adjoint equivalences.  Moreover, if $T$
is $w$-idempotent for some $w$, then:
\begin{itemize}
\item any two $T$-algebra structures on $A\in\sK$ are
  isomorphic via a unique isomorphism of the form $(1_A,\alpha)$, and
\item for any two $w$-$T$-morphisms $(f,\fbar),(g,\gbar)\maps
(A,a)\rightrightarrows (B,b)$, any 2-cell $\alpha\maps f\to g$ in \sK\
is a $T$-transformation.
\end{itemize}
In particular, if $T$ is $w$-idempotent, then the forgetful
functor $T\Alg_w \to \sK$ is 2-fully-faithful (an isomorphism on
hom-categories).

\subsection{2-comonads}
\label{sect:comonad}

In this section we briefly treat the dual case of 2-comonads. A 2-comonad
on a 2-category \sK is of course the same as a 2-monad on $\sK\op$, so formally
there is not much to say.
(As usual, $\sK\op$ denotes reversal of 1-cells, but not 2-cells.)
However, since there has been little discussion of comonads
in the 2-dimensional context, we describe the conventions we adopt, and some
of their ramifications.

Just as for 2-monads, we only consider the strict notion of 2-comonad, consisting
of a 2-functor $W$ equipped with 2-natural transformations $d:W\to W^2$
and $e:W\to 1$ satisfying the usual laws. Once again, we consider only strict
coalgebras, consisting of an object $C$ equipped with a morphism $c:C\to WC$, 
once again satisfying the usual laws.
We need say nothing here about the notion of strict morphism and pseudo 
morphism of coalgebras; what is worth pointing out is the meaning of lax and 
colax. 

Our starting point is the fact that if $T$ is an endo-2-functor on a
2-category \sK, with a right adjoint $T^*$, then to give a 2-monad
structure on $T$ is equivalent to giving a 2-comonad structure on
$T^*$, and furthermore the Eilenberg-Moore 2-categories $T\Alg_s$ and
$T^*\Coalg_s$ agree.  We shall define lax and colax morphisms of
coalgebras in such a way that the isomorphism $T\Alg_s\cong
T^*\Coalg_s$ extends to an isomorphism $T\Alg_w\cong T^*\Coalg_w$.
This way, in concrete cases where the algebras/coalgebras are
understood, we may speak of $w$-morphisms without specifying whether
these are defined using $T$ or $T^*$.

A lax morphism of $T$-algebras $(A,a)\to(B,b)$ involves a morphism
$f:A\to B$ in \sK equipped with a suitable 2-cell
$$\xymatrix{
TA \ar[r]^{Tf} \ar[d]_{a} \drtwocell<\omit>{\overline{f}} & TB \ar[d]^{b} \\
A \ar[r]_{f} & B. }$$
This corresponds, under the adjunction $T\dashv T^*$, to a 2-cell 
$$\xymatrix{
A \ar[r]^{f} \ar[d]_{\tilde{a}} \drtwocell<\omit>{\tilde{f}} & B \ar[d]^{\tilde{b}} \\
T^* A \ar[r]_{T^* f} & T^* B. }$$
Accordingly, for a 2-comonad $W$ and $W$-coalgebras $(C,c)$ and $(D,d)$,
we define a \emph{lax morphism} of $W$-coalgebras to be a morphism $f:C\to D$
equipped with a 2-cell 
$$\xymatrix{
C \ar[r]^{f} \ar[d]_{c} \drtwocell<\omit>{\tilde{f}} & D \ar[d]^{d} \\
WC \ar[r]_{Wf} & WD }
$$
satisfying the usual coherence conditions (in dual form). We write
$W\Coalg_l$ for the 2-category of $W$-coalgebras, lax $W$-morphisms,
and $W$-transformations, and $W\Coalg_c$ and $W\Coalg$ for the evident
variants.

Now a 2-monad $T$ on \sK also induces a 2-comonad $T\op$ on $\sK\op$,
and the diagram for a lax morphism of $T$-algebras, when drawn in
$\sK\op$, becomes the diagram below (drawn with two different 
orientations to make the comparison easier).
$$\xymatrix{
TA \drtwocell<\omit>{\overline{f}} & TB \ar[l]_{Tf} & 
B \ar[r]^{f} \ar[d]_{b} \drtwocell<\omit>{^\tilde{f}} & A \ar[d]^{a}  \\
A \ar[u]^{a} & B \ar[l]^{f} \ar[u]_{b} &  TB \ar[r]_{Tf} & TA }$$
Thus we have a {\em colax} morphism of coalgebras, so
$T\op\Coalg_c=(T\Alg_l)\op$, and more generally
$T\op\Coalg_w=(T\Alg_\wbar)\op$ for any $w$.

Finally, just as a 2-monad $T$ is called \emph{$w$-idempotent} when
the forgetful 2-functor $T\Alg_w\to\sK$ is fully faithful, we say that
a 2-comonad $W$ is \textbf{$w$-idempotent} when the 2-functor
$W\Coalg_w\to\sK$ is fully faithful; i.e.\ when every morphism between
$W$-coalgebras admits a unique structure of $w$-$W$-morphism.  Since
$W\Coalg_w\cong W\op\Alg_\wbar$, this is equivalent to the 2-monad
$W\op$ being $\wbar$-idempotent.  On the other hand, if $W=T^*$ for a
2-monad $T$, then $W$ is $w$-idempotent if and only if $T$ is
$w$-idempotent.

As a case of particular interest, if \sD and \sK are
2-categories with \sK complete and cocomplete (such as \Cat), then the
forgetful 2-functor $[\sD,\sK]\to[\ob\sD,\sK]$ has both adjoints, and
is monadic and comonadic.  If $T$ and $T^*$ are the corresponding
monad and comonad, then lax $T$-morphisms, which as we have just seen
are the same as lax $T^*$-morphisms, can be identified with lax
natural transformations, and similarly in the pseudo and colax cases.

\subsection{Weak morphism classifiers}
\label{sec:weak-morph-class}

If $T$ is a 2-monad on a 2-category \sK and the 2-category
$T\Alg_s$ admits a certain kind of 2-colimit called a
\textbf{$w$-codescent object}, then the (non-full) inclusion
$T\Alg_s\to T\Alg_w$ has a left adjoint $Q_w$, whose value at a
$T$-algebra $(A,a)$ is the $w$-codescent object of the diagram
$$\xymatrix{
T^3A \ar@<2ex>[r]^{mTA} \ar[r]^{TmA}  \ar@<-2ex>[r]^{T^2a} & 
T^2A \ar@<2ex>[r]^{mA} \ar@<-2ex>[r]^{Ta} & TA \ar[l]_{TiA} }
$$
For instance, in the case $w=l$, this means that we have a universal
map $z\colon T A\to Q_l(A,a)$ equipped with a 2-cell $\zeta\colon
z.mA\to z.Ta$ satisfying two compatibility conditions.

The fact that $Q_w$ is left adjoint to the inclusion of $T\Alg_s\to
T\Alg_w$ means that $w$-morphisms $A\lto B$ are in bijection with
strict morphisms $Q_w A\to B$, and likewise for 2-cells between them.
The functor $Q_w$ is called the \textbf{$w$-morphism classifier}; see
\cite{bkp:2dmonads}.  Note that $Q_p$ is traditionally denoted $(-)'$.

Several conditions on $T$ ensuring that $T\Alg_s$ has $w$-codescent
objects are considered in~\cite{lack:codescent-coh}, including:
\begin{itemize}
\item \sK is cocomplete, and $T$ has a rank (that is, its 2-functor part preserves $\alpha$-filtered colimits for some $\alpha$).
\item \sK has, and $T$ preserves, $w$-codescent objects.
\item \sK has, and $T$ preserves coinserters and coequifiers.
\end{itemize}

Dually, if $W$ is a 2-comonad such that $W\Alg_s$ has $w$-descent
objects, then the inclusion $W\Alg_s\to W\Alg_w$ has a right adjoint
$R_w$ called the \textbf{$w$-morphism coclassifier}. Thus
$w$-$W$-morphisms $A\lto B$ are in bijection with strict morphisms
$A\to R_w B$, and likewise for 2-cells.

Finally, if a 2-monad $T$ has a right adjoint $T^*$, which becomes a
2-comonad with the same algebras and morphisms as in
\S\ref{sect:comonad}, then $T\Alg_s = T^*\Alg_s$ has all limits and
colimits that \sK does. Thus, if \sK has $w$-descent and $w$-codescent
objects, then $T\Alg_s\to T\Alg_w$ has both left and right adjoints,
giving natural bijections between weak morphsms $A\lto B$, strict
morphisms $Q_wA\to B$, and strict morphisms $A\to R_wB$, and likewise
for 2-cells.

We write $Q_w$ equally for the left adjoint $T\Alg_w \to T\Alg_s$ and
for the composite $T\Alg_s \hookrightarrow T\Alg_w \xrightarrow{Q_w} T\Alg_s$.
In this latter incarnation, $Q_w$ is a 2-comonad on $T\Alg_s$, since it is a right adjoint followed by its left adjoint.
Moreover, since $T\Alg_s\hookrightarrow T\Alg_w$ is the identity on objects, the adjunction $T\Alg_w
\rightleftarrows T\Alg_s$ can be identified with the
co-Kleisli adjunction for this comonad.  Dually, when $R_w$ exists for
a comonad $W$, it is a
2-monad on $W\Alg_s$ for which $W\Alg_w$ is the Kleisli category.

The components of the counit of the adjunction defining $Q_w$ are
strict $T$-morphisms $q_A\maps Q_w A\to A$, and the components of the
unit are $w$-$T$-morphisms $p_A\maps A\lto Q_w A$.  The triangle
identities for the adjunction say that $q\circ p = 1_A$ and $q\circ
Q_w(p)=1_{Q_w A}$.  Of course, $q$ is also the counit of the comonad
$Q_w$, and $Q_w(p)$ is its comultiplication.

\begin{lem}\label{lem:w-idempotent}
Let $T$ be a 2-monad on a 2-category \sK. Suppose that $T\Alg_s$ admits 
$w$-codescent objects, so that the 2-comonad $Q_w$ on $T\Alg_s$ 
which is the classifier for weak morphisms exists. If moreover \sK
admits $\wbar$-limits of arrows, then $Q_w$ is $w$-idempotent. 
\end{lem}
\begin{proof}
We write this out in the case $w=l$, and we write $Q$ for $Q_l$. The proof
is based on an argument in \cite{bkp:2dmonads} in the pseudo setting.

Let $W = F U$ be the comonad on $T\Alg_s$ generated by the adjunction. Write
$w:W\to 1$ for the counit and $d:W\to W^2$ for the comultiplication. Then
$Q$ is given by the $l$-codescent object of
$$\xymatrix{
W^3 \ar@<2ex>[r]^{wW^2} \ar[r]^{WwW} \ar@<-2ex>[r]^{W^2w} & 
W^2 \ar@<2ex>[r]^{wW} \ar@<-2ex>[r]^{Ww} &
W \ar[l]_{d}  }$$
via a map $z:W\to Q$ and 2-cell $\zeta:z.wW\to z.Ww$. Let $p:U\to UQ$
be the composite 
$$\xymatrix{U \ar[r]^-{iU} & TU=UW \ar[r]^-{Uz} & UQ }.$$
Then $p\colon A\to Q A$ becomes a lax $T$-morphism
$(p,\bar{p}):(A,a)\to Q(A,a)$ where $\bar{p}=\zeta.iTA$.

By assumption, \sK has oplax limits of arrows.
By~\cite[Theorem~3.2]{lack:lim-lax}, therefore, $T\Alg_l$ also has oplax limits of
arrows, and the projections are strict and jointly detect strictness.
(This will also be a special case of our main theorem; see
\S\ref{sec:oll}.)
Let
$$\xymatrix @R1pc {
  & A \ar@{~>}[dd]^{p} \\
  L \ar[ur]^{u} \ar[dr]_{v}  \rtwocell<\omit>{^\lambda}  & {} \\
  & QA }$$
be the oplax limit of $p$ in $T\Alg_l$, so that $u$ and $v$ are
strict and jointly detect strictness.
There is a unique lax morphism $(d,\bar{d}):A\to L$ with
$u(d,\bar{d})=1$, $v(d,\bar{d})=(p,\bar{p})$, and $\lambda d$ the identity.

This map  $(d,\bar{d}):A\to L$ factorizes through $(p,\bar{p}):A\to QA$ via 
a unique strict map $c:QA\to L$. Now $uc$ is strict and $uc(p,\bar{p})=u(d,\bar{d})=1$, so $uc=q:QA\to A$. Similarly $vc$ is strict
and $vc(p,\bar{p})=v(d,\bar{d})=(p,\bar{p})$ and so $vc=1$. It follows that 
$\lambda c:vc\to puc$ is a 2-cell $\eta:1\to pq$ in $T\Alg_l$. 
We shall
show that it is the unit of an adjunction $q\dashv p$ with identity counit;
in other words, that $q\eta$ and $\eta p$ are both identities.

Now $\eta p=\lambda cp=\lambda d$, which is an identity by definition of $d$.
On the other hand $\eta:1\to pq$ and so $q\eta:q\to qpq=q$. Since $q$ is strict, $q\eta$
will be an identity if and only if $q\eta p$ is; but this follows immediately
from the fact that $\eta p$ is an identity.

Thus $q\dashv p$ with identity counit, and so $Qq\dashv Qp$ with
identity counit.  But $Qp:Q\to Q^2$ is the comultiplication of the
comonad and $q$ is its counit, so $Q$ is lax-idempotent.
\end{proof}

In particular, for the case of pseudo morphisms, we have $pq\iso 1_{Q_p A}$, so that
$p$ and $q$ are inverse adjoint equivalences in $T\Alg_p$ \cite{bkp:2dmonads}.  
In this case, $q\maps Q_pA\to A$ is a \emph{cofibrant replacement} in a
suitable Quillen model structure on $T\Alg_s$ whose homotopy
2-category is $T\Alg_p$; see~\cite{lack:htpy-2monads}.  The cofibrant
objects, traditionally called \textbf{flexible} algebras, are those
for which there exists a \emph{strict} $T$-morphism $s\maps A\to Q_p
A$ with $qs=1_A$.  In this case we also have $sq\iso 1_{Q_p A}$, so
that $q$ is an equivalence in $T\Alg_s$ as well.  Of course, any
coalgebra for the comonad $Q_p$ is flexible, but not every flexible
object is a $Q_p$-coalgebra.  Note, though, that the flexible objects
are precisely the retracts of $Q_p$-coalgebras;
see~\cite{gt:nwfs,garner:soa} for a general theory of such ``algebraic
cofibrancy.''

\section{Enriched category theory over \sF}
\label{sec:f-categories}

\subsection{The cartesian closed category \sF}
\label{sec:ccc-f}

Let $\Cat^\bbtwo$ be the category of arrows in \Cat; we denote by \sF\
its full subcategory determined by the functors which are injective on objects
and fully faithful. We sometimes call such functors {\em full embeddings}.
Thus an object of \sF\ is a full embedding
\[\xymatrix{A_\tight \ar@{^(->}[r]^j &  A_\loose}\]
and a morphism in \sF\ is a commutative square
\[\xymatrix{A_\tight \ar@{^(->}[r]^{j_A}\ar[d]_{f_\tight} & A_\loose \ar[d]^{f_\loose} \\
  B_\tight \ar@{^(->}[r]_{j_B} & B_\loose.}\]
Since $j_B$ is monic in \Cat, in such a commutative square $f_\tight$ is
determined uniquely, if it exists, by $f_\loose$. 
We speak of $A_\tight$ as the \textbf{tight} part of $A$ and
$A_\loose$ as the \textbf{loose} part, and similarly for $f_\tight$
and $f_\loose$.

Note that \sF is  naturally  a 2-category: its 2-cells
$\alpha\colon f\to g$, as inherited from $\Cat^\bbtwo$, are commuting
diagrams of 2-cells of the form
\[\xymatrix{A_\tight
  \ar@{^(->}[r]^{j_A}\dtwocell_{g_\tight}^{f_\tight}{\alpha_\tight} &
  A_\loose \dtwocell^{f_\loose}_{g_\loose}{\alpha_\loose} \\
  B_\tight \ar@{^(->}[r]_{j_B} & B_\loose.}\]
Since $j_B$ is fully faithful, such a 2-cell $\alpha:f\to g$ is
determined by $\alpha_\loose\colon f_\loose \to g_\loose$.

Now, since full embeddings are the right class of a factorization
system on \Cat\ (the left class consists of functors that are
surjective on objects), \sF\ is reflective in $\Cat^\bbtwo$ and
therefore complete and cocomplete, with limits formed pointwise.
Colimits in \sF are formed by taking the colimit in $\Cat^\bbtwo$,
then applying the reflection, which amounts to taking the full embedding
part of the (surjective on objects, full embedding)
factorization of a functor. 

Moreover, \sF\ is cartesian closed. This can be seen as an instance of
the Day reflection theorem \cite{Day-reflection}, or can be checked
directly. To see explicitly what the internal hom $[B,C]$ in \sF must
be, it is convenient to introduce two special objects of \sF. We
denote by $1_\tight$ the terminal object $1\to 1$ of \sF, and we
denote by $1_\loose$ the object $0\to 1$, where $0$ is the empty
category and $1$ the terminal category. Note that $1_\tight$ and
$1_\loose$ together generate \sF as a 2-category; in fact, they are
the representables in $\Cat^\bbtwo$, seen as objects of \sF.
 Moreover, for any
$A\in \sF$ we have
\begin{align*}
  A_\tight &\cong \sF(1_\tight,A) \qquad\text{and}\\
  A_\loose &\cong \sF(1_\loose,A)
\end{align*}
where $\sF(-,-)$ denotes the \Cat-valued hom of the 2-category \sF.
In particular, this tells us that we must have
\[ [B,C]_\tight \cong \sF(1_\tight,[B,C]) \cong
\sF(1_\tight \times B,C) \cong \sF(B,C)
\]
and
\[ [B,C]_\loose \cong \sF(1_\loose,[B,C]) \cong
\sF(1_\loose \times B,C) \cong \Cat(B_\loose,C_\loose).
\]
That is, a tight object of $[B,C]$ is simply a morphism $B\to C$ in
\sF, while a loose object of $[B,C]$ is a functor $B_\loose \to
C_\loose$, and the morphisms in either case are natural
transformations $\xymatrix{B_\loose \rtwocell & C_\loose}$. (As with
the 2-cells in \sF, in the tight case this uniquely determines a
compatible transformation between tight parts.) Comparing
$[B,C]_\tight$ with $[B,C]_\loose$, we can say informally that a
morphism in $[B,C]$ is tight just when it ``preserves tightness,''
in the sense that  it takes tight objects of $B$ to tight objects of $C$.

We can equivalently construct the full embedding
$[B,C]_\tight \hookrightarrow [B,C]_\loose$ using the following pullback:
$$\xymatrix@C=3pc{
[B,C]_\tight \ar@{^(->}[r]^{j_{[B,C]}} \ar[dd] & [B,C]_\loose \ar@{=}[d] \\
& [B_\loose,C_\loose] \ar[d]^{[j_B,C_\loose]} \\
[B_\tight,C_\tight] \ar@{^(->}[r]_{[B_\tight,j_\loose]} & [B_\tight,C_\loose] 
}$$

In practice, many full subcategories are {\em replete}, in the sense that
any object isomorphic to one in the subcategory is itself in the subcategory. 
A non-replete full subcategory is equivalent to its repletion as a
category, but \emph{not} as an object of \sF.

There is also a larger sub-2-category \sFF\ of $\Cat^\bbtwo$
containing all the fully faithful functors, not necessarily injective
on objects. As mentioned in the introduction, all our results have
straightforward extensions to \sFF-categories, but we shall not
mention them.

\subsection{\sF-categories}
\label{sec:fcats}

Since \sF is cartesian closed, we can now consider the notion
of \textbf{\sF-category}, or category enriched in \sF. Of course, an
\sF-category \bbA has a collection of objects, together with hom-objects
$\bbA(x,y)$ in \sF and composition and identity maps also in \sF. Each
hom-object $\bbA(x,y)$ thus consists of two categories
$\bbA(x,y)_\tight$ and $\bbA(x,y)_\loose$ related by a full embedding.
It is easy to see that the categories $\bbA(x,y)_\tight$
must form the hom-categories of a 2-category $\sA_\tight$.
Likewise, the categories $\bbA(x,y)_\loose$ must form a 2-category $\sA_\loose$ with the same objects, and the full embeddings relating them must fit together into a
2-functor $J_{\bbA}:\sA_\tight\to\sA_\loose$ which is the {\em
  identity on objects}, {\em faithful}, and {\em locally fully faithful}. Furthermore,
any such 2-functor determines a unique \sF-category, so we will generally identify \sF-categories with such 2-functors.

We refer to morphisms in $\sA_\tight$ as \textbf{tight morphisms} and those in $\sA_\loose$ as \textbf{loose morphisms}.
We generally write tight morphisms with straight arrows $A\to B$ and loose ones with wavy arrows $A\lto B$.
We will also omit the subscript on $J$ when it is evident from context, and since it is the identity on objects, we will not notate its application to objects.
However, we will usually notate $J$ when applied to morphisms, or when composed with other 2-functors.

\begin{rmk}
In terms of the generating objects $1_\tight$ and $1_\loose$
introduced in \S\ref{sec:ccc-f}, we have two monoidal functors
$\sF(1_\tight,-), \sF(1_\loose,-)\colon \sF\to \Cat$ related by a
monoidal transformation arising from the inclusion $1_\loose
\hookrightarrow 1_\tight$. The 2-categories $\sA_\tight$ and
$\sA_\loose$ and the 2-functor $J_{\bbA}$ are then the ``change of
base'' of the \sF-category \bbA\ along these functors and
transformation.
\end{rmk}

When we write our \sF-categories as 2-functors in this way, an
\textbf{\sF-functor} $F:\bbA\to\bbB$ consists of 2-functors
$F_\tight:\sA_\tight\to\sB_\tight$ and
$F_\loose:\sA_\loose\to\sB_\loose$ making the evident square commute.
Since $J_{\bbB}$ is monic in \twocat, $F_\tight$ is determined
uniquely, if it exists, by $F_\loose$; thus we can say informally that an
\sF-functor $\bbA\to\bbB$ is a 2-functor $\sA_\loose \to \sB_\loose$
which ``preserves tightness.''

Likewise, an \textbf{\sF-natural transformation} $m:F\to G$ reduces to
a pair of 2-natural transformations $m_\tight:F_\tight\to G_\tight$
and $m_\loose:F_\loose\to G_\loose$ subject to the evident condition.
Since $J_\bbA$ is the identity on objects, the components of $m_\loose$
are determined by those of $m_\tight$, so its existence is a mere
additional property imposed on $m_\tight$  (``naturality with
respect to loose maps, in addition to tight ones'').  On the other
hand, since $J_\bbB$ is faithful, we can equally regard the existence
of $m_\tight$ as a property of $m_\loose$, namely that all of its
components are tight.

\begin{eg}
Any 2-category \sK may be regarded as an \sF-category \bbK in which
$\sK_\tight=\sK_\loose=\sK$, so that ``all morphisms are tight''. We
call such an \sF-category {\bf chordate}. On the other hand, we may
instead take $\sK_\loose=\sK$ but let $\sK_\tight$ be the locally full
sub-2-category of $\sK$ containing all the objects but only the
identity morphisms (``only identities are tight''). We call such an
\sF-category {\bf inchordate}.  Note that in the inchordate case $\sK_\tight$ is
\emph{not} generally a discrete 2-category: it contains only the identity
1-morphisms of \sK, but all the endo-2-cells of these.

For an abstract point of view, recall that $\sF(1_\loose,-)\colon \sF\to\Cat$ induces a 2-functor $\sF\Cat\to\twocat$ which sends an \sF-category to its loose part.
This 2-functor has a left adjoint which sends a 2-category to the corresponding
inchordate \sF-category, and a right
adjoint which sends a 2-category to the corresponding chordate
\sF-category. The latter 2-functor can also
be induced directly by the finite-product-preserving functor
$\Cat\to\sF$ which sends a category $C$ to its identity functor.
\end{eg}

\begin{eg}\label{eg:fcat-alg}
In our motivating examples of \sF-categories, the objects are some
sort of category with structure, the tight morphisms are the functors
which preserve the structure strictly, the loose morphisms are the
functors which preserve the structure in some weaker sense, and the 2-cells are 
suitably compatible with the extra structure.
For example, for any 2-monad $T$ on a 2-category \sK and any
$w=p,c,l$, the inclusion $J\colon T\Alg_s\to T\Alg_w$ is a
prototypical \sF-category, which we denote $T\bbAlg_w$.
It comes with a forgetful \sF-functor $U_w\colon T\bbAlg_w \to \bbK$,
where $\bbK$ is the chordate \sF-category associated to \sK.
\end{eg}

\begin{eg}
Another important class of \sF-categories, less relevant in this
paper, is where $\sKl$ is some 2-category of interest, and $\sKt$ the
sub-2-category of left adjoints.
\end{eg}

\begin{eg}
Our last, and very important, example of an \sF-category comes from
the general fact that any monoidal closed category is enriched over
itself. We shall write \bbF for the \sF-category which arises from \sF
in this way. The hom-objects of \bbF are, of course, given by the
cartesian closed internal hom of \sF as described in
\S\ref{sec:ccc-f}. Thus, the objects of \bbF are the objects of \sF,
its tight morphisms are the morphisms of \sF, its loose morphisms are
functors between loose parts, and its 2-cells are transformations
between the latter. In particular, the 2-category $\sF_\tight$ is just
$\sF$ with its 2-category structure as mentioned previously, while
$\sF_\loose$ can be obtained as the fully-faithful reflection of the
composite $\cod \circ N$, as in the following diagram:
\begin{equation}
\xymatrix{
\sF_\tight \ar[r]^-{N} \ar@{->|}[d]_{J_\bbF} & \Cat^\bbtwo \ar[d]^{\cod} \\
\sF_\loose \ar@{^{(}->}[r]_-{M} & \Cat. }\label{eq:fcatf}
\end{equation}
Here $N$ is the inclusion, $J_\bbF$ is the identity on objects, and
$M$ is 2-fully-faithful (i.e.\ an isomorphism on hom-categories).
We shall sometimes, as in this diagram, display fully faithful maps using
a hooked arrow, and bijective-on-objects ones using a bar at the tip of the arrow. 
Since $\cod\circ N$ is locally fully faithful, so is $J_\bbF$.
($M$ is actually an equivalence of 2-categories, but it is important to maintain the distinction between $\sF_\loose$ and $\Cat$, since $J_\bbF$ is the identity on objects but $M\circ J_\bbF$ is not.)
\end{eg}

Since \sF is a complete and cocomplete symmetric monoidal closed
category, we have all of the basic machinery of enriched category
theory at our disposal. In
\S\S\ref{sec:functor-fcats}--\ref{sec:eg-wgts} we will discuss
enriched functor categories, limits, and colimits in the particular
case of enrichment over \sF.

\subsection{Functor \sF-categories}
\label{sec:functor-fcats}

Given \sF-categories \bbD\ and \bbK\ with \bbD\ small, we can form the functor
\sF-category $[\bbD,\bbK]$ whose objects are \sF-functors from \bbD\ to \bbK.
A morphism in $[\bbD,\bbK]_\tight$ is just an \sF-natural transformation, while
a morphism in $[\bbD,\bbK]_\loose$ between $F,G\colon \bbD\to\bbK$ consists of
a 2-natural transformation $F_\loose\to G_\loose$.
We also have 2-cells in
$[\bbD,\bbK]_\loose$, which are modifications between the 2-natural transformations $F_\loose\to G_\loose$ just considered.
In \S\ref{sec:wf-transf}, we shall  consider weakenings of these notions, where the morphisms are not
required to be strictly natural.

We now turn to the special case where $\bbK=\bbF$. An \sF-functor 
$G\colon\bbD\to\bbF$ is often called a \emph{weight}; it amounts to a commutative
square of 2-functors as in the left half of the following diagram (the right half simply reproduces~\eqref{eq:fcatf}):
$$\xymatrix{
\sD_\tight \ar[d]_{J_\bbD} \ar[r]^{G_\tight} & \sF_\tight \ar[d]_{J_\bbF} \ar[r]^-{N} &
\Cat^\bbtwo \ar[d]^{\cod} \\
\sD_\loose \ar[r]_{G_\loose} & \sF_\loose \ar[r]_-{M} & \Cat 
}$$
Since $M$ is 2-fully-faithful and $J_\bbD$ is the identity on objects,
$G_\loose$ is uniquely determined by $G_\tight$ and the composite $MG_\loose$. 
On the other hand, $G_\tight$ is uniquely determined by the composites 
$\dom N G_\tight$, $\cod N G_\loose$, and a 2-natural
transformation $\dom N G_\tight\to \cod N G_\tight$ whose components are 
full embeddings; we then also need $\cod N G_\tight=MG_\loose J_\bbD$.
Thus, altogether, to give a weight is to give 2-functors 
$\Phi_\tight:\sD_\tight\to\Cat$ and $\Phi_\loose:\sD_\loose\to\Cat$, and a
2-natural transformation $\phi:\Phi_\tight\to\Phi_\loose J_\bbD$ whose components
are full embeddings; we write
$\Phi$ for such a weight $(\Phi_\tight,\Phi_\loose,\phi)$.

Suppose now that $\Psi=(\Psi_\tight,\Psi_\loose,\psi)$ is another such weight.
We compute the \sF-valued hom $[\bbD,\bbF](\Phi,\Psi)$, using our description 
of functor \sF-categories. An object of $[\bbD,\bbF](\Phi,\Psi)_\loose$
(that is, a loose morphism in $[\bbD,\bbF]$) is just
a 2-natural transformation between the corresponding 2-functors 
$\sD_\loose\to\sF_\loose$, but since $M$ is 2-fully-faithful, that is the same
as a 2-natural transformation $\Phi_\loose\to\Psi_\loose$. A morphism in
$[\bbD,\bbF](\Phi,\Psi)_\loose$ (that is, a 2-cell in 
$[\bbD,\bbF]$) is a modification between 2-naturals
$\Phi_\loose\to\Psi_\loose$. In other words, 
$$[\bbD,\bbF](\Phi,\Psi)_\loose = [\sD_\loose,\Cat](\Phi_\loose,\Psi_\loose).$$
A morphism $\Phi\to\Psi$ is tight when, seen as a 2-natural 
$\Phi_\loose\to\Psi_\loose$, it restricts to a
2-natural transformation $\Phi_\tight\to\Psi_\tight$; in other words, we have a 
pullback 
\[
\xymatrix{
  [\bbD,\bbF](\Phi,\Psi)_\tight \ar[dd] \ar[r] &
  [\bbD,\bbF](\Phi,\Psi)_\loose \ar@{=}[r] &
  [\sD_\loose,\Cat](\Phi_\loose,\Psi_\loose) \ar[d]^{[J_\bbD,\Cat]}\\
  &&
  [\sD_\tight,\Cat](\Phi_\loose J_\bbD,\Psi_\loose J_\bbD) \ar[d]^{[\sD_\tight,\Cat](\phi,\Psi_\loose J_\bbD)} \\
  [\sD_\tight,\Cat](\Phi_\tight,\Psi_\tight) \ar[rr]_{[\sD_\tight,\Cat](\Phi_\tight,\psi)} 
  & &
  [\sD_\tight,\Cat](\Phi_\tight,\Psi_\loose J_\bbD) }
\]


A particular important class of weights are the representables:
for any \sF-category \bbK and any object $X$ of \bbK, we have a representable
\sF-functor $R=\bbK(X,-):\bbK\to\bbF$. In terms of the previous paragraph, we 
have $\bbK(X,-)_\tight=\sK_\tight(X,-)$ and $\bbK(X,-)_\loose=\sK_\loose(X,-)$, while 
$r\colon R_\tight\to R_\loose J_\bbK$ is the map
\[J_{\bbK}\colon \sK_\tight(X,-)\to\sK_\loose(J_\bbK X,J_\bbK-) = \sK_\loose(X,J_\bbK-)
\]
given by the action of $J_\bbK$ on hom-categories.

\subsection{\sF-weighted limits}
\label{sec:limits}

Suppose now $S:\bbD\to\bbA$ is an \sF-functor, of which we shall shortly 
consider the $\Phi$-weighted limit, and let $A$ be an object of \bbA. These 
induce an \sF-functor $\bbA(A,S):\bbD\to\bbF$.  Writing this functor in the form
$\Psi=(\Psi_\tight,\Psi_\loose,\psi)$ of \S\ref{sec:functor-fcats}, we have
$\Psi_\tight=\sA_\tight(A,S_\tight)$, $\Psi_\loose=\sA_\loose(A,S_\loose)$, and $\psi$ given by 
$$\xymatrix @R1pc {
\sD_\tight \ar[r]^{S_\tight} \ar[dd]_{J_\bbD} & \sA_\tight \ar[dd]_{J_\bbA} \ar[dr]^{\sA_\tight(A,-)}\\
&  {}\rtwocell<\omit>{\hspace{-1cm}J_\bbA} & \Cat \\
\sD_\loose \ar[r]_{S_\loose} & \sA_\loose \ar[ur]_{\sA_\loose(A,-)}
}$$
We are now ready to consider the weighted limit $\{\Phi,S\}$, for some weight $\Phi\colon \bbA\to\bbF$.
This limit, if it exists, is characterized by an isomorphism
$$\bbA(A,\{\Phi,S\}) \cong [\bbD,\bbF](\Phi,\bbA(A,S))$$
in \sF, natural in $A$. This involves an isomorphism
$$\sA_\loose(A,\{\Phi,S\})\cong [\sD_\loose,\Cat](\Phi_\loose,\sA_\loose(A,S_\loose))$$
and a left leg for the square
\begin{equation}
\xymatrix{
\sA_\tight(A,\{\Phi,S\}) \ar[r] \ar@{.>}[ddd] & \sA_\loose(A,\{\Phi,S\}) \ar[d]^{\cong} \\
& [\sD_\loose,\Cat](\Phi_\loose,\sA_\loose(A,S_\loose)) \ar[d]^{[J_{\bbD},\Cat]} \\
& [\sD_\tight,\Cat](\Phi_\loose J_\bbD,\sA_\loose(A,S_\loose J_\bbD)) \ar[d]^{[\sD_\loose,\Cat](\phi,1)} \\
[\sD_\tight,\Cat](\Phi_\tight,\sA_\tight(A,S_\tight)) \ar[r] & [\sD_\tight,\Cat](\Phi_\tight,\sA_\loose(A,J_\bbA S_\tight))
}\label{eq:flim-pb}
\end{equation}
so that it becomes a pullback.
The first isomorphism says that $\{\Phi,S\}$ is the 2-categorical
limit $\{\Phi_\loose,S_\loose\}$ in \sAl; we shall write $p_\loose:\Phi_\loose\to\sAl(\{\Phi,S\},S_\loose)$ for the corresponding unit. The pullback square~\eqref{eq:flim-pb} 
specifies a further universal property involving the tight maps, which we now
analyze.

First of all, to give a dotted map making the square commute is to give a map 
$p_\tight:\Phi_\tight\to\sA_\tight(\{\Phi,S\},S_\tight)$ making 
$$\xymatrix{
\Phi_\tight \ar[r]^{\phi} \ar[d]_{p_\tight} & \Phi_\loose J \ar[d]^{p_\loose J} \\
\sA_\tight(\{\Phi,S\},S_\tight) \ar[r]_{J_\bbA} & \sA_\loose(\{\Phi,S\},S_\loose) }$$
commute. Since the bottom leg of this latter square is injective on objects
and fully faithful, such a $p_\tight$ is unique if it exists, and will exist if 
and only if for each $D\in\bbD$ the composite 
$$\xymatrix{
\Phi_\tight D \ar[r]^{\phi D} & \Phi_\loose D \ar[r]^-{p_\loose D} &
\sA_\loose(\{\Phi,S\},SD) }$$
takes its values in $\sA_\tight(\{\Phi,S\},SD)$. In other words, for any 
$D\in\bbD$ and any $a\in\Phi_\tight D$, the morphism 
$p_{\lambda,\phi(a)}:\{\Phi,S\}\to SD$ is tight.

Second, we require that the resulting square~\eqref{eq:flim-pb} be a pullback.
Since the horizontals are fully faithful, we need only check the universal property
at the level of objects.  This says that a loose morphism $h:A\to\{\Phi,S\}$ is tight if  
the composite 
$$\xymatrix{
\Phi_\tight D \ar[r]^{\phi D} & \Phi_\loose D \ar[r]^-{p_\loose D} & 
\sA_\loose(\{\Phi,S\},SD) \ar[rr]^-{\sA_\loose(h,SD)} & &
\sA_\loose(A,SD) }$$
takes its values in $\sA_\tight(A,SD)$ for each $D\in\bbD$. In other words, $h$
is tight if $p_{\loose,\phi(a)}\circ h$ is tight for each $a\in\Phi_\tight D$. 
We express this by saying
that the $p_{\loose,\phi(a)}$ \textbf{jointly detect tightness}.

Combining the two conditions, we have:

\begin{prop}
 Let $\Phi:\bbD\to\bbF$ be a weight and $S:\bbD\to\bbK$ an \sF-functor. 
 An \sF-categorical limit $\{\Phi,S\}$ in \bbK is a 2-categorical limit
$\{\Phi_\loose,S_\loose\}$ in \sKl with the extra property that for  $D\in\bbD$
and $a\in\Phi_\tight(D)$, the projections $p_{\lambda,\phi(a)}:\{\Phi_\loose,S_\loose\}\to SD$
are tight and jointly detect tightness.
\end{prop}

In particular, if \bbK is chordate, then an \sF-limit
$\{\Phi,S\}$ in \bbK is nothing but a 2-categorical limit
$\{\Phi_\loose,S_\loose\}$ in the 2-category $\sKt=\sKl$.

Finally, we consider two slightly different notions of ``limits
lifting along a functor,'' one well-known and one less so.

\begin{defn}\label{def:lifting}
  Let \sV be a closed symmetric monoidal category, $U\maps \sA\to\sB$
  a \sV-functor, and $\Phi\maps \sD\to\sV$ a \sV-weight.  If for any
  diagram $G\maps \sD\to\sA$ such that the limit $\{\Phi, U G\}$
  exists, the limit $\{\Phi, G\}$ also exists and is preserved by $U$,
  we say that $\Phi$-weighted limits \textbf{lift along} $U$ or
  \textbf{lift to} \sA, or that $U$ \textbf{lifts} $\Phi$-weighted
  limits.  If $U$ furthermore reflects all such $\Phi$-weighted
  limits, we say that it \textbf{creates} them.
\end{defn}

The notion of ``creation'' of limits is standard, although there is
some variation in its usage.
Some authors use it only when \sB has all $\Phi$-weighted limits,
which we do not generally require.  Others require that $\{\Phi, G\}$
must map exactly onto $\{\Phi, U G\}$ and be literally unique with
that property; this will be true for the \sF-functors $U:\sA\to\sB$ we
consider in this context, but we shall neither use nor verify it.

Note that if $\Phi$-weighted limits lift along $U$, then any
$\Phi$-weighted cone over $G$ in \sA which maps to a limiting cone in
\sB must factor through the limit $\{\Phi, G\}$ in \sA by a map which
is inverted by $U$.  Thus, if $U$ is conservative, reflection is
automatic, and so lifting and creating are equivalent.

The \sF-functor $U_w\colon T\bbAlg_w \to \bbK$ is always conservative,
since the underlying ordinary category of $T\bbAlg_w$ consists of
\emph{strict} $T$-morphisms.  Thus, in this case there is no
difference between lifting and creation.  By contrast, the 2-functor
$U_w\colon T\Alg_w \to \sK$ is not in general conservative for $w=l$
or $c$, and in this case the limits which lift are \emph{not}
generally reflected.  (Thus the statements in~\cite{lack:lim-lax}
referring to ``creation'' of limits are only about ``lifting'' of
limits according to our present terminology.)  We regard this as
another advantage of \sF-categories over 2-categories.

\subsection{Examples of weights}
\label{sec:eg-wgts}

In this section we consider a few specific examples of weights, and
describe the corresponding notions of limit. For now, we focus on
examples which lift to $T\bbAlg_w$ for some $w$, and are thus of
interest in our primary examples. We will mention some more
``pathological'' examples in \S\ref{sec:tightly-rigged}. We shall describe
in each case what is known about lifting the limit to $T\bbAlg_w$ for a 2-monad $T$,
as in Example~\ref{eg:fcat-alg};
we shall see in Proposition~\ref{prop:reduction} that this implies a lifting result for any \sF-monad $T$.

\subsubsection{Tight limits}
\label{sec:tight-limits}

Let \sD be a 2-category, and \bbD the corresponding chordate \sF-category,
with $\sDt=\sDl=\sD$. A 2-functor $M\colon\sDt\to\Cat$ gives rise to an 
\sF-weight
$\Phi:\bbD\to\bbF$ with $\Phi_\tight=\Phi_\loose=M$, and $\phi$ the identity. Then
for any weight $\Psi:\bbD\to\bbF$, the \bbF-valued hom $[\bbD,\bbF](\Phi,\Psi)$ 
is given by the full embedding
$$\xymatrix{
[\sD,\Cat](M,\Psi_\tight) \ar[rr]^{[\sD,\Cat](M,\psi)} && [\sD,\Cat](M,\Psi_\loose). }$$

A diagram $S\colon \bbD\to\bbA$ of shape $\bbD$ is just a 2-functor
$S:\sD\to\sAt$; then the corresponding loose part $\sDl=\sD\to\sAl$ is the composite
$J_\bbA S:\sDt\to\sAl$. Thus the tight part of the universal property of the limit 
$\{\Phi,S\}$ says that $\sAt(A,\{\Phi,S\})\cong[\sD,\Cat](M,\sAt(A,S))$, and so 
$\{\Phi,S\}$ is the 2-categorical limit $\{M,S\}$ in \sAt; while the loose part says
that this limit is preserved by $J_\bbA:\sAt\to\sAl$. 
We call a limit of this type {\bf tight}. 

In the case where \bbA is $T\bbAlg_w$ for a 2-monad $T$, such tight limits
amount to limits in $T\Alg_s$ preserved by the inclusion $T\Alg_s\to T\Alg_w$. 
These tight limits do lift to $T\bbAlg_w$ for any 2-monad $T$ and any $w$;
see~\cite{bkp:2dmonads} and~\cite[Prop.~4.1]{lack:lim-lax}.

\subsubsection{The oplax limit of a loose morphism}
\label{sec:oll}

Let $\sD_\loose$ be the arrow category $\bbtwo=\{d\to c\}$, seen as a
locally discrete 2-category, let $\sD_\tight$ be the discrete 2-category 
with two objects, and let $J_\bbD:\sD_\tight\to\sD_\loose$ be the evident inclusion. Then a diagram $S:\bbD\to\bbA$ is precisely a loose morphism 
in \bbA; we shall write it as $s:Sd\to Sc$.

Let $\Phi_\loose\colon \sD_\loose\to\Cat$ be the 2-functor which picks out the functor $c\colon 1\to\bbtwo$,
let $\Phi_\tight\colon \sD_\tight\to\Cat$ be the 2-functor constant at the 
terminal category $1$, and let $\phi\colon \Phi_\tight\to\Phi_\loose J_\bbD$
have components $1\colon 1\to 1$ and $c\colon 1\to\bbtwo$. 

As always, a limit $\{\Phi,S\}$ must in particular be a limit $\{\Phi_\loose,S_\loose\}$, which means an oplax limit in $\sA_\loose$
of the arrow $s:Sd\to Sc$.  This consists of an object $L$ with loose 
morphisms $u:L\lto Sd$ and $v:L\lto Sc$, and a 2-cell $\sigma:v\to su$,
these data being universal in $\sA_\loose$. The tight aspect of the universal
property now says that we have a pullback 
$$\xymatrix{
  \sA_\tight(A,L) \ar[r] \ar@{.>}[d] &
  \sA_\loose(A,L) \ar[d] \\
  \sA_\tight(A,Sd)\times\sA_\tight(A,Sc) \ar[r] & 
  \sA_\loose(A,Sd)\times\sA_\loose(A,Sc) }$$
in \Cat.  In other words, $u$ and $v$ are tight and jointly detect tightness.

In particular, this means that we have a bijection between loose morphisms
$x:A\lto L$ and pairs of loose morphisms $u:A\lto Sd$, $v:A\lto Sc$
equipped with a 2-cell $\xi:v\to su$; and similarly a bijection between
tight morphisms $x:A\to L$ and pairs of tight morphisms $u:A\to Sd$,
$v:A\to Sc$ equipped with a 2-cell $\xi:J(v)\to s\circ J(u)$.

By~\cite[Theorem~3.2]{lack:lim-lax}, oplax limits of loose morphisms
lift to $T\bbAlg_l$ for any 2-monad $T$. Dually, lax limits of loose
morphisms  (where the 2-cell $\sigma$ is reversed)  lift to $T\bbAlg_c$.
By~\cite[Remark~2.7]{bkp:2dmonads}, lax, oplax, and also pseudo limits
(the case where $\sigma$ is invertible) of loose morphisms all lift to 
$T\bbAlg_p$ for any 2-monad $T$.

\subsubsection{Inserters}
\label{sec:inserters}

Let \bbD be the inchordate \sF-category on a parallel pair of arrows
$(a\rightrightarrows b)$; thus a functor $\bbD\to\bbK$ is a parallel  pair of loose
morphisms in \bbK.  Let $\Phi_\loose\colon \sD_\loose\to\Cat$ pick out
the two distinct functors $1\rightrightarrows \bbtwo$, and let 
 $\Phi_\tight(a)=1$
and $\Phi_\tight(b)=0$.

A $\Phi$-weighted limit of $f,g\colon A \lto B$ is in particular an
inserter, i.e.\ a morphism $i\colon I\to A$ with a 2-cell $f i \lto g
i$ which is universal among such.  Moreover, the morphism $i$ (but not
the composites $f i$ and $g i$) must be tight and must detect
tightness.   We call such a limit a \textbf{$p$-rigged inserter};
by~\cite[Prop.~2.2]{bkp:2dmonads}, such inserters
lift to $T\bbAlg_p$ for any 2-monad $T$.

Now instead suppose we take $\sD_\tight$ to be the arrow category
\bbtwo, equipped with one of its inclusions into the parallel pair
$\sD_\loose$.  We let $\Phi_\loose$ be as before, with the tight
morphism in \bbD going to the functor $d \colon 1\to \bbtwo$, and we
let $\Phi_\tight$ be constant at $1$.

In this case, a \bbD-diagram in \bbK is a parallel pair $f,g\colon
A\to B$ where $f$ is tight and $g$ is loose, and a $\Phi$-weighted
limit of such is an inserter $f i \to g i$ such that $i$ 
is tight and detects tightness; thus $fi$ is also tight. 
We call such a limit an \textbf{$l$-rigged inserter};
by~\cite[Prop.~4.4]{lack:lim-lax}, such inserters lift to
$T\bbAlg_l$ for any 2-monad $T$.

If instead we require $g$ to be tight,  we obtain the notion of
\textbf{$c$-rigged inserter}, which lifts to $T\bbAlg_c$.

\subsubsection{Equifiers}
\label{sec:equifiers}

Let \bbD be the inchordate \sF-category on a parallel pair of 2-cells
(between a parallel pair of morphisms $a\rightrightarrows b$), so that a \bbD-shaped diagram
is a parallel pair of 2-cells between a parallel pair of loose
morphisms.  Let $\Phi_\loose$ be the diagram
 \[ \xymatrix{ 1 \rtwocell{\Downarrow} & \bbtwo }
 \]
in \Cat, where the two parallel 2-cells are equal, and let 
$\Phi_\tight(a)=1$ and $\Phi_\tight(b)=0$. 

Then a $\Phi$-weighted limit of $\alpha,\beta\colon f \rightrightarrows g$ is an
equifier, i.e.\ a morphism $e$ such that $\alpha .e = \beta .e$ which
is universal with this property, such that moreover $e$ is tight and
detects tightness.  We call such a limit a \textbf{$p$-rigged equifier};
by~\cite[Prop.~2.3]{bkp:2dmonads}, such equifiers lift to $T\bbAlg_p$ for any 2-monad $T$.

Now suppose that in \bbD we require the morphism that is the domain of
the 2-cells to be tight, and take $\Phi_\tight$ to be constant at $1$.
Then a \bbD-diagram is a parallel pair of 2-cells whose common domain
is tight, and a $\Phi$-weighted limit is again an equifier which again
is tight and detects tightness.   We call such a limit an \textbf{$l$-rigged equifier};
by~\cite[Prop.~4.3]{lack:lim-lax},
such equifiers lift to $T\bbAlg_l$ for any 2-monad $T$.  Of
course, dually we have \textbf{$c$-rigged equifiers}, which lift to $T\bbAlg_c$.

\subsubsection{Descent objects}
\label{sec:descobj}

Write $\bbone=\{0\}$ for the terminal category, $\bbtwo=\{0\to 1\}$ for 
the free-living arrow, and $\bbthree=\{0\to 1\to 2\}$ for the free-living
composable pair.  Consider the functors
\[\xymatrix{
\bbone \ar@<2ex>[r]^{\delta_0} \ar@<-2ex>[r]_{\delta_1} &
\bbtwo \ar[l]|{\sigma} \ar@<2ex>[r]^{\delta_0} \ar[r]|{\delta_1} 
\ar@<-2ex>[r]_{\delta_2}  & \bbthree }\]
where $\delta_i$ is the inclusion which omits $i$.
Let \sDl be the locally discrete sub-2-category of \Cat generated by the functors in the diagram, and let 
$\Phi_\loose:\sDl\to\Cat$ be the inclusion. A diagram 
$G:\sDl\to\sA$ in a 2-category has the form 
$$\xymatrix{
A_0 \ar@<2ex>[r]^{\delta_0} \ar@<-2ex>[r]_{\delta_1} &
A_1 \ar[l]|{\sigma} \ar@<2ex>[r]^{\delta_0} \ar[r]|{\delta_1} 
\ar@<-2ex>[r]_{\delta_2} & A_2 }$$
and the $\Phi_\loose$-weighted limit of this diagram we call
an {\bf $l$-descent object}.
This is an object $A$ universally equipped with a morphism $p\colon A \to A_0$ and a 2-cell $\pi\colon\delta_1 a\to \delta_0 a$ such that $\delta_0\pi.\delta_2\pi=\delta_1\pi$ and $\sigma\pi=1$. 

Let $\sDt$ be the sub-2-category of $\sDl$ generated by the functors
\[\xymatrix{
\bbone \ar@<2ex>[r]^{\delta_0}  &
\bbtwo \ar[l]|{\sigma} \ar@<2ex>[r]^{\delta_0} \ar[r]|{\delta_1} 
  & \bbthree }\]
and let $\Phi_\tight:\sDt\to\Cat$ be the 2-functor constant 
at $\bbone$.   Define
$\phi:\Phi_\tight\to\Phi_\loose J_\bbD$ to have components at $\bbone$, $\bbtwo$,
and $\bbthree$ given by the identity, $\delta_0:\bbone\to\bbtwo$, and
$\delta_0\delta_0:\bbone\to\bbthree$, respectively.  This now gives a weight 
$\Phi:\bbD\to\bbF$. The $\Phi$-weighted limit of a diagram 
$$\xymatrix{
A_0 \ar@<2ex>[r]^{\delta_0} \ar@{~>}@<-2ex>[r]_{\delta_1} &
A_1 \ar[l]|{\sigma} \ar@<2ex>[r]^{\delta_0} \ar[r]|{\delta_1} 
\ar@<-2ex>@{~>}[r]_{\delta_2} & A_2 }$$
in \bbA is an $l$-descent object $(p\colon A\to A_0,\pi\colon \delta_1 p\to \delta_0 p)$
in $\sAl$ for which $p$ is tight and detects tightness.
We call such a limit an \textbf{$l$-rigged $l$-descent object}.

The lifting of such descent objects to $T\bbAlg_l$ for a 2-monad $T$ was not
considered explicitly in \cite{lack:lim-lax}, but it could be treated by similar
techniques to those used there: either by giving a direct construction, or by 
constructing the descent object using inserters and equifiers.


Dually, there are $c$-descent objects, in which the direction of the 2-cell
$\pi$ is reversed, and the corresponding  
\textbf{$c$-rigged $c$-descent objects} can be shown to lift to $T\Alg_c$.
There are also $p$-descent objects, in which $\pi$ is required to be invertible,
and these lift to $T\Alg$.

Note that $\sDl$ admits an automorphism which swaps the \Cat-weights
for $l$-descent objects and $c$-descent objects, but this is no longer
true for \bbD and the rigged weights.

\subsubsection{Eilenberg-Moore objects}
\label{sec:EM-objects}

Recall from \S\ref{sec:monads-in-2cat} that monads in a 2-category \sK are
in bijection with 2-functors from $\sB\bbDelta_+$ to \sK. We may regard
$\sB\bbDelta_+$ as an inchordate \sF-category \bbD, 
so that $\sD_\loose=\sB\bbDelta_+$ and $\sD_\tight$ is the terminal 2-category. 

If $s$ is a monad on an object $A\in\sK$, and $S:\sB\bbDelta_+\to\sK$ is the 
corresponding 2-functor, an Eilenberg-Moore object $A^s$ for $s$ is a
limit of $S$ weighted by a particular weight $\sB\bbDelta_+\to\Cat$ which
we shall call $\Phi_\loose$.  We obtain a weight $\Phi:\bbD\to\bbF$ 
by setting $\Phi_\tight=\bbone$ and $\phi:\Phi_\tight\to\Phi_\loose J$
the map which picks out the ``generating projection.''

An \bbF-functor $S:\bbD\to\bbK$ is equivalently a monad $s$ on some object
$A$ of  $\sK_\loose$, and a $\Phi$-weighted limit of $S$ is now an 
Eilenberg-Moore object $u:A^s\to A$ for $s$ in $\sK_\loose$ for which $u$ is
tight and detects tightness.

By~\cite[Prop.~4.5]{lack:lim-lax},
limits of this sort lift to $T\bbAlg_c$ for any 2-monad $T$.  Dually,
Eilenberg-Moore objects of comonads lift to $T\bbAlg_l$.  They also lift
to $T\bbAlg_p$, by the results of \cite{bkp:2dmonads}, since they can 
be constructed using inserters and equifiers.

\subsubsection{Powers (cotensors)}
\label{sec:powers}

Let \bbD be the terminal \sF-category, with $\sDt=\sDl=\bbone$. Then 
a weight $\Phi:\bbD\to\bbF$ consists of an object
$X=(x:X_\tight\to X_\loose)$ of \bbF.  A $\Phi$-weighted limit is called a \textbf{power}
or cotensor. 

A diagram $S:\bbD\to\bbK$ consists of an object $S$ of \bbK. The power 
of $S$ by $X$ is written $X\pitchfork S$.
The loose part of its universal
property says that it is the 2-categorical power $L=X_\loose\pitchfork S$ in 
\sKl, defined by a natural isomorphism 
$$\sKl(A,L)\cong\Cat(X_\loose,\sKl(A,S)).$$
The tight part says that a morphism $f:A\to L$ is tight if and only if the 
corresponding $\tilde{f}:X_\loose\to\sKl(A,S)$ restricts to give the dotted 
part of the following commutative square.
$$\xymatrix{
X_\tight \ar[d]_{x} \ar@{.>}[r] & \sKt(A,S) \ar[d]^{J_\bbK} \\
X_\loose \ar[r]_-{\tilde{f}} & \sKl(A,S) }$$
In other words, the projections $X_\loose\pitchfork S\to S$ which correspond to
objects of $X_\tight$ are tight and jointly detect tightness.

Notice, in particular, that if $X$ and \bbK are both chordate, then the tight
part of the universal property is automatic.  Then 
$X$-powers lift to $T\bbAlg_w$, for any
$w$ and any 2-monad $T$
\cite{bkp:2dmonads}.

On the other hand, if $X_\tight$ is empty, then all maps $A\to L$ must be tight.
In particular, if $X_\tight$ is empty and $X_\loose$ is terminal, then all maps 
$A\to L$ are tight and they are in bijection with loose maps $A\to S$, and 
finally this bijection extends to 2-cells. Thus $L$ is a (slightly odd) kind of 
``loose morphism coclassifier''. Such a limit does not generally lift
to $T\bbAlg_w$.

\section{Weak aspects of \sF-category theory}
\label{sec:monads}

In the previous section we developed some of the standard enriched-categorical
notions in the case of enrichment over \sF. In this section we turn to those notions
where some weakness is involved; this is of course absent from general 
enriched category theory.

\subsection{Weak \sF-natural transformations}
\label{sec:wf-transf}

We begin by generalizing the notions of lax, oplax, and pseudo natural
transformations from the  2-categorical setting  to the \sF-categorical
setting.  Given
\sF-categories \bbD and \bbK, we can define an \sF-category
$\bbNat_w(\bbD,\bbK)$ for each flavour of weakness ($w=p,c,l$), where
in each case the objects are the \sF-functors from \bbD to \bbK.

Given \sF-functors $M,N:\bbD\to\bbK$, we define a \textbf{loose $w$-natural transformation} $f\colon M\to N$ to be a $w$-natural transformation (of 2-functors) $f_\loose\colon M_\loose \to N_\loose$, such that $f_\loose J_\bbD \colon M_\loose J_\bbD \to N_\loose J_\bbD$ is strictly 2-natural.
Such a loose $w$-natural transformation is \textbf{tight}
when its components $f_\loose D:M_\loose D\to N_\loose D$ are all tight; 
this amounts to giving a 2-natural transformation $f_\tight:M_\tight\to N_\tight$ such that $f_\loose J_\bbD=J_\bbK f_\tight$.
Note that even a tight $w$-natural transformation is more general than an \sF-natural transformation, in which $f_\loose$ would also have to be 2-natural.

Finally, a \textbf{modification} between $w$-natural transformations $f,g\colon M\to N$ is a modification $f_\loose \to g_\loose$.
When $f$ and $g$ are tight, such a modification induces a unique modification $f_\tight \to g_\tight$, since $J_\bbK$ is locally fully faithful.

We define the \sF-category $\bbNat_w(\bbD,\bbK)$ in the obvious way: its objects are \sF-functors $\bbD\to\bbF$, its tight and loose morphisms are tight and loose $w$-natural transformations, respectively, and its 2-cells are modifications.
For $w=p,l,c$ we may write $\bbPs(\bbD,\bbK)$, $\bbLax(\bbD,\bbK)$, and $\bbOplax(\bbD,\bbK)$ respectively.

Note in particular that being \emph{weak} is independent of being
\emph{loose}; thus  a strict transformation can be either
tight or loose (these are the tight and loose morphisms in $[\bbD,\bbK]$) and
likewise a $w$-transformation can be either tight or loose
(these are the tight and loose morphisms in $\bbNat_w(\bbD,\bbK)$).

\subsection{Weak \sF-transformation classifiers}
\label{sec:weak-transf-class}

We now use the 2-categorical weak morphism classifiers from
\S\ref{sec:weak-morph-class} to build corresponding classifiers for
weak \sF-transformations. Specifically, given a small \sF-category
\bbD and a cocomplete \sF-category \bbK, we shall construct a left
\sF-adjoint to the inclusion $[\bbD,\bbK]\to\bbNat_w(\bbD,\bbK)$. 

First of all, let $\ob\bbD$ be the discrete \sF-category with the same objects 
as \bbD, and $H:\ob\bbD\to\bbD$ the inclusion. Restriction along $H$ 
and left Kan extension induces a comonad $W=W_\bbD$ on $[\bbD,\bbK]$. For
each $F:\bbD\to\bbK$, we form the tight $w$-codescent object $Q_\bbD F$ of
the diagram
$$\xymatrix@C=4pc{
W^3F \ar@<2ex>[r]^{wW^2 F} \ar[r]^{WwWF} \ar@<-2ex>[r]^{W^2wF} & 
W^2 F \ar@<2ex>[r]^{wWF} \ar@<-2ex>[r]^{WwF} & 
WF \ar[l]_{dF} }$$
whose universal property means that loose maps $Q_\bbD F\lto G$ 
in $[\bbD,\bbK]$ correspond 
to $w$-natural transformations $F_\loose\to G_\loose$.  They are tight if their
components are tight. 

Next, let \bbC be $\sDt$, regarded as a chordate \sF-category, with 
$J:\bbC\to\bbD$ the evident inclusion. For an \sF-functor $F:\bbD\to\bbK$,
we may restrict along $J$ and then left Kan extend to obtain an \sF-functor
$\Lan_J(FJ)$, whose universal property means that loose maps 
$\Lan_J(FJ)\lto G$ correspond to (strict) natural transformations $J_\bbK F_\tight\to J_\bbK G_\tight$.
They are tight if their components are tight, or in other words, if they are 
natural transformations $F_\tight\to G_\tight$.

On the other hand, there is a comonad $W_\bbC$ on $[\bbC,\bbK]$ analogous
to $W_\bbD$, and  if  we form the corresponding $Q_\bbC$, then the \sF-functor
$\Lan_J(Q_\bbC(FJ))$ has the universal property that loose maps 
$\Lan_J(Q_\bbC(FJ))\lto G$ correspond to $w$-natural transformations
$J_\bbK F_\tight\to J_\bbK G_\tight$.  They are tight if their components are tight, or in 
other
words, if they are $w$-natural transformations $F_\tight\to G_\tight$. 

The weak morphism classifier is now given by the pushout $\QF F$ as 
in 
$$\xymatrix{
\Lan_J(Q_\bbC(FJ)) \ar[r] \ar[d] & \Lan_J(FJ) \ar[d] \\
Q_\bbD F \ar[r] & \QF F }$$
whose universal property says that a loose morphism $\QF F\lto G$ is a 
$w$-natural $F_\loose\to G_\loose$ for which the induced $F_\loose J\to G_\loose J$
is strict; it is tight if its components are tight. This is exactly what is needed 
for the weak morphism classifier.

Dually, if \bbD is small and \bbK is complete, we have weak morphism
coclassifiers. If \bbK is both complete and cocomplete, then we have both
classifiers and coclassifiers for weak morphisms, giving 
left and right adjoints $\QF_w$ and $\RF_w$ to the inclusion
$[\bbD,\bbK]\to\bbNat_w(\bbD,\bbK)$.

If $\bbK=\bbF$, so that $\Phi\colon\bbD\to\bbF$ can be expressed
as $\phi\colon \Phi_\tight \to \Phi_\loose J$, we can alternatively
construct the weak morphism classifier as follows.
For any \sF-category \bbD, let $Q^\bbD_w$ denote the 2-categorical
weak morphism classifier for the 2-monad on $[\sD_\tight,\Cat]$ whose
category of algebras is $[\sD_\loose,\Cat]$.
We call this the \textbf{relative $w$-transformation
  classifier} for \bbD; its universal property says that strict
natural transformations $Q^\bbD_w F \to G$ correspond to $w$-natural
transformations $F\to G$ which become strict when restricted to $\sDt$.
Then we can define
\begin{align*}
  (\QF_w\Phi)_\tight &= \Phi_\tight\\
  (\QF_w\Phi)_\loose &= Q^\bbD_w (\Phi_\loose).
\end{align*}
The structure map $(\QF_w\Phi)_\tight\to (\QF_w\Phi)_\loose J$ for $\QF_w\Phi$ is 
given by the composite
$$\xymatrix{
\Phi_\tight \ar[r]^-{\phi} & \Phi_\loose J \ar[r]^-{p} & Q^\bbD_w(\Phi_\loose) J. }$$
To see that this is a pointwise full embedding, consider the case $w=l$.
Observe that $\phi$ is a 
pointwise full embedding since $\Phi$ is a weight, while $p$ is a  pointwise
full embedding since it has a left adjoint $qJ:Q^\bbD_w(\Phi_\loose) J\to \Phi_\loose J$
with identity counit (essentially by the argument given in 
Lemma~\ref{lem:w-idempotent} or see \cite{bkp:2dmonads}). Thus 
the composite $p \phi$ is also a pointwise full embedding.

Unlike in the 2-categorical case, the weak \sF-transformation
classifier $\QF_w$ is seemingly not a special case of any construction
that applies to more general \sF-monads; see \S\ref{sec:f-monads}.

As in the 2-categorical case, however, composing $\QF_w$ with the
inclusion gives a comonad on $[\bbD,\bbF]$, which we also call
$\QF_w$, and whose co-Kleisli \sF-category is $\bbNat_w(\bbD,\bbF)$.
Similarly, the composite of $\RF_w$ with the inclusion gives a monad,
also called $\RF_w$, whose Kleisli \sF-category is
$\bbNat_w(\bbD,\bbF)$.

We summarize all the weak transformation classifiers we will need in
this paper, and their notations, as follows.
\begin{itemize}
\item For any small 2-category \sD and any cocomplete 2-category \sK,
  we have a 2-comonad $Q^\sD_w$ on $[\sD,\sK]$, which classifies
  2-categorical $w$-natural transformations.
\item For any small \sF-category \bbD, we have a 2-comonad $Q^\bbD_w$
  on $[\sDl,\Cat]$, which classifies 2-categorical $w$-transformations
  that become strict when restricted to $\sDt$ (the \emph{relative
    $w$-transformation classifier}).  Comparing universal properties,
  we see that if \bbD is inchordate, then $Q^\bbD_w = Q^{\sDl}_w$.
\item For any small \sF-category \bbD and any cocomplete \sF-category
  \bbK, we have an \sF-comonad $\QF^\bbD_w$ on $[\bbD,\bbK]$, which
  classifies both tight and loose weak \sF-transformations as defined
  in \S\ref{sec:wf-transf}.  In the case $\bbK=\bbF$, we have
  $(\QF^\bbD_w\Phi)_\tight = \Phi_\tight$ and $(\QF^\bbD_w\Phi)_\loose
  = Q^\bbD_w (\Phi_\loose)$.
\end{itemize}
We will frequently omit the superscripts and/or subscripts on these
classifiers when there is no danger of confusion.

Note that $\QF$ is left adjoint to $\RF$, since we have
\[ [\bbD,\bbK](\QF_w F,G) \cong \bbNat_w(\bbD,\bbK)(F,G) \cong
[\bbD,\bbK](F,\RF_w G).
\]
Moreover, we also have the following (standard)
``adjointness'' with respect to weighted limits.

\begin{lem}\label{thm:qr-adjt}
  For any complete \sF-category \bbK, any weight $\Phi\colon
  \bbD\to\bbF$, and any diagram $G\colon \bbD\to\bbK$, we have
  $\{\Phi, \RF_w G \} \cong \{ \QF_w \Phi, G\}$.
\end{lem}
\begin{proof}
  For any $A\in\bbK$, we have
  \begin{align*}
    \bbK(A, \{\Phi, \RF_w G \} )
    &\cong [\bbD,\bbF] (\Phi, \bbK(A,\RF_w G) )\\
    &\cong [\bbD,\bbF] (\Phi, \RF_w \bbK(A,G) )\\
    &\cong [\bbD,\bbF] (\QF_w \Phi, \bbK(A,G) )\\
    &\cong \bbK(A, \{\QF_w \Phi,G\} ).
  \end{align*}
  where we have used the fact that since $\RF$ is a limit
  construction, it is preserved by the representable $\bbK(A,-)$.
\end{proof}

In particular, for any $D\in\bbD$ we have
\[ (\RF_w G)(D) \cong \{ \bbD(D,-), \RF_w G \} \cong \{ \QF_w \bbD(D,-), G \}
\]
so that $\RF_w$ is itself a weighted limit construction.  (To those who
are familiar with the behavior of weighted limits, this is also evident
from our construction of $\RF$ out of other weighted limits.)

\subsection{\sF-monads}
\label{sec:f-monads}

By an \textbf{\sF-monad} we mean, of course, a monad $T\maps \bbK\to\bbK$ in
the 2-category $\sF\Cat$ of \sF-categories, \sF-functors, and
\sF-transformations.  In particular, this means that the components of
its multiplication and unit are tight, and strictly natural with
respect to both tight and loose morphisms.  We denote the
Eilenberg-Moore object in $\sF\Cat$ of such a $T$ by $T\bbAlg_s$.  The
objects of $T\bbAlg_s$ are the (strict) $T_\tight$-algebras, and the
tight morphisms are the strict $T_\tight$-morphisms, which we call
\textbf{strict $T$-morphisms}.  The loose morphisms in $T\bbAlg_s$, on
the other hand, are the strict $T_\loose$-morphisms (where we regard
$T_\tight$-algebras as $T_\loose$-algebras in the evident way).
$T\bbAlg_s$ has the usual universal property with respect to
\sF-functors $G\maps \bbX\to \bbK$ equipped with a $T$-algebra
structure $TG\to G$ which is \sF-natural.

If we replace \sF-naturality in this universal property by weak
\sF-naturality of the three kinds considered in \S\ref{sec:wf-transf},
we thereby characterize a trio ($w=p,c,l$) of \sF-categories which we denote
$T\bbAlg_w$.  Explicitly:
\begin{itemize}
\item An object of $T\bbAlg_w$ is a (strict) $T_\tight$-algebra (hence also a $T_\loose$-algebra).
\item A tight morphism in $T\bbAlg_w$ is a strict $T_\tight$-morphism
  (hence also a strict $T_\loose$-morphism).
\item A loose morphism in $T\bbAlg_w$ is a $w$-$T_\loose$-morphism; we
  call these \textbf{$w$-$T$-morphisms}.
\item A transformation is a $T_\loose$-transformation.
\end{itemize}
For instance, a loose morphism $(f,\fbar)\maps (A,a)\lto (B,b)$ in $T\bbAlg_l$
consists of a loose morphism $f\maps A\lto B$ in \bbK\ together with a
2-cell
\[\xymatrix{TA \ar@{~>}[r]^{Tf}\ar[d]_a  \drtwocell\omit{\fbar} & TB \ar[d]^b\\
  A \ar@{~>}[r]_f & B}\]
satisfying the usual axioms.
Note that if \bbK is chordate, so that
$T$ is just a 2-monad on the 2-category $\sK_\tight = \sK_\loose$,
then the \sF-category $T\bbAlg_w$ defined above can be identified with
the \sF-category of the same name defined in \autoref{eg:fcat-alg}.
This is essential for applications to 2-category theory.

The universal property of $T\bbAlg_w$ says that
\begin{equation}\label{eq:univ-prop-Talg_w}
\bbNat_{\wbar}(\bbX,T\bbAlg_w) \iso \bbNat_{\wbar}(\bbX,T)\bbAlg_w.
\end{equation}
(Note the reversal of sense in the weak natural transformations, as in \S\ref{sec:2-monads}.)
As with the 2-categorical version~\eqref{eq:natalg}, we find it
helpful to make~\eqref{eq:univ-prop-Talg_w} more explicit in a couple
of cases.  Suppose that $w=l$ and that $\bbX$ is the inchordate
\bbtwo.  Then an object of $\bbOplax(\bbX,T\bbAlg_l)$ is simply a
loose and lax $T$-morphism $(f,\fbar)\colon (A,a) \lto (B,b)$, as
above.  On the other hand, an $\bbOplax(\bbX,T)$-algebra consists of a
loose morphism $f\colon A\lto B$ in \sK (that is, an object of
$\bbOplax(\bbX,\bbK)$) together with an oplax natural transformation
from $Tf$ to $f$ whose components are tight (since the structure maps
of any algebra for an \sF-monad are tight, and the tight morphisms in
$\bbOplax(\bbX,\bbK)$ have tight components).  This consists of tight
morphisms $a\colon TA\to A$ and $b\colon TB\to B$ and a 2-cell
$\fbar\colon b.Tf \to f.a$, and as before the algebra axioms assert
precisely that $(A,a)$ and $(B,b)$ are $T$-algebras and $(f,\fbar)$ is
a lax $T$-morphism.

Now a loose morphism in $\bbOplax(\bbX,T\bbAlg_l)$ from $(f,\fbar)$ to
$(g,\gbar)\colon (C,c) \lto (D,d)$ is a loose oplax transformation;
thus it consists of loose and lax $T$-morphisms $(h,\hbar)\colon
(A,a)\lto (C,c)$ and $(k,\kbar)\colon (B,b) \lto (D,d)$ (these being
loose morphisms in $T\bbAlg_l$), together with a $T$-transformation
$\alpha\colon (k,\kbar)(f,\fbar) \to (g,\gbar)(h,\hbar)$.
 On the other hand, a loose and lax
morphism of $\bbOplax(\bbX,T)$-algebras consists of a loose oplax
transformation from $f\colon A\to B$ to $g\colon C \to D$, hence loose
morphisms $h\colon A\lto C$ and $k\colon B\lto D$ and a 2-cell
$\alpha\colon k f \to g h$, together with a modification consisting of
2-cells $\hbar\colon c.Th \to h.a$ and $\kbar\colon d.Tk \to k.b$.  As
before, the axioms assert precisely that $(h,\hbar)$ and $(k,\kbar)$
are lax $T$-morphisms and $\alpha$ is a $T$-transformation.

Finally, a tight morphism in $\bbOplax(\bbX,T\bbAlg_l)$ from
$(f,\fbar)$ to $(g,\gbar)$ is a tight oplax transformation; thus it
consists of tight and strict $T$-morphisms $h\colon (A,a) \to (C,c)$
and $k\colon (B,b) \to (D,d)$ and a $T$-transformation $\alpha\colon k
f \to g h$.  On the other side, a tight and strict morphism of
$\bbOplax(\bbX,T)$-algebras consists of a tight oplax transformation
from $f$ to $g$, hence tight morphisms $h\colon A\to C$ and $k\colon
B\to D$ and a 2-cell $\alpha\colon k f \to g h$, such that $c.Th =
h.a$ and $d.Tk = k.b$.  The axioms say exactly that $h$ and $k$ are
strict $T$-morphisms and $\alpha$ is a $T$-transformation.

If instead we take \bbX to be the chordate \bbtwo, then in
$\bbOplax(\bbX,T\bbAlg_l)$ we require that $f$ and $g$ must be tight
and strict $T$-morphisms.  On the other side, we obtain the fact that
$f$ and $g$ are tight from the fact that we have functors out of \bbX,
while the fact that $\fbar$ and $\gbar$ are identities follows from
the requirement that an oplax \sF-transformation be strictly natural
with respect to tight morphisms.

Note that in order to make~\eqref{eq:univ-prop-Talg_w} true, we need
the ``notions of tightness'' in $T\bbAlg_w$ and in
$\bbNat_w(\bbD,\bbK)$ to be different.  Specifically, the tight
morphisms in $T\bbAlg_w$ are strict, whereas the tight morphisms in
$\bbNat_w(\bbD,\bbK)$ are strictly natural only with respect to tight
maps.  Unfortunately, this seems to mean that unlike the situation in
2-category theory, the weak \sF-transformation category
$\bbNat_w(\bbD,\bbK)$ is not of the form $T\bbAlg_w$ for any \sF-monad
$T$.

Finally, as in the 2-categorical case, any lax morphism of \sF-monads
induces a functor between \sF-categories of algebras and weak
morphisms in a straightforward way.

\section{Rigged limits lift}
\label{sec:rigged-limits}

Let $\Phi\maps \bbD\to\bbF$ be an \sF-weight, let \bbK\ be an
\sF-category, and let $T$ be an \sF-monad on \bbK. (The reader is
encouraged to think of \bbK\ as the chordate \sF-category associated
to a 2-category and of $T$ as arising from a 2-monad, since this is
the case of most interest.) Our goal in this section is to show that
$\Phi$-weighted limits are created by $U_w\colon T\bbAlg_w
\to \bbK$ for any \sF-monad $T$
if and only if $\Phi$ is ``rigged'' (and we will define what this means).
Recall that as remarked after \autoref{def:lifting}, since the \sF-categorical
$U_w$ is conservative, it creates any limits that it lifts.

Once again, $w$ could be any of $p$, $l$, or $c$; there will be a
notion of ``$w$-rigged'' for each value of $w$.  We will state the
definitions and theorems of this section for general $w$, but in most
of the proofs we will describe only the case $w=l$ explicitly.  The
case $w=c$ is a formal dual; while the proofs may easily be adapted to
the case $w=p$ by requiring the relevant 2-cells to be invertible.

\subsection{Rigged weights}
\label{sec:rigged-weights}

To make a start, we suppose that our weight $\Phi:\bbD\to\bbF$ is a
$\QF_\wbar$-coalgebra, with structure map $s\colon \Phi \to \QF_{\wbar}
\Phi$.  Note the reversal of sense: we are considering liftings to
$T\bbAlg_w$, but we assume $\Phi$ to be a $\QF_\wbar$-coalgebra.

At the moment, this hypothesis may seem somewhat unmotivated, but at
least in the case $w=p$ it is a strengthening of \emph{flexibility},
while the notion of \emph{PIE-limit} (the limits known to lift when $w=p$) is
also a strengthening of flexibility.  In \S\ref{sec:pie} we will show
that in fact, PIE-limits are precisely the 2-categorical
$Q$-coalgebras.  Moreover, one of the results of~\cite{lack:lim-lax}
is that \emph{oplax limits} lift to $T\Alg_l$, and the weights for
oplax limits are exactly the \emph{cofree} $Q_c$-coalgebras.

In \S\ref{sec:lifting} we will show that under the hypothesis
that $\Phi$ is a $\QF_\wbar$-coalgebra, we can 
construct, for each 
$G:\bbD\to T\bbAlg_w$, a $T$-algebra $L$
with the universal property of the limit $\{\Phi_\loose,G_\loose\}$ in 
$(T\bbAlg_w)_\loose$, lifting the limit $\{\Phi_\loose,UG_\loose\}$ in \sKl.
This hypothesis is not, however, enough to get the full universal property of the 
limit $\{\Phi,G\}$.  Recall from \S\ref{sec:limits} that in
addition to being a limit in the 2-category of loose morphisms, an
\sF-limit must ``detect tightness.''

\begin{eg}
Consider powers (\S\ref{sec:powers})
by an object $0\to C$ of \sF, with $C$ a non-empty category. Let
$T$ be a 2-monad on a 2-category \sK, seen as a chordate \sF-category,
and let $(B,b)$ be a $T$-algebra. 
A power of $B\in\bbK$ by $0\to C$ consists of a power $C\pitchfork B$ in \sK;
the tight part of the universal property is automatic, since all morphisms in \sK
are tight. This power will lift to a power of $(B,b)$ in $T\bbAlg_w$ if (i) the power 
$C\pitchfork B$ lifts to a power $C\pitchfork(B,b)$ in $(T\bbAlg_w)_\loose$, and 
(ii) all morphisms into $C\pitchfork(B,b)$ are tight (that is, strict). Now (i) will
always hold, but (ii) generally will not, even in the case where $C$ is the terminal
category \bbone.
\end{eg}

To ensure the tight aspect of the universal property, we must impose
an additional condition on $\Phi$.
Recall that we write $J\maps \sD_\tight \to\sD_\loose$ for the
2-functor that underlies an \sF-category \bbD, and $\ph\maps
\Phi_\tight \to \Phi_\loose J$ for the structure map of an \sF-weight
$\Phi$.

\begin{defn}\label{def:rigged}
  We say that an \sF-weight $\Phi\maps \bbD\to\bbF$ is
  \textbf{$w$-rigged} if
  \begin{enumerate}
  \item $\Phi$ is a $\QF_\wbar$-coalgebra, and
  \item The induced morphism $\phbar\maps \Lan_J \Phi_\tight \to \Phi_\loose$ is pointwise surjective on objects.
  \end{enumerate}
\end{defn}

The extra condition may seem somewhat odd; its importance is due to
the following alternative characterization.

\begin{lem}\label{thm:so-char}
  The following are equivalent for a 2-natural transformation $f\maps
  \Phi\to \Psi$ between 2-functors $\Phi,\Psi\maps \sD\to\Cat$.
  \begin{enumerate}
  \item $f$ is pointwise surjective on objects.\label{item:so1}
  \item Precomposition with $f$ reflects identities, i.e.\ for any
    $g,h\maps \Psi\to \Upsilon$ and modification $\beta\maps g\to h$,
    if $\beta f$ is an identity then so is $\beta$.\label{item:so2}
  \item As in~\ref{item:so2}, but only when $\beta$ is known to be an
    isomorphism.\label{item:so3}
  \end{enumerate}
\end{lem}
\begin{proof}
  Since $(\beta f)_{d,x} = \beta_{d,f(x)}$ for $d\in\sD$ and $x\in
  \Phi(d)$, if $f$ is pointwise surjective on objects and $(\beta
  f)_{d,x}$ is an identity for all $d$ and $x$, then $\beta_{d,y}$ is
  an identity for all $d\in\sD$ and $y\in \Psi(d)$.
  Thus~\ref{item:so1} implies~\ref{item:so2}, and
  clearly~\ref{item:so2} implies~\ref{item:so3}, so
  suppose~\ref{item:so3}.  Pick any $d_0\in\sD$ and $y_0\in \Psi(d)$,
  and let $\Upsilon = \Ran_{d_0} C$ be the co-free diagram at $d_0\in
  \sD$ on the chaotic category $C$ with two objects 0 and 1; thus a
  2-natural transformation $\Psi\to\Upsilon$ is determined by a
  functor $\Psi(d_0)\to C$.  Let $g\maps \Psi\to\Upsilon$ be
  determined by the functor $\Psi(d_0)\to C$ constant at 0, and let
  $h$ be determined by the functor $\Psi(d_0)\to C$ sending $y_0$ to 1
  and everything else to 0.  Then there is an invertible modification
  $\beta\maps g\xrightarrow{\iso} h$ such that $\beta_{d,y}$ is an
  identity whenever $d\neq d_0$ or $y\neq y_0$, but $\beta_{d_0,y_0}$
  is not an identity.  Thus $\beta f$ cannot be an identity either,
  and so there must be some $x\in \Phi(d_0)$ with $f(x)=y_0$.  Hence
  $f$ is pointwise surjective on objects, and so~\ref{item:so3}
  implies~\ref{item:so1}.
\end{proof}

\begin{rmk}
Our \sF-categories depend heavily on the class of full embeddings (functors
which are injective on objects and fully faithful). These are the right class 
of a factorization system on \Cat for which the left class consists of the 
functors which are surjective on objects. But we have also referred in 
passing to the more general \sFF-categories, which involve merely fully faithful
functors, and we have promised that all our results extend to the setting of 
\sFF-categories. Since fully faithful functors form the right class of a factorization
system on \Cat for which the left class consists of the functors which are bijective
on objects, one might guess that the notion of rigging for \sFF-categories
would involve $\phi:\Phi_\tight\to\Phi_\loose J$ which is pointwise bijective on 
objects. This is not the case: we still use surjectivity on objects,
and \autoref{thm:so-char} explains why.

On the other hand, we would need to modify surjectivity on objects to
obtain a notion of rigging appropriate for \sF-categories of algebras
which combine pseudo and lax morphisms, instead of strict and lax
ones, as suggested in the introduction.
\end{rmk}

The relationship between the two conditions defining a rigged
weight is further clarified by the following observation.

Recall that  $\QF_\loose = Q^\bbD$ is the 2-categorical relative
$\wbar$-morphism classifier on $[\sDl,\Cat]$, and hence is
$\wbar$-idempotent.  In particular, for $\QF$-coalgebras
$\Phi,\Psi\colon \bbD\to\bbF$, any loose map $\Phi \lto \Psi$ in
$[\bbD,\bbF]$, being a 2-natural transformation $\Phi_\loose \to
\Psi_\loose$, is automatically a $\wbar$-$\QF_\loose$-morphism.
Moreover, any morphism between such loose maps is automatically a
$\QF_\loose$-transformation.  The next lemma says that if $\Phi$ and
$\Psi$ are $w$-rigged, we also have a corresponding property for tight
morphisms.

\begin{lem}\label{thm:rigged-strictmor}
If $\Phi$ and $\Psi$ are $\QF_\wbar$-coalgebras, and $\Phi$ is $w$-rigged,
then any (strict) \sF-natural transformation $\Phi\to\Psi$ is automatically
a strict $\QF_\wbar$-morphism.
\end{lem}
\begin{proof}
  As usual, we write in the case $w=l$.
  Let $\QF = \QF_c$, let $\Phi,\Psi$ have structure maps 
$s_\Phi\colon \Phi\to \QF \Phi$ and $s_\Psi\colon \Psi\to \QF \Psi$, and let $f\colon \Phi \to \Psi$ be \sF-natural.  Suppose that $\Phi$ is $w$-rigged. 
  We must show that $(\QF f)_\tight (s_\Phi)_\tight = (s_\Psi)_\tight
  f_\tight$ and $(\QF f)_\loose (s_\Phi)_\loose = (s_\Psi)_\loose f_\loose$.
  However, $(s_\Phi)_\tight$ and $(s_\Psi)_\tight$ are the identity, since $\QF_\tight$ is the identity, so the first of these is trivial.
  Moreover, since $\QF_\loose$ is colax-idempotent, we have a unique colax
  $\QF_\loose$-morphism structure map
  $\fbar\colon (\QF f)_\loose. (s_\Phi)_\loose \to (s_\Psi)_\loose. f_\loose$, so it
  suffices to show that \fbar is an identity.

  Now by definition, \fbar is the composite
  \[ (\QF_\loose f_\loose). s_\Phi \xrightarrow{\eta. (\QF_\loose
    f_\loose). s_\Phi } s_\Psi. q_\Psi. (\QF f)_\loose. s_\Phi = s_\Psi . f_\loose . q_\Phi. s_\Phi = s_\Psi f_\loose. \]
  where $\eta$ is the unit of the adjunction  $q_\Phi \dashv s_\Phi$. 
  Now let us apply $J_\bbD$ to this composite and precompose with $\phi\colon \Phi_\tight \to \Phi_\loose J$.
  Since $J(\QF_\loose f_\loose) . J(s_\Phi) . \phi = J(s_\Psi) . \psi . f_\tight$ (using the fact that  $s_\Phi$  is strictly \sF-natural), we obtain
  \[ J(s_\Psi) . \psi . f_\tight \xrightarrow{J(\eta . s_\Psi) . \psi . f_\tight} J(s_\Psi . q_\Psi . s_\Psi) . \psi . f_\tight \]
  But the counit of the adjunction  $q_\Phi\dashv s_\Phi$  
is an identity, hence so also is  $\eta.s_\Phi$. 
  Thus, $J(\fbar). \phi$ is an identity, which equivalently means that
  \[ \xymatrix{ \Lan_J \Phi_\tight \ar[r]^-{\phbar} & \Phi_\loose \rrtwocell{\fbar} && \QF_\loose \Psi_\loose }\]
  is an identity.  Since \phbar is pointwise surjective on objects, by \autoref{thm:so-char} this implies that \fbar is an identity, as desired.
\end{proof}

Thus, the forgetful functor $\QF_\wbar \bbCoalg_\wbar \to
[\bbD,\bbF]$, when restricted to $w$-rigged weights, is fully faithful
in the \sF-enriched sense.  In particular, if an \sF-weight
$\Phi$ admits two \QF-coalgebra structures of which one (hence also
the other) is rigged, then the identity $\Phi\to\Phi$ is a strict
\QF-coalgebra map between them, and hence the two structures coincide.
Thus, ``being $w$-rigged'' (unlike ``being a \QF-coalgebra'') is a mere property of an
\sF-weight, not structure on it.  (However, see also
\S\ref{sec:canonical}.)

\subsection{Reduction to special \bbK}
\label{sec:reduction}

We start with the following simplification, which will be useful at
various stages in the proofs.

\begin{prop}
\label{prop:reduction}
For a weight $\Phi:\bbD\to\bbF$ and a given choice of $w$, 
the following conditions are equivalent:
\begin{enumerate}
\item $\Phi$-weighted limits are created by $U_w:T\bbAlg_w\to\bbK$ for 
all \sF-categories \bbK and all \sF-monads $T$ on \bbK;\label{item:red1}
\item $\Phi$-weighted limits are created by $U_w:T\bbAlg_w\to\bbK$ for 
all complete \sF-categories \bbK and all \sF-monads $T$ on \bbK;\label{item:red2}
\item $\Phi$-weighted limits are created by $U_w:T\bbAlg_w\to\bbK$ 
for all presheaf \sF-categories $\bbK=[\bbC,\bbF]$ and all \sF-monads $T$
on \bbK;\label{item:red3}
\item $\Phi$-weighted limits are created by $U_w:T\bbAlg_w\to\bbK$ for 
all small \sF-categories \bbK and all \sF-monads $T$ on \bbK.\label{item:red4}
\item $\Phi$-weighted limits are created by $U_w:T\bbAlg_w\to\bbK$ for 
all chordate \sF-categories \bbK and all \sF-monads $T$ on \bbK.\label{item:red5}
\end{enumerate}
\end{prop}

\begin{proof}
Clearly~\ref{item:red1} implies all the other conditions, and~\ref{item:red2} implies~\ref{item:red3}. We shall show that~\ref{item:red3} implies~\ref{item:red4}, 
that~\ref{item:red4} implies~\ref{item:red1}, and that~\ref{item:red5} implies~\ref{item:red2}.

Suppose~\ref{item:red3} and let $T$ be an \sF-monad on a small \sF-category \bbK,
and let $G:\bbD\to T\bbAlg_w$ be an \sF-functor for which the limit 
$\{\Phi,UG\}$ exists. Since \bbK is small, we may form the presheaf \sF-category 
$\widehat{\bbK}=[\bbK\op,\bbF]$ and the monad $\widehat{T}$ on $\widehat{\bbK}$ 
induced by left Kan extension along $T$. There is an induced
fully faithful $T\bbAlg_w\to\widehat{T}\bbAlg_w$ lifting the 
Yoneda embedding $\bbK\to\widehat{\bbK}$. A $\widehat{T}$-algebra is
in the image of $T\bbAlg_w$ if and only if the underlying object in $\widehat{\bbK}$ 
is in the image of the Yoneda embedding. Now the limit 
$\{\Phi,UG\}$ in \bbK is preserved by Yoneda, and since $\widehat{\bbK}$ is 
a presheaf category this limit lifts to $\widehat{T}\bbAlg_w$; but now
this limit also lies in $T\bbAlg_w$ and so is a limit there. Thus~\ref{item:red3} implies~\ref{item:red4}.

Suppose~\ref{item:red4} and let
$T$ be an \sF-monad on an arbitrary \sF-category \bbK, and let $G:\bbD\to T\bbAlg_w$ be an \sF-functor for which the limit $\{\Phi,UG\}$ exists. 
First choose a small full subcategory \bbB of \bbK which is closed under the action
of $T$ and contains $\{\Phi,UG\}$ and the image of $UG$. Then the limit 
$\{\Phi,UG\}$ lifts to $S\bbAlg_w$, where $S$ is the restriction of $T$ to \bbB. 
Our lifted limit has the correct universal property in $S\bbAlg_w$, but we still
need to check the universal property in the larger \sF-category $T\bbAlg_w$. 
But this can be done one object at a time: for each object $C\in\bbK$, we may
enlarge \bbB to a small full subcategory $\bbC$ of \bbK having the same properties
as before, but also containing $C$. Now our lifted limit also has the correct universal
property in $R\bbAlg_w$, where $R$ is the restriction of $T$ to \bbC, and since
we can do this for any object $C$, it has the correct universal property in all of 
$T\bbAlg_w$. Thus~\ref{item:red4} implies~\ref{item:red1}.

Suppose~\ref{item:red5} and let $T$ be an \sF-monad on a complete \sF-category \bbK, and let $G:\bbD\to T\bbAlg_w$ be an \sF-functor. The \sF-monad $T$ on 
\bbK induces a 2-monad $T_\loose$ on \sKl. As usual, we regard \sKl as a chordate
\sF-category. The canonical \sF-functor $\bbK\to\sKl$ lifts to an \sF-functor
$P:T\bbAlg_w\to T_\loose\bbAlg_w$ whose loose part $P_\loose$ is 2-fully faithful;
a $T_\loose$-algebra $(A,a)$ lies in the image of $P_\loose$ if and only if $a:TA\to A$
is tight. 

By~\ref{item:red5}, the limit $L=\{\Phi,UPG\}$ in \sKl lifts to a limit
$(L,\ell)=\{\Phi,PG\}$ in $T_\loose\bbAlg_w$; this will be a $T$-algebra
if and only if $\ell$ is tight. Furthermore, since $(L,\ell)$ is in particular
a limit $\{\Phi_\loose,P_\loose G_\loose\}$ in $(T_\loose\bbAlg_w)_\loose$, and
$P_\loose$ is 2-fully faithful and so reflects limits, if $(L,\ell)$ is a $T$-algebra
then it will be a limit $\{\Phi_\loose,G_\loose\}$ in $(T\bbAlg_w)_\loose$.

Now the projections $p_{\loose,\phi(a)}:L\to UG_\loose D$, for
$D\in\bbD$ and $a\in\Phi_\tight D$, are tight in \bbK and jointly detect
tightness. They are also strict $T_\loose$-morphisms, and so the
square
\[\xymatrix@C=4pc{
TL \ar[r]^-{Tp_{\loose,\phi(a)}} \ar[d]_{\ell} & TUG_\loose D \ar[d] \\
L \ar[r]_-{p_{\loose,\phi(a)}} & UG_\loose D ,}
\]
in which the right leg is the structure map of $G_\loose D$, is
commutative, and $\ell$ will be tight if and only if the common
composites $TL \to U G_\loose D$ are all tight.  But the right leg is
tight since each $G D$ is a $T$-algebra, and the top leg is tight
since $p_{\loose,\phi(a)}$ is a tight projection. Thus $\ell$ is
tight, and $(L,\ell)$ is the limit $\{\Phi_\loose,G_\loose\}$ in
$(T\bbAlg_w)_\loose$.

Moreover, since the above projections $p_{\lambda,\phi(a)}$ 
are strict $T_\loose$-morphisms,
they are also strict $T$-morphisms, and since they are tight, they are
tight morphisms in $T \bbAlg_w$.  Thus, it remains to show that they
jointly detect tightness in $T \bbAlg_w$.

Let $(A,a)$ be a $T$-algebra, and $(f,\fbar):(A,a)\to(L,\ell)$ a
loose morphism in $T\bbAlg_w$. Supppose that the composite
$$\xymatrix@C=3pc{
  (A,a) \ar[r]^{(f,\fbar)} & (L,\ell) \ar[r]^-{p_{\loose,\phi(a)}} &
  GD }$$ is tight for each $D\in\bbD$ and each $a\in\Phi_\tight D$.
The projections $p_{\loose,\phi(a)}$ jointly detect tightness of
$T_\loose$-morphisms (that is, they jointly detect strictness) and so
$(f,\fbar)$ is strict. On the other hand, the projections
$p_{\loose,\phi(a)}:L\to UGD$ jointly detect tightness of morphisms in
\bbK, and so $f$ is tight. Thus $(f,\fbar)$ is tight in $T\bbAlg_w$,
and so $(L,\ell)$ has the full universal property of $\{\Phi,G\}$.
\end{proof}

Our main interest is in the weights which satisfy~\ref{item:red1}, but
it is also useful to have~\ref{item:red5}, which says that
restricting our attention to 2-monads on 2-categories does not affect
the resulting class of \sF-limits.  In other words, the introduction
of \sF-categories has not ``changed the problem'' from the original
2-categorical question.

The other conditions are more technical.  In particular, it will be convenient 
in our analysis of the lifting of limits to suppose that \bbK is complete, and 
by this last result there is no loss of generality in doing so. 
In fact we could have given still more equivalent conditions; for example that \bbK 
was complete and chordate, or small and chordate.

\subsection{Lifting of limits}
\label{sec:lifting}

We now embark on the actual proof that rigged limits lift.
Suppose that $w=l$ and that \bbK is complete,
and consider a diagram $G:\bbD\to
T\bbAlg_{ l}$ with the limit
$L=\{\Phi,UG\}\in\bbK$. By the isomorphism
\begin{equation}
  \bbOplax(\bbD,T\bbAlg_{
    l})\cong\bbOplax(\bbD,T)\bbAlg_{ l}\label{eq:oplaxalg}
\end{equation}
we have a tight oplax \sF-natural transformation $g\colon TUG\to UG$,
which makes $UG$ into an $\bbOplax(\bbD,T)$-algebra in
$\bbOplax(\bbD,\bbK)$.
Now consider the upper path around the following diagram
\begin{equation}\label{eq:defn-zeta}
\xymatrix @C1pc {
\Phi \ar[rr]^-{\eta} \ar[d]_{s} && 
\bbK(L,UG-) \ar[rr]^-{T} &&
\bbK(TL,TUG-) \ar[d]^-{\bbK(TL,g-)} \\ 
\QF\Phi \ar[rrrr]_-{\zeta} &&&& \bbK(TL,UG-) }
\end{equation}
in which the first two (horizontal) morphisms are tight and strictly \sF-natural, while
the third (vertical) is tight oplax natural.
Therefore, by the universal property of $\QF = \QF_c$, there is 
a unique tight and strict transformation $\zeta\colon \QF \Phi\to\bbK(TL,UG-)$ making the diagram commute. 
Finally, by the universal property of the limit $L$, there
is a unique tight $\ell\colon TL\to L$ making the following diagram commute:
\begin{equation}\label{eq:defn-l}
\xymatrix{
\Phi \ar[r]^-{\eta} \ar[d]_s & \bbK(L,UG-) \ar[d]^{\bbK(\ell,UG-)} \\
\QF \Phi \ar[r]_-{\zeta} & \bbK(TL,UG-) }
\end{equation}
One could now prove directly that this map $\ell\colon TL\to L$ makes
$L$ into a $T$-algebra, but the calculations are lengthy and not particularly 
enlightening.
Thus, we will instead give a more formal approach involving
monad morphisms.

The overall strategy is this. Rather than construct the limit
$\{\Phi,G\}$ separately for each $G\colon \bbD\to T\bbAlg_l$, we do
this functorially.  More precisely, we construct the \sF-functor
$\bbOplax(\bbD,T\bbAlg_l)\to T\bbAlg_l$ which (in a later theorem)
will turn out to send $G$ to the limit $\{\Phi,G\}$. We use
$\bbOplax(\bbD,T\bbAlg_l)$ rather than the simpler $[\bbD,T\bbAlg_l]$
in order to take advantage of the isomorphism~\eqref{eq:oplaxalg}.
This reduces the problem to constructing an \sF-functor
$\bbOplax(\bbD,T)\bbAlg_l\to T\bbAlg_l$, which in turn can be done by
constructing a lax monad morphism from $\bbOplax(\bbD,T)$ to $T$, since
monad morphisms induce liftings not just to their Eilenberg-Moore
objects but also to variants using weak morphisms.

\begin{rmk}\label{rmk:appendix}
It turns out that if we make this construction sufficiently
functorial in the weight $\Phi$ as well, then it is possible to
deduce the universal properties of these limits $\{\Phi,G\}$ from
their functoriality. In an appendix to the paper, we describe 
the resulting proof, which treats weighted limits using profunctors.
In many ways this gives a fuller picture of the situation, but it is 
also somewhat longer, so we have chosen here an intermediate 
approach, which constructs algebra structure using the monad-theoretic 
ideas of the previous paragraph, but then proves the universal property
by showing that the relevant hom-objects can be constructed as certain
descent objects.  
\end{rmk}

\begin{prop}\label{prop:alg-structure-on-L}
If \bbK is complete and $\Phi$ is a $Q_\wbar$-coalgebra, 
the \sF-functor $F=\{\Phi,-\}:[\bbD,\bbK]\to\bbK$
extends to an \sF-functor $F':\bbNat_\wbar(\bbD,\bbK)\to\bbK$ which lifts
to an \sF-functor $F'':\bbNat_\wbar(\bbD,T\bbAlg_w)\to T\bbAlg_w$.
\end{prop}

\proof
The proof involves three monads and various relationships between them. 
The first two monads are $T$ itself and the induced $[\bbD,T]$ on $[\bbD,\bbK]$. 
The \sF-functor $F:[\bbD,\bbK]\to\bbK$ can be given the structure of 
a monad morphism from $[\bbD,T]$ 
to $T$. Explicitly, the 2-cell $\psi:TF\to F[\bbD,T]$ which makes $F$ a monad
morphism has component at $G:\bbD\to\bbK$ given by the canonical 
comparison map $T\{\Phi,G\}\to\{\Phi,TG\}$. Alternatively, one can construct
the lifting as the composite 
$$\xymatrix{
[\bbD,T]\bbAlg_s \ar[r]^{\cong} & 
[\bbD,T\bbAlg_s] \ar[r] & 
T\bbAlg_s }$$
where the first map is the canonical isomorphism, and the second comes 
from the fact that  limits lift to Eilenberg-Moore objects; then as usual this
lifting of $F$ determines a monad morphism structure on $F$. 

Next we need to introduce weakness into the situation. This is done via the 
weak morphism coclassifier $\RF=\RF_\wbar$, seen as a monad on $[\bbD,\bbK]$. Recall
that $\bbNat_\wbar(\bbD,\bbK)$ is the Kleisli object for $\RF$; in particular 
weak maps $G\to H$ correspond to strict maps $G\to \RF H$. 
Since the monad $[\bbD,T]$ on $[\bbD,\bbK]$ 
extends to a monad $\bbNat_\wbar(\bbD,T)$ on the Kleisli object 
$\bbNat_\wbar(\bbD,\bbK)$ of $\RF$, there is an induced distributive law 
$k:[\bbD,T]\RF\to \RF[\bbD,T]$. 

The last ingredient is the relationship between $F$ and $\RF$. This is
where we use the assumption that $\Phi$ is a $\QF_\wbar$-coalgebra. The
coaction $s:\Phi\to \QF_\wbar\Phi$ induces an opaction 
\[f:F\RF=\{\Phi,\RF-\}\cong\{\QF\Phi,-\}\to \{\Phi,-\}=F\]
of $\RF$ on $F$, and so an
extension $F':\bbNat_\wbar(\bbD,\bbK)\to\bbK$ of $F$ to the Kleisli object of $\RF$. 

Next we put these pieces together to make $F'$ into a monad morphism
from $\bbNat_\wbar(\bbD,T)$ to $T$.   Since $\bbNat_\wbar(\bbD,T)$ is
an extension of $T$ to the Kleisli category of $\RF$, and $T$ is an
extension of itself to the Kleisli category of the identity monad, by
\autoref{lem:monad-stuff-2}  it will suffice to show that the
diagram 
$$\xymatrix{
TF\RF \ar[r]^-{\psi} \ar[d]_{Tf} & F[\bbD,T]\RF \ar[r]^-{Fk} & F\RF[\bbD,T] \ar[d]^{f[\bbD,T]} \\
TF \ar[rr]_-{\psi} && F[\bbD,T] }$$
commutes.  This means that for each $G:\bbD\to\bbK$ the corresponding diagram
$$\xymatrix{
T\{\Phi,\RF G\} \ar[r] \ar[d] & \{\Phi,T\RF G\} \ar[r] & \{\Phi,\RF TG\} \ar[d] \\
T\{\Phi,G\} \ar[rr] && \{\Phi,TG\} }$$
commutes; but this reduces, using the isomorphism $\{\Phi,\RF -\}\cong\{\QF \Phi,-\}$
to commutativity of the diagram 
$$\xymatrix{
T\{\QF \Phi,G\} \ar[r] \ar[d]_{T\{s,G\}} & \{\QF \Phi,TG\} \ar[d]^{\{s,TG\}} \\
T\{\Phi,G\} \ar[r] & \{\Phi,TG\} }$$
expressing the naturality of the canonical comparisons (appearing as the 
horizontal maps) with respect to $s$. 
 
Finally, we can take $F''$ to be the composite 
$$\xymatrix{
\bbNat_\wbar(\bbD,T\bbAlg_w) \ar[r]^{\cong} & 
\bbNat_\wbar(\bbD,T)\bbAlg_w \ar[r] & 
T\bbAlg_w }$$
in which the second map is the lifting of the above monad morphism $F'$ from
$\bbNat_\wbar(\bbD,T)$ to $T$.

Since $F''$ is given by lifting $F'$, for any  diagram $G:\bbD\to T\bbAlg_w$ the 
induced $T$-algebra will have underlying object $L$ calculated by composing 
with $U$ to get $UG:\bbD\to\bbK$ and taking the limit $\{\Phi,UG\}$. The algebra
structure $\ell:TL\to L$ is given by the composite
\begin{equation}
  \xymatrix@C=1.5pc{
    T\{\Phi,UG\} \ar[r] & \{\Phi,TUG\} \ar[r] & \{\Phi,\RF UG\} \ar[r]^{\cong} & 
    \{\QF \Phi,UG\} \ar[r] & \{\Phi,UG\} }\label{eq:lifted-tstr}
\end{equation}
where the first map is the canonical comparison, the second comes from the 
weakly natural actions $TGD\to GD$, the third is the canonical isomorphism,
and the last is induced by composition with $s:\Phi\to \QF \Phi$. This 
agrees with the earlier description of $\ell$. 
\endproof

The proposition shows that we have an algebra structure on $L=\{\Phi,UG\}$, 
and a degree of functoriality, but it does not give any sort of universal property. 
We now turn to this, starting with the universal property with respect to loose maps.

\begin{thm}\label{thm:coalg-looselift}
  Let $\Phi\maps \bbD\to\bbF$ be an \sF-weight which is a
  $\QF_\wbar$-coalgebra, let $T$ be an \sF-monad on a complete \sF-category \bbK,
  and let $G\maps \bbD\to T\bbAlg_w$ be an \sF-functor.
Then the $T$-algebra structure on $\{\Phi,U G\}$ constructed in 
\autoref{prop:alg-structure-on-L}
gives it the universal property of the limit
  $\{\Phi_\loose, G_\loose\}$ in the 2-category $(T\bbAlg_w)_\loose$.
\end{thm}
\proof
Once again, we write the proof only for the case $w=l$, with $\QF =\QF _c$, and
we continue writing $L$ for the limit $\{\Phi,UG\}$ 
and $\ell:TL\to L$ for its induced algebra structure.

Write $M$ for the 2-monad on $[\sDt,\Cat]$ whose Eilenberg-Moore
2-category is $[\sDl,\Cat]$ with forgetful 2-functor
$[J_\bbD,\Cat]:[\sDl,\Cat]\to[\sDt,\Cat]$.  Thus an $M$-algebra is a
2-functor $\sDl\to\Cat$, and a colax $M$-morphism is an oplax natural
transformation whose restriction along $J_\bbD$ is strict.  (This
is the 2-monad $M$ for which $\QF _\loose = Q^\bbD$  is the colax morphism
classifier.) We have the 2-category $M\Alg_c$ of $M$-algebras
and colax morphisms.

Let $\bA=(A,a)$ and $\bB=(B,b)$ be $T_\loose$-algebras. 
Then we have a  $c$-descent object
$$\xymatrix{
T_\loose\Alg_l(\bA,\bB) \ar@{.>}[r] &
\sKl(A,B) \ar@{~>}@<2ex>[r]  \ar@<-2ex>[r] &
\sKl(TA,B) \ar[l] \ar@<-2ex>[r] \ar@{~>}@<2ex>[r] \ar[r] &
\sKl(T^2A,B) }$$
in \Cat, in which the straight arrows are all given by precomposition
by some arrow in \sKl, while the wriggly arrows are given by applying $T$ and
then composing with $b$.  In particular, for each $D\in\sDl$ we have a
 $c$-descent object  
$$\small\xymatrix@C=1.5pc{
T_\loose\Alg_l(\bA,G_\loose D) \ar[r] &
\sKl(A,UG_\loose D) \ar@{~>}@<2ex>[r]  \ar@<-2ex>[r] &
\sKl(TA,UG_\loose D) \ar[l] \ar@<-2ex>[r] \ar@{~>}@<2ex>[r] \ar[r] &
\sKl(T^2A,UG_\loose D). }$$
Now the straight arrows are strictly natural in $D$ with respect to
all arrows of \sDl, 
while the wriggly ones are strictly natural with respect to tight maps and 
oplax natural with respect to loose ones (since $G$ takes tight maps
to strict $T_\loose$-morphisms and loose maps to lax $T_\loose$-morphisms).
In other words, the straight arrows are strict morphisms of
$M$-algebras, while the wriggly arrows are
colax morphisms of $M$-algebras.

Now as remarked in \S\ref{sec:descobj}, $c$-rigged $c$-descent
objects are created, and in particular reflected, by the forgetful
\sF-functor $M\bbAlg_c\to[\sDt,\Cat]$ (viewing $[\sDt,\Cat]$ as a
chordate \sF-category.)  Thus  we have a $c$-rigged $c$-descent object 
\begin{equation}
  \footnotesize\xymatrix@C=1.5pc{ T_\loose\Alg_l(\bA,G_\loose -) \ar@{.>}[r]
    & \sKl(A,UG_\loose -) \ar@{~>}@<2ex>[r] \ar@<-2ex>[r] &
    \sKl(TA,UG_\loose -) \ar[l] \ar@<-2ex>[r] \ar@{~>}@<2ex>[r] \ar[r]
    & \sKl(T^2A,UG_\loose -) }\label{eq:malgc-descent}
\end{equation}
in $M\bbAlg_c$, which is to say a $c$-descent object in $M\Alg_c$ for which 
the projection $T_\loose\Alg_l(\bA,G_\loose-)\to\sKl(A,UG_\loose-)$ is strict and
detects strictness.

{\bf Step 1: $\QF \Phi$-limits lift. }
Here is a rough sketch of this step.  By \autoref{thm:qr-adjt}, the
limit $\{\QF \Phi,U G\}$ is also the limit $\{\Phi, \RF (U G)\}$.  But
since $\RF$ takes colax transformations to strict ones, $\RF(U G)$
lifts to a diagram ``$\RF(G)$'' of \emph{strict} $T$-morphisms, and so
the limit $\{\Phi, \RF (U G)\}$ lifts to $T\bbAlg_l$.

We then show that this limit also
has the universal property of $\{\QF \Phi _\loose ,G_\loose\}$ by considering 
Figure~\ref{fig:descent}. We are
to prove that the two objects at the top are isomorphic. The strategy for this is 
to show that the vertical columns exhibit the objects at the top as descent 
objects, and therefore deduce the invertibility of the top horizontal map from
the invertibility of the other horizontal maps.
\begin{figure}[h]
\centering
\xymatrix@C=2pc{
[\sDl,\Cat](\QF \Phi_\loose, T_\loose\Alg_l(\bA,G_\loose-)) \ar[d] \ar@{.>}[r]^-{\cong} & 
[\sDl,\Cat](\Phi_\loose, T_\loose\Alg_l(\bA,\RF (G)_\loose-)) \ar[d]\\
[\sDl,\Cat](\QF \Phi_\loose,\sKl(A,UG_\loose-)) \ar@<-2ex>[d] \ar@{~>}@<2ex>[d] 
\ar[r]^-{\cong} &  
[\sDl,\Cat](\Phi_\loose,\sKl(A,\RF (UG)_\loose-)) \ar@<-2ex>[d] \ar@{~>}@<2ex>[d] \\
[\sDl,\Cat](\QF \Phi_\loose,\sKl(TA,UG_\loose-)) \ar@<-2ex>[d] \ar@{~>}@<2ex>[d] 
\ar[d] \ar[u] \ar[r]^-{\cong} & 
[\sDl,\Cat](\Phi_\loose,\sKl(TA,\RF (UG)_\loose-)) \ar@<-2ex>[d] \ar@{~>}@<2ex>[d] 
\ar[d] \ar[u]\\
[\sDl,\Cat](\QF \Phi_\loose,\sKl(T^2A,UG_\loose-)) \ar[r]^-{\cong} &
[\sDl,\Cat](\Phi_\loose,\sKl(T^2A,\RF (UG)_\loose-))
 }
\caption{Two descent objects\label{fig:descent}}
\end{figure}

We now turn to the details, including an explanation of the arrows in
Figure~\ref{fig:descent}.   Write $L'=\{\QF \Phi,UG\}$, and
$\ell'\colon TL'\to L'$ for the corresponding structure map
constructed as in \autoref{prop:alg-structure-on-L}.  Since \bbK is
complete, \autoref{thm:qr-adjt} implies that $L'$ is also the limit
$\{\Phi,\RF (UG)\}$ in \bbK, and hence also the 2-categorical
limit $\{ \Phi_\loose, \RF (UG)_\loose \}$ in \sKl.

Now recall from \autoref{prop:alg-structure-on-L}
that we have a distributive law $k:[\bbD,T]\RF \to \RF [\bbD,T]$,
according to which $T$ extends to a monad $\bbOplax(\bbD,T)$ on the
Kleisli category $\bbOplax(\bbD,\bbK)$ of $\RF $.  It also follows that
the right adjoint $\RF \colon \bbOplax(\bbD,\bbK) \to [\bbD,\bbK]$ is a
lax monad morphism from $\bbOplax(\bbD,T)$ to $[\bbD,T]$.  Therefore,
since $UG$ is a $\bbOplax(\bbD,T)$-algebra, $\RF (UG)$ is a
$[\bbD,T]$-algebra, with structure map
\[ [\bbD,T]\RF (UG) \xrightarrow{k} \RF (TUG) \xrightarrow{\RF (g)} \RF (UG). \]
where $g:TUG_\loose\to UG_\loose$ has components given by the structure maps
of the $T$-algebras $GD$. 
Since $[\bbD,T]\bbAlg_s \cong [\bbD, T\bbAlg_s]$, we have a
functor $\RF (G)\colon \bbD \to T\bbAlg_s$.

In particular, $\RF (G)_\loose$ is a functor $\sDl \to
(T\bbAlg_s)_\loose$.  But $(T\bbAlg_s)_\loose$ is the full
sub-2-category of $T_\loose \Alg_s$ consisting of those $T_\loose$-algebras
with tight structure map, and we know that all
2-categorical limits lift to 2-categories of algebras and strict
morphisms, and that these limits are preserved by the functor into the
2-category of weak morphisms.

Therefore,  the object $L'$ acquires a $T_\loose$-algebra
structure which makes it into the limit $\{\Phi_\loose, \RF
(G)_\loose\}$ in $T_\loose \Alg_s$, hence also in $T_\loose \Alg_l$.
Tracing through the definition of this $T_\loose$-algebra structure,
we find that it is equal to $\ell'$ as defined above, and in
particular is tight.  Thus, the $T$-algebra $(L',\ell')$ has the
universal property of $\{\Phi_\loose, \RF (G)_\loose\}$ in $T_\loose
\Alg_l$, hence also in its full sub-2-category $(T
\bbAlg_l)_\loose$.  It thus remains only to identify this
universal property with that of the desired limit $\{\QF \Phi_\loose,
G_\loose\}$.

Now the representable $M\Alg_c(\Phi_\loose ,-)$ preserves any existing limits, and in 
particular preserves the descent object~\eqref{eq:malgc-descent}. But for any \Cat-weight 
$\Psi$, we have $M\Alg_c(\Phi_\loose,\Psi)\iso [\sDl,\Cat](\QF \Phi_\loose,\Psi)$, and so the
left-hand column of Figure~\ref{fig:descent} is a descent object  in \Cat.

We have continued to use wriggly arrows in this column, although this 
has no formal meaning in the 2-category \Cat, in order to draw attention to 
the fact that the definition of these arrows is less straightforward than the 
straight ones, which are all induced by composition with some arrow in \sKl.
For example, 
the wriggly arrow at the middle level of the left-hand column is defined (on objects)
like this. Given a 2-natural $x:\QF \Phi_\loose\to\sKl(A,UG_\loose-)$ form the 
lax-natural composite 
$$\xymatrix @C=2pc {
\QF \Phi_\loose \ar[r]^-{x} & 
\sKl(A,UG_\loose-) \ar[r]^-{T} &
\sKl(TA,TUG_\loose-) \ar@{~>}[rr]^-{\sKl(TA,g-)} && 
\sKl(TA,UG_\loose-) }$$
where $g$ is as above.
This corresponds to a unique 2-natural
$y:\QF ^2\Phi_\loose\to\sKl(TA,UG_\loose-)$, which we compose with the 
comultiplication $d:\QF \Phi_\loose\to \QF ^2\Phi_\loose$ to obtain the
map $yd:\QF \Phi_\loose\to \sKl(TA,UG_\loose-)$ which is the image of our $x$. The 
case of the lower wriggly map in the left-hand column of Figure~\ref{fig:descent}
is similar.

The solid horizontal isomorphisms in Figure \ref{fig:descent}
are instances of the adjointness of $\QF$ and $\RF$, as in \autoref{thm:qr-adjt}. 
The straight maps on the
right-hand side are again just composition, and these obviously
commute with the horizontal isomorphisms.  Moreover, tracing through
the definitions, we see that in order for the wriggly arrows to
commute with the horizontal isomorphisms, the middle wriggly map on the right-hand
side must be given by composing with the map
\[
\sKl(A, \RF (UG)_\loose)  \xrightarrow{T} 
\sKl(TA, [\bbD,T]\RF (UG)_\loose) \xrightarrow{\sKl(TA,\RF (g).k)}
\sKl(TA, \RF (UG)_\loose)
\]
where $\RF (g).k$ is the (strictly 2-natural) $[\bbD,T]$-algebra
structure of $\RF (G)$, as above.  The lower wriggly map is similar.  Therefore,
the $c$-descent object of the right-hand column is exactly
$[\sDl,\Cat](\Phi_\loose, T_\loose\Alg_l(\bA,\RF (G)_\loose-))$, as
shown, and so we have the dotted isomorphism across the top.  But this
says exactly that to give a limit $\{\Phi_\loose, \RF (G)_\loose\}$ is
the same as to give a limit $\{\QF  \Phi_\loose, G_\loose \}$.  Thus
$(L',\ell')$ is the latter limit, as desired.

{\bf Step 2: $\Phi_\loose$-limits lift.} 
Since  $T_\loose\Alg_l(\bA,G_\loose-)\to\sKl(A,UG_\loose-)$, as a morphism
in $M\Alg_c$,  is strict
and detects strictness, the following square 
$$\xymatrix{
[\sDl,\Cat](\Phi_\loose,T_\loose\Alg_l(\bA,G_\loose-)) \ar[r] \ar[d] &
M\Alg_c(\Phi_\loose,T_\loose\Alg_l(\bA,G_\loose-))  \ar[d] \\
[\sDl,\Cat](\Phi_\loose,\sKl(A,UG_\loose-)) \ar[r]  &
M\Alg_c(\Phi_\loose,\sKl(A,UG_\loose-)) }$$
is a pullback in \Cat. 
We can also write this, using the $c$-morphism classifier 
$Q^\bbD$ for $M\Alg_c$, as a pullback
\begin{equation}\label{ref:rhspbk}
\xymatrix{
[\sDl,\Cat](\Phi_\loose,T_\loose\Alg_l(\bA,G_\loose-)) \ar[r]^-{q^*} \ar[d] &
[\sDl,\Cat](Q^\bbD \Phi_\loose,T_\loose\Alg_l(\bA,G_\loose-)) \ar[d] \\
[\sDl,\Cat](\Phi_\loose,\sKl(A,UG_\loose-)) \ar[r]_-{q^*}  &
[\sDl,\Cat](Q^\bbD \Phi_\loose,\sKl(A,UG_\loose-)) }
\end{equation}
in which the horizontal arrows are given by composition with $q\colon \QF \Phi\to\Phi$.
As above, write $\bLL=(L',\ell')$ for $\{\QF \Phi_\loose,G_\loose\}$, and $\bL=(L,\ell)$ for 
the $T$-algebra which we 
are to show has the universal property of $\{\Phi_\loose,G_\loose\}$.
The map $q:\QF \Phi\to\Phi$ induces a morphism $q^*=\{q,UG_\loose\}:L\to L'$. 
In the diagram
\begin{equation}\label{ref:lhspbk}
\xymatrix{
T_\loose\Alg_l(\bA,\bL) \ar@{.>}[r] \ar[d] & 
T_\loose\Alg_l(\bA,\bLL) \ar[d] \\
\sKl(A,L) \ar[r]_-{\sKl(A,q^*)} & \sKl(A,L') }
\end{equation}
all vertices except the top left are known to be isomorphic to the corresonding
vertices in the previous square~\eqref{ref:rhspbk}, and these isomorphisms are 
compatible with the edges of the square. We need to show that the top left
vertices in the two squares are also isomorphic;
this is equivalent to filling in the dotted arrow in the square~\eqref{ref:lhspbk} 
 in such a way as to give a pullback.

The structure map $s:\Phi\to \QF \Phi$ induces a map $s^*=\{s,UG_\loose\}:L'\to L$.
By the description of the $T$-actions $T\{\Phi,UG\}\to\{\Phi,UG\}$  given
at the end of the proof of Proposition~\ref{prop:alg-structure-on-L}, it is 
clear that these actions are strictly natural with respect to strict morphisms of 
$\QF $-coalgebras, and so in particular with respect to $s:\Phi\to \QF \Phi$. Thus
$s^*$ is a strict morphism $\bLL\to\bL$ of $T$-algebras. Since $q\dashv s$
with identity counit, $s^*:L'\to L$ is left adjoint to $q^*:L\to L'$ with 
identity counit. The unit $\eta:1\to q^*s^*$ induces a 2-cell
$$\xymatrix@C=1.5pc{ 
\ell'.Tq^* \ar[rr]^-{\eta.\ell'.Tq^*} && q^*.s^*.\ell'.Tq^* \ar@{=}[r] &
q^*.\ell.Ts^*.Tq^* \ar@{=}[r] & q^*.\ell }$$
which makes $q^*$ into a lax $T_\loose$-morphism $(q^*,\overline{q}^*):\bL\to\bLL$.
Then composition with $(q^*,\overline{q}^*)$ gives the dotted arrow making the 
square~\eqref{ref:lhspbk} commute; it remains
to show that the square is a pullback. Since the horizontal arrows are
both given by composition with a morphism having a left adjoint
with identity counit, they are both full embeddings.
Therefore, the square being a pullback amounts to saying  that a lax $T$-morphism 
$(f,\bar{f}):\bA\to\bLL$ factorizes through $(q^*,\overline{q}^*):\bL\to\bLL$ 
provided that $f$ factorizes through $q^*$. But if $f=q^*g$, then $g=s^*f$,
and we now define $(g,\gbar)$ to be the composite $s^*(f,\fbar)$. 
We must show that $(q^*,\overline{q}^*)s^*(f,\bar{f})=(f,\fbar)$. 

Now $q^*.s^*.f=f$ by assumption. The 2-cell part of $(q^*,\overline{q}^*)s^*(f,\fbar)$
is given by the top path around the square
$$\xymatrix{
\ell'.Tq^*.Ts^*.Tq^*.Tg \ar[r]^-{\eta.1} \ar@{=}[d] & 
q^*.s^*.\ell'.Tq^*.Ts^*.Tq^*.Tg \ar@{=}[d] \\
\ell'.Tq^*.Tg \ar[dd]_{\fbar} & q^*.\ell.Ts^*.Tq^*.Tg \ar@{=}[d] \\
& q^*.s^*.\ell'. Tq^*.Tg \ar[d]^{q^*.s^*.\fbar} \\
q^*.g.a \ar[r]_-{\eta.q^*.g.a} & q^*.s^*.q^*.g.a}$$
which is equal to the bottom path. But $\eta.q^*.g.a$ is the identity 
by one of the triangle equations, so this bottom path is just $\fbar$.
\endproof

\begin{thm}\label{thm:rigged-lift}
  If $\Phi$ is a $w$-rigged \sF-weight, then for  any \sF-monad $T$ on an
  \sF-category \bbK, the forgetful functor
  $U_w\maps T\bbAlg_w \to \bbK$ creates $\Phi$-weighted limits.
\end{thm}
\begin{proof}
Once again, we treat only the case $w=l$ of lax $T$-morphisms.
By~\ref{prop:reduction} we may suppose that \bbK is complete. 

Let $G:\bbD\to T\bbAlg_w$ be the diagram of which we wish to calculate
the limit $\{\Phi,G\}$. We know that the limit $L=\{\Phi_\loose,UG_\loose\}$ in 
\sKl lifts to a limit $\bL=(L,\ell)=\{\Phi_\loose,G_\loose\}$ in $(T\bbAlg_w)_\loose$;
we want to show that it is also the limit $\{\Phi,G\}$ in $T\bbAlg_w$.
This amounts to proving that the family of projections $p_{\loose,\phi(a)}:\bL\to GD$,
where $D\in\bbD$ and $a\in\Phi_\tight D$, are tight and jointly detect tightness
in $T\bbAlg_w$. Tightness in $T\bbAlg_w$ has two aspects: strictness as
a $T$-morphism, and tightness at the level of the underlying morphism in \bbK.

We know, by the tight part of the universal property of $\{\Phi,UG\}$, that 
the $Up_{\loose,\phi(a)}:L\to UGD$ are tight and jointly 
detect tightness in \bbK, so we only need to worry about strictness. The
fact that the $p_{\loose,\phi(a)}$ are strict follows from commutativity of the 
diagrams~\eqref{eq:defn-zeta} and~\eqref{eq:defn-l} on 
page~\pageref{eq:defn-l}. It remains to show
that they jointly detect strictness. 

Suppose then that $\bA=(A,a)$ is a $T$-algebra and $\beff=(f,\bar{f}):\bA\to\bL$
a lax $T$-morphism, with $f$ tight in \bbK, whose composite with each $p_{\loose,\phi(a)}$ is
strict. Then the composite 2-cell
$$\xymatrix{
\Phi_\tight \ar[r]^{\phi} & 
\Phi_\loose J \ar[r]^-{p_\loose J} &
\sKl(L,UG_\loose J-) \ar@/^2ex/[rr]^{\sKl(\ell.Tf,1)}_(0.4){~}="1" 
\ar@/_2ex/[rr]^(0.4){~}="2"_{\sKl(fa,1)} &&
\sKl(TA,UG_\loose J-) 
\ar@{=>}"1";"2"^{\sKl(\bar{f},1)} 
}$$
is an identity, and so by adjointness the composite 2-cell
$$\xymatrix{
\Lan_J\Phi_\tight \ar[r]^-{\overline{\phi}} & 
\Phi_\loose  \ar[r]^-{p_\loose} &
\sKl(L,UG_\loose-) \ar@/^2ex/[rr]^{\sKl(\ell.Tf,1)}_(0.4){~}="1" 
\ar@/_2ex/[rr]^(0.4){~}="2"_{\sKl(fa,1)} &&
\sKl(TA,UG_\loose-) 
\ar@{=>}"1";"2"^{\sKl(\bar{f},1)} 
}$$
is an identity. But $\overline{\phi}$ is pointwise surjective on objects, 
so by \autoref{thm:so-char}, the 2-cell
$$\xymatrix{
\Phi_\loose  \ar[r]^-{p_\loose} &
\sKl(L,UG_\loose-) \ar@/^2ex/[rr]^{\sKl(\ell.Tf,1)}_(0.4){~}="1" 
\ar@/_2ex/[rr]^(0.4){~}="2"_{\sKl(fa,1)} &&
\sKl(TA,UG_\loose-) 
\ar@{=>}"1";"2"^{\sKl(\bar{f},1)} 
}$$
is an identity, and now finally by the universal property of the limit
$L=\{\Phi_\loose,UG_\loose\}$ it follows that $\bar{f}$ is an identity. 
\end{proof}

\subsection{All limits that lift are rigged}
\label{sec:lifting-limits}

In the previous subsection we showed that all rigged limits lift; here
we prove the converse.

\begin{lem}\label{thm:yon-qcoalg}
  Every representable $Y_D = \bbD(D,-)$ is $w$-rigged, for any $w$.
\end{lem}
\begin{proof}
  We define the \QF-coalgebra structure of $Y_D = \bbD(D,-)$ to be the
  map $s\colon Y_D \to \QF Y_D$ which corresponds under the Yoneda lemma to
  the element $p(1_D) \in \QF Y_D(D)$.  The coassociativity and counit 
axioms both assert the equality of maps with domain $Y_D$, so can 
be verified using the Yoneda lemma with the  calculations
$$  q.s.1_D = q.p.1_D = 1_D $$
for the counit and 
$$  Qs.s.1_D = Qs.p.1_D = p.s.1_D = p.p.1_D = Qp.p.1_D = d.p.1_D= d.s.1_D$$
for coassociativity. Finally, observe that 
$\Lan_J\sDt(D,-)\cong\sDl(D,-)$, and that this isomorphism is precisely
the $\phbar$ which must be pointwise surjective on objects in order
that $Y_D$ be $w$-rigged. 
\end{proof}

\begin{thm}\label{thm:lifts-rigged}
  Suppose that $\Phi\maps \bbD\to\bbF$ is an \sF-weight such that
  $\Phi$-weighted limits lift to $T\bbAlg_w$ for any \sF-monad $T$.
  Then $\Phi$ is $w$-rigged.
\end{thm}
\begin{proof}
  By duality, if $\Phi$-weighted limits lift to \sF-categories $T\bbAlg_w$ of
  algebras for \sF-monads $T$, then $\Phi$-weighted \emph{colimits} must
  lift to \sF-categories $W\bbCoalg_\wbar$ of \emph{coalgebras} for \sF-\emph{comonads} $W$.
  (Note the reversal of sense from $w$ to $\wbar$, as remarked in \S\ref{sect:comonad}.)
  But by general enriched category theory, $\Phi$ itself is the
  $\Phi$-weighted colimit in $[\bbD,\bbF]$ of the Yoneda embedding
  $Y\maps \bbD\op\to [\bbD,\bbF]$, i.e.\ $\Phi = \Phi\star Y$.

   Since the forgetful functor $\QF_\wbar \bbCoalg_\wbar \to [\bbD,\bbF]$ is
  \sF-fully-faithful on rigged weights, by \autoref{thm:yon-qcoalg}
  the Yoneda embedding lifts uniquely to a functor
  $\tilde{Y}\maps \bbD\op \to \QF_\wbar\bbCoalg_\wbar$.
  By assumption, the $\Phi$-weighted colimit of $Y$,
  namely $\Phi$ itself, thereby acquires a $\QF_\wbar$-coalgebra structure.
  Thus it remains to show that $\phbar\maps \Lan_J \Phi_\tight \to
  \Phi_\loose$ is pointwise surjective on objects; we will also do
  this by lifting $\Phi\star Y$ to an appropriate category of
  coalgebras.

  For any 2-functor $\Psi\maps \sD_\loose \to \Cat$, define
  $\Psitil\maps \bbD\to\bbF$ by $\Psitil_\loose=\Psi$,
  $\Psitil_\tight=\Psi J$.  Then for any $\Phi\maps \bbD\to\bbF$, any
  2-natural transformation $\Phi_\loose\to\Psi$ extends uniquely to an
  \sF-natural transformation $\Phi\to\Psitil$.  That is,
  $\widetilde{(-)}$ is right adjoint to $(-)_\loose$.

  As usual, we write in the case $w=l$.  Let $W$
  be the \sF-comonad $(-)^{\mathbf{N}}$ on $[\bbD,\bbF]$; that is,
  cotensoring with the discrete monoid of natural numbers.  Then a
  $W$-coalgebra is an \sF-weight $\Phi$ equipped with an endomorphism
  $e_\Phi\maps \Phi\to\Phi$.  A tight arrow in $W\bbCoalg_c$ is an
  \sF-natural transformation commuting strictly with the
  endomorphisms, while a loose one is a map $f\maps \Phi_\loose\to
  \Upsilon_\loose$ equipped with a modification
  \[\xymatrix{\Phi_\loose \ar[r]^f\ar[d]_{(e_\Phi)_\loose} & \Upsilon_\loose \ar[d]^{(e_\Upsilon)_\loose}\\
    \Phi_\loose\ar[r]_f & \Upsilon_\loose \ultwocell\omit{\fbar}.}\]
  We have an evident \sF-functor $i\maps [\bbD,\bbF]\to W\bbCoalg_c$ that
  equips an \sF-weight with its identity endomorphism.

  Now, as before we have $\Phi = \Phi\star Y$ in $[\bbD,\bbF]$,
  with colimiting cocone
  $c\maps \Phi\to [\bbD,\bbF](Y,\Phi)$.  We also have the composite
  functor $iY\maps \bbD\op\to W\bbCoalg_c$ which lifts $Y$.
  Therefore, the forgetful \sF-functor $W\bbCoalg_c \to [\bbD,\bbF]$
  creates a colimit of $iY$, which must be given by a $W$-coalgebra
  structure $e_\Phi$ on $\Phi$.

  However, we also have the $\Phi$-weighted cocone $ic\maps \Phi \to
  W\bbCoalg_c(iY,i\Phi)$, so there is a unique induced tight $W$-map
  $h\maps (\Phi,e_\Phi) \to (\Phi,1) = i\Phi$.  Since both the
  colimiting cocone and the cocone $ic$ project to $c$ in $[\bbD,\bbF]$,
  the map $h$ must project to the identity of $\Phi$, and hence $e_\Phi
  = 1_\Phi$ as well.  Thus $ic$ is actually itself colimiting, i.e.\
  $i\Phi \iso \Phi \star i Y$.

  Now let $\xymatrix{\Phi_\loose \rtwocell^f_g{\beta} & \Psi}$ be a
  modification such that the composite
  \[\xymatrix{\Phi_\tight \rrtwocell<5>^{fJ. \varphi}_{gJ.\varphi}{\quad\beta J
      . \varphi} && \Psi J}
  \]
  is an identity, where $\varphi\maps \Phi_\tight \to \Phi_\loose J$ is the
  structure map of $\Phi$.  We equip $\Psitil\times\Psitil$ with the
  endomorphism $(\pi_2,\pi_2)$ sending $(x,y)$ to $(y,y)$, making it a
  $W$-coalgebra.  Then $\beta$ defines a loose arrow $i\Phi \lto
  \Psitil\times\Psitil$ in $W\bbCoalg_c$:
  \[\xymatrix{\Phi_\loose \ar[r]^<>(.5){(f,g)}\ar[d]_{1} & \Psi\times\Psi \ar[d]^{(\pi_2,\pi_2)}\\
    \Phi_\loose\ar[r]_<>(.5){(f,g)} & \Psi\times\Psi \ultwocell\omit{(\beta,1)\quad}}\]
  (the top-right composite being $(g,g)$ and the bottom-right composite
  being $(f,g)$).  We claim that in fact, this is a tight arrow, and
  therefore $\beta$ itself is an identity.

  Since $i\Phi = \Phi\star iY$ in $W\bbCoalg_c$, and $(f,g)$ does extend
  to a tight map $\Phi\to \Psitil\times\Psitil$ in $[\bbD,\bbF]$, to
  show that $\beta$ is tight in $W\bbCoalg_c$ it suffices to show that
  for every tight coprojection $iYD \xrightarrow{k} i\Phi$, the composite
  $iYD \xrightarrow{k} i\Phi \lto \Psitil\times\Psitil$ is tight.  This
  amounts to saying that the composite
  \[\xymatrix{(YD)_\loose \rrtwocell<5>^{f.k_\loose}_{g.k_\loose}{\quad\beta .k_\loose} && \Psi}
  \]
  is an identity.  But by assumption and since $k$ is tight, the
  composite
  \[\xymatrix@C=4pc{(YD)_\tight \rrtwocell<5>^{(f.k_\loose)J.y}_{(g.k_\loose)J.y}{\qquad\quad(\beta .k_\loose) J.y} && \Psi}
  \]
  is the identity, and therefore so is
  \[\xymatrix@C=3pc{\Lan_y (YD)_\tight \rrtwocell<5>^{f.k_\loose.\ybar}_{g.k_\loose.\ybar}{\qquad\beta .k_\loose.\ybar} && \Psi}
  \]
  where $y\maps (YD)_\tight \to (YD)_\loose J$ is the structure map of $YD$ and
  $\ybar\maps \Lan_J(YD)_\tight \to (YD)_\loose$ is its adjunct.  But since $YD$
  is representable, \ybar\ is an isomorphism, so this implies that
  $\beta.k_\loose$ is also an identity, as desired.

  We have proven that given any modification $\xymatrix{\Phi_\loose
    \rtwocell^f_g{\beta} & \Psi}$, if $\beta J.\varphi$, or
  equivalently $\beta.\phbar$, is an identity, then $\beta$ is an
  identity.  By \autoref{thm:so-char}, this implies that $\phbar\maps
  \Lan_J \Phi_\tight \to \Phi_\loose$ must be pointwise surjective on
  objects.
\end{proof}

Combining the results of this section and the previous one, we finally obtain
our characterization theorem:

\begin{thm}
  For an \sF-weight $\Phi$, the following are equivalent.
  \begin{enumerate}
  \item $\Phi$ is $w$-rigged, as in \autoref{def:rigged}.
  \item For any \sF-monad $T$ on an \sF-category \bbK, the functor $U_w\colon T\bbAlg_w \to \bbK$ creates $\Phi$-weighted limits.
  \item For any 2-monad $T$ on a 2-category \sK, the functor $U_w\colon T\bbAlg_w \to \bbK$ (where $\bbK$ denotes the chordate \sF-category on \sK) creates $\Phi$-weighted limits.
  \end{enumerate}
\end{thm}

\subsection{Rigged colimits}
\label{sec:colimits}

We end this section by briefly considering colimits in
categories of algebras, which are generally more subtle than
limits.  Even in the case of ordinary categories, colimits are
not in general created by monadic functors.  One thing one can say is
that if a monad $T$ preserves colimits with a given weight, then the
category of $T$-algebras has colimits of that sort created by the
forgetful functor.  We now show that rigged \sF-weights satisfy an
analogous property for categories of algebras and weak morphisms.  The
analogous result for PIE-colimits in the 2-categorical context has
been proven independently by John Bourke.

\begin{thm}
  Let $\Phi$ be $\wbar$-rigged, and let $T$ be an \sF-monad on an
  \sF-category \bbK such that \bbK has, and $T$ preserves,
  $\Phi$-weighted colimits.  Then the forgetful functor $U_w\colon
  T\bbAlg_w \to \bbK$ creates $\Phi$-weighted colimits.
\end{thm}

\begin{proof}
   First of all, we observe that just as in
  \autoref{prop:reduction}, we may assume that \bbK is small; for
  otherwise we can pick a small full subcategory of it, closed under
  $\Phi$-weighted colimits and the action of $T$, and containing the
  image of $G$ and any other  given object.

  Now, if \bbK is small, let $\widehat{\bbK} \subset
  [\bbK\op,\bbF]$ be the subcategory consisting of those presheaves on
  $\bbK$ which preserve $\Phi$-weighted limits (i.e.\ which take 
$\Phi$-weighted colimits in \bbK to limits in $\bbF$).  Since \bbK has all
  $\Phi$-weighted colimits, by~\cite[6.17]{kelly:enriched} $T$
  extends to an \sF-monad $\widehat{T}$ on $\widehat{\bbK}$ which has
  a right adjoint $\widehat{T}^*$.  Therefore, as in
  \S\ref{sect:comonad}, $\widehat{T}^*$ becomes an \sF-comonad whose
  coalgebras are the same as $\widehat{T}$-algebras.

  However, since $\Phi$ is $\wbar$-rigged, $\Phi$-weighted colimits
  lift from $\widehat{\bbK}$ to $\widehat{T}^*\bbCoalg_w$, hence
  also to $\widehat{T}\bbAlg_w$.  Moreover, the embedding
  $\bbK\hookrightarrow \widehat{\bbK}$ preserves $\Phi$-weighted
  colimits.  Thus, any $G\colon \bbD \to T\bbAlg_w$ has a colimit
  $\Phi*G$ in \bbK, which remains a colimit in $\widehat{\bbK}$ and
  thus lifts to $\widehat{T}\bbAlg_w$.  But the underlying presheaf of
  this colimit $\widehat{T}$-algebra is representable, hence it is a
  $T$-algebra and thus a $\Phi$-weighted colimit of $G$.
\end{proof}

\section{On rigged weights}
\label{sec:rigged}

Our goal in this section is to analyze the notion of rigged weight a
little further, to clarify the relationship between the two parts of
the definition and the connection to 2-categorical weights such as
PIE-weights.

\subsection{The structure of $\QF$-coalgebras}
\label{sec:struct-q-coalg}

We begin by unpacking the notion of $\QF_{\wbar}$-coalgebra a little.
Let \bbD be an \sF-category, and $Q_\wbar \coloneqq Q^\bbD_\wbar$ the
$\wbar$-transformation classifier for the monad on $[\sDt,\Cat]$ whose
category of algebras is $[\sDl,\Cat]$.  Recall from
\S\ref{sec:weak-transf-class} that $\QF_{\wbar}$ can be constructed as
$(\QF_\wbar \Phi)_\tight = \Phi_\tight$ and $(\QF_\wbar \Phi)_\loose =
Q_\wbar (\Phi_\loose)$, with the structure map being the composite
\[ \xymatrix{ \Phi_\tight \ar[r]^-{\phi} &
  \Phi_\loose J \ar[r]^-{p} &
  Q_\wbar (\Phi_\loose) J.}
\]
Therefore, the coaction $s\colon \Phi \to \QF_\wbar \Phi$ of a
$\QF_\wbar$-coalgebra must be of the form
\[\xymatrix{
\Phi \ar[d] & 
\Phi_\tight \ar[r]^-{\phi} \ar@{=}[d] & \Phi_\loose J \ar[d]^{sJ} \\
\QF_\wbar\Phi & \Phi_\tight \ar[r]_-{p.\phi} & (Q_\wbar\Phi_\loose)J }
\]
It follows that $\Phi$ is a $\QF_\wbar$-coalgebra if and only if
\begin{enumerate}
\item $\Phi_\loose$ is a $Q_\wbar$-coalgebra, with coaction $s\colon
  \Phi_\loose \to Q_\wbar \Phi_\loose$, and
\item $sJ.\phi = p.\phi$.
\end{enumerate}

We can refine the second of these conditions a little further.
Suppose that $w=l$, and recall that $Q_c$ is colax-idempotent.
Therefore, if $s\colon \Phi_\loose\to Q_c\Phi_\loose$ makes
$\Phi_\loose$ into a $Q_c$-coalgebra, then $s$ is right adjoint to
$q\colon Q_c\Phi_\loose\to\Phi_\loose$ with identity counit, and hence
$sJ$ is right adjoint to $qJ$ with identity counit.  On the other
hand, $p\colon \Phi_\loose J\to Q_c\Phi_\loose J$ is left adjoint to
$qJ$ with identity unit, so we get a string of adjunctions $p\dashv
qJ\dashv sJ$.  We write $\alpha\colon 1\to sq$ for the unit of the
adjunction $q\dashv s$ and $\beta\colon p.qJ\to 1$ for the counit of
the adjunction $p\dashv qJ$.  Then the diagram
\[\xymatrix{
p \ar[r]^-{\alpha J.p} \ar@{=}[d] & sJ.qJ.p \ar@{=}[d] \\
p.qJ.sJ \ar[r]_-{\beta.sJ} & sJ }
\]
commutes, and we write $\tau\colon p\to sJ$ for the common value.

In the case $w=c$ there is an analogous $\tau\colon sJ\to p$, while
for $w=p$ the 2-cell is invertible.

\begin{lem}\label{thm:eqisid}
  If $\Psi$ is a $Q_{\wbar}$-coalgebra, then a morphism
  $\upsilon\colon \Upsilon \to \Psi J$ satisfies $sJ.\upsilon =
  p.\upsilon$ if and only if $\tau.\upsilon$ is an identity.
\end{lem}
\begin{proof}
  If $\tau.\upsilon\colon p.\upsilon \to sJ.\upsilon$ is an identity, its
  source and target must be equal.  On the other hand, by definition
  $\tau = \beta. sJ$, so if $sJ.\upsilon=p.\upsilon$ then
  \[ \tau.\upsilon = \beta.sJ.\upsilon = \beta.p.\upsilon, \] and
  $\beta.p$ is an identity by one of the triangle identities for the
  adjunction between $p$ and $qJ$.
\end{proof}

In particular, this applies to $\phi\colon \Phi_\tight \to \Phi_\loose
J$ whenever $\Phi$ is a $\QF_\wbar$-coalgebra.

\subsection{Canonical riggings}
\label{sec:canonical}

It is natural to ask, given a coalgebra for the comonad given by the
2-categorical relative $w$-morphism classifier
$Q^\bbD_\wbar$, how can we extend it to a $w$-rigged \sF-weight?  In all
the examples in \S\ref{sec:eg-wgts}, there was an obvious
``canonical'' choice of which projections to make tight.  The
following proposition says that this situation is generic.

\begin{prop}\label{thm:canonical}
  If $\Psi$ is a $Q_{\wbar}$-coalgebra, then the category of all
  $\QF_\wbar$-coalgebras $\Phi$ with $\Phi_\loose = \Psi$ (as
  $Q_\wbar$-coalgebras) is a preorder with a greatest element.
  Moreover, if any such $\Phi$ is $w$-rigged, so is the greatest one.
\end{prop}
\begin{proof}
  We saw in \S\ref{sec:struct-q-coalg} that to give a $\QF_\wbar$-coalgebra
  $\Phi$ with $\Phi_\loose = \Psi$ is precisely to give a pointwise
  full embedding $\phi\colon \Phi_\tight \to \Psi J$ such that
  $sJ.\phi = p.\phi$, or equivalently such that $\tau.\phi$ is an
  identity.  Since full embeddings are monic and fully faithful, the
  category of such (for fixed $\Psi$) is a preorder.  Its greatest
  element is the identifier of $\tau$ (note that any identifier is a
  full embedding).

  Finally, if $\Phi$ is $w$-rigged with $\Phi_\loose = \Psi$, and
  $\psi\colon \Psi_\tight \to \Psi J$ is the identifier of $\tau$,
  then $\phi\colon \Phi_\tight \to \Phi_\loose = \Psi J$ factors
  through $\psi$ via some $k\colon \Phi_\tight \to \Psi_\tight$.  Thus
  the composite
  \[ \xymatrix@C=3pc{ \Lan_J\Phi_\tight \ar[r]^{\Lan_J(k)} &
    \Lan_J\Psi_\tight \ar[r]^-{\overline{\psi}} &
    \Psi }
  \]
  is the pointwise surjective on objects $\overline{\phi}$, so
  $\overline{\psi}$ must also be pointwise surjective on objects.
\end{proof}

Note that \autoref{thm:eqisid} implies that the identifier of $\tau$
is also the equalizer of $sJ$ and $p$.

The canonical rigging constructed in \autoref{thm:canonical} does,
however, depend on the choice of $Q_\wbar$-coalgebra structure on
$\Psi$.  Since $Q_\wbar$ is $\wbar$-idempotent, any two such coalgebra
structures are uniquely isomorphic, but they need not be identical.
Here is an example which admits two distinct $Q_\wbar$-coalgebra
structures, with correspondingly distinct canonical riggings.
 (This should be contrasted with the remark after \autoref{thm:rigged-strictmor} that an \sF-weight with a \emph{given} tight part can ``be rigged'' in at most one way.)

\begin{eg}\label{eg:two-qcoalg}
  Let $\sD_\loose$ have two objects $a$ and $b$, with two morphisms
  $r\colon b\to a$ and $i\colon a\to b$ such that $r i = 1_a$, and
  hence $f\coloneqq i r$ is idempotent.  Let $\Phi_\loose$ be constant
  at $\bbone$.  Since $a$ is initial in $\sD_\loose$ (or,
  equivalently, $\Phi_\loose$ is the representable $\sD_\loose(a,-)$),
  for any 2-categorical diagram $G\colon \sD_\loose \to \sK$, the
  object $G(a)$ is a $\Phi_\loose$-weighted limit of $G$.

  Let $\sD_\tight$ contain the identities together with the idempotent
  $f$.  Then $Q_w \Phi_\loose$ is constant at the free-living isomorphism
  $\bbtwo_{\iso}$, and so $\Phi_\loose$ has two distinct (but, of
  course, isomorphic) $Q_w$-coalgebra structures.  One of the
  corresponding identifiers has $\Phi_\tight(a)=\bbone$ and
  $\Phi_\tight(b)=\emptyset$, while the other has
  $\Phi_\tight(a)=\emptyset$ and $\Phi_\tight(b)=\bbone$.  Both
  resulting \sF-weights are rigged.
  
  More explicitly, suppose $T$ is a 2-monad on a 2-category \sK.  Then
  a \bbD-diagram in $T\bbAlg_w$ consists of a strict idempotent $f$ of
  a $T$-algebra $\bB$, together with a splitting of the underlying
  morphism $f$ in \sK, and a $T$-algebra structure \bA on the
  splitting making the section and retraction into \emph{weak}
  $T$-morphisms (though their composite, being the idempotent, is
  strict).

  On the one hand, clearly \bA is a limit of this diagram in
  $T\bbAlg_w$, and its identity projection to itself is strict and
  detects strictness.  But on the other hand, since $f$ is a strict
  $T$-morphism, we can give the underlying object $A$ a different
  $T$-algebra structure induced directly from \bB, in which case the
  section $i$ becomes strict and strictness-detecting.  These two
  $T$-algebra structures on $A$ are isomorphic in $T\Alg_w$, by an
  isomorphism whose 1-morphism part is the identity $1_A$, but they
  are generally not equal (i.e.\ this isomorphism is not generally a
  strict $T$-morphism).
\end{eg}

On the other hand, a 2-categorical $Q_\wbar$-coalgebra may have no
extension to a $w$-rigged \sF-weight at all.

\begin{eg}\label{eg:not-rigged}
  Let \bbD be locally discrete, with two objects $a$ and $b$, and with
  morphisms generated by two tight morphisms $r,s\colon
  a\rightrightarrows b$ and a loose morphism $g\colon b\lto a$,
  subject to $sg=1$ and $rgr=rgs$.  Then $f \coloneqq rg$ is a loose
  idempotent with $fr=fs$.  We can write out all the homsets
  explicitly as follows:
  \begin{alignat*}{2}
    \sD_\loose(a,a) &= \{ 1_a, gr, gs, gfr \} &\qquad
    \sD_\loose(b,b) &= \{ 1_b, f \}\\
    \sD_\loose(a,b) &= \{ r, s, fr \} &\qquad
    \sD_\loose(b,a) &= \{ g, gf \}
  \end{alignat*}
  The only tight morphisms are the identities and $r,s$.

  Let $\Phi_\loose$ be constant at \bbone and $\Phi_\tight$ be
  constant at $\emptyset$.  Then $\Phi_\loose J\colon \sD_\tight \to
  \Cat$ is the quotient of the representable $\sD_\tight(a,-)$ by the
  equivalence relation setting $[r]=[s]$.  Therefore, $\Lan_J
  (\Phi_\loose J)$ is the quotient of $\sD_\loose(a,-)$ by the
  equivalence relation generated by $[r]=[s]$, which implies
  $[gr]=[gs]$ but no more relations.

  Since for any $w$, $Q_w\Phi_\loose$ is a type of codescent object
  of $\Lan_J (\Phi_\loose J)$, it has the same objects as the latter.
  We can therefore define a morphism $\Phi_\loose \to
  Q_w\Phi_\loose$ which picks out $[gfr] = [gfs] \in
  Q_w\Phi_\loose(a)$ and $[fr] = [fs] \in Q_w\Phi_\loose(b)$.
  Verifying the coassociativity and counit axioms, and recalling that
  $\Phi_\tight\equiv \emptyset$ so that this extends to a morphism $\Phi \to
  \QF_w \Phi$, we see that $\Phi$ is a $\QF_w$-coalgebra for any $w$.

  Of course, with $\Phi_\tight\equiv \emptyset$ and $\Phi_\loose$
  nontrivial, $\Phi$ is not rigged.  In fact, however,
  $\Phi_\tight\equiv \emptyset$ is the \emph{only} $\Phi_\tight$ for
  which the above morphism $\Phi_\loose \to Q_w\Phi_\loose$ extends
  to an \sF-natural transformation.  This is because the inclusion
  $p\colon \Phi_\loose J \to (Q_w\Phi)_\loose J$ picks out instead
  $[1_a]$ and $[r]=[s]$, and so the identifier constructed above is
  itself empty.

  More concretely, a $\Phi$-weighted limit of a \bbD-shaped diagram is
  really just a splitting of the loose idempotent $f$.  Splitting of
  general loose idempotents is flexible but not rigged, but in this
  case the additional existence of the tight morphisms $r$ and $s$
  enables us to make the weight into a $\QF$-coalgebra, or equivalently
  to show that the limit of a \bbD-diagram in $T\bbAlg_w$ actually
  gets a strict $T$-algebra structure.  But neither of the
  projections is necessarily a strict $T$-morphism, and so the weight
  cannot be rigged.
\end{eg}

We can, however, identify conditions on the \sF-category \bbD which
ensure that any $Q_{\wbar}$-coalgebra can be rigged.  For a 2-category
\sC, we write $\sC_0$ for its underlying 1-category, and similarly for
2-functors.

\begin{prop}\label{thm:corefleq-rigged}
  Suppose that $\Lan_{J_0}\colon [(\sDt)_0,\Set] \to [(\sDl)_0,\Set]$
  preserves coreflexive equalizers.  Then the maximal extension of any
  $Q_\wbar$-coalgebra to a $\QF_\wbar$-coalgebra is $w$-rigged, and
  moreover its structure map $\phbar\colon \Lan_J \Phi_\tight \to
  \Phi_\loose$ is pointwise bijective on objects.
\end{prop}
\begin{proof}
  Let $\Phi$ be the maximal extension of $\Phi_\loose$ to a
  $\QF_\wbar$-coalgebra, and write $Q=Q_\wbar = (\QF_\wbar)_\loose$.
  Then we have an equalizer diagram
  \[ \xymatrix{ \Phi_\tight  \ar[r]^-{\phi} &
    \Phi_\loose J \ar@<1mm>[r]^-{sJ} \ar@<-1mm>[r]_-{p} &
    Q \Phi_\loose J }
  \]
  in $[\sDt,\Cat]$.  Moreover, the parallel pair $(sJ,p)$ is
  coreflexive, since $q J$ is a common splitting.  Thus, by
  assumption, the top row of the following diagram becomes an
  equalizer diagram in $[(\sDl)_0,\Set]$ after composing with
  $\ob\colon \Cat_0 \to\Set$.
  \[ \xymatrix@C=3pc{ \Lan_J \Phi_\tight  \ar[r]^-{\Lan_J \phi} \ar[d]_{\phbar} &
    \Lan_J (\Phi_\loose J) \ar@<1mm>[r]^-{\Lan_J sJ}
    \ar@<-1mm>[r]_-{\Lan_J p} \ar[d]_{\overline{p_{\Phi_\loose}}} &
    \Lan_J (Q \Phi_\loose J) \ar[d]^{\overline{p_{Q\Phi_\loose}}}\\
    \Phi_\loose \ar[r]_-{s} &
    Q \Phi_\loose \ar@<1mm>[r]^-{Q s} \ar@<-1mm>[r]_-{Q p} &
    Q Q \Phi_\loose
  }
  \]
  Since ${\overline{p_{\Phi_\loose}}}$ and
  ${\overline{p_{Q\Phi_\loose}}}$ exhibit their codomains as
  codescent objects of their domains (by construction of $Q$),
  they are pointwise bijective on objects.  Moreover, the bottom row
  is also an equalizer diagram, since $\Phi_\loose$ is a
  $Q$-coalgebra and $Q p$ is the comultiplication of
  $Q$.

  Thus, if we can show that each square in the diagram commutes, it
  will follow that $\phbar$ is also pointwise bijective on objects.
  For this, it suffices to verify that the adjunct diagram commutes in
  $[\sDt,\Cat]$:
  \[ \xymatrix@C=3pc{ \Phi_\tight  \ar[r]^-{\phi} \ar[d]_{\phi} &
    \Phi_\loose J \ar@<1mm>[r]^-{sJ} \ar@<-1mm>[r]_-{p} \ar[d]_{p_{\Phi_\loose}} &
    Q \Phi_\loose J \ar[d]^{p_{Q\Phi_\loose}}\\
    \Phi_\loose J \ar[r]_-{sJ} &
    (Q \Phi_\loose) J \ar@<1mm>[r]^-{(Q s) J} \ar@<-1mm>[r]_-{(Q p)J} &
    (Q Q \Phi_\loose) J
  }
  \]
  But now the left-hand square is just the equation $p.\phi = sJ.\phi$
  which $\Phi$ must satisfy to be a $\QF_\wbar$-coalgebra, while the two
  right-hand squares are naturality squares for $p$.
\end{proof}

The hypothesis of \autoref{thm:corefleq-rigged} holds in particular if
\bbD is inchordate, so that $(\sDt)_0$ is discrete.  In this case,
$Q_\wbar$ is just the 2-categorical $\wbar$-transformation classifier
on $[\sDl,\Cat]$, so we have:

\begin{cor}\label{thm:2catcoalg-rigged}
  If $\Psi\colon \sD\to\Cat$ is a coalgebra for the 2-categorical
  $\wbar$-transformation classifier $Q_\wbar$, then it has a canonical
  extension to a $w$-rigged \sF-weight, whose domain is the inchordate
  \sF-category on \sD, and for which $\phbar$ is pointwise bijective
  on objects.
\end{cor}

One easy application of this result is to \emph{cofree}
$Q_\wbar$-coalgebras, i.e.\ weights of the form $Q_\wbar \Psi \colon
\sD \to \Cat$.  Since a $Q_c \Psi$-weighted limit is simply an
\emph{oplax} $\Psi$-weighted limit, \autoref{thm:2catcoalg-rigged}
implies that \emph{oplax limits are canonically $l$-rigged}, and
dually.  The lifting of oplax limits to $T\bbAlg_l$, for $T$ a
2-monad, was also proven directly in~\cite{lack:lim-lax}, and the
lifting of pseudo limits to $T\bbAlg_p$ was shown
in~\cite{bkp:2dmonads}.

Note that in the cofree case, we have a split equalizer
\[ \xymatrix@C=4pc{
  \Psi J \ar[r]^-p  &
  (Q \Psi) J \ar@<1mm>[r]^-{(Q p_\Psi)J} \ar@<-1mm>[r]_-{p_{Q\Psi}} 
  \ar@/^10mm/[l]^{q J} &
  (Q Q \Psi) J \ar@/^10mm/[l]^{(Q q) J} }
\]
so that the canonical rigging of $Q \Psi$ is just $p\colon \Psi
J \to (Q \Psi)J$.  In other words, the tight projections are the
obvious ``generating'' ones of the oplax limit.

\subsection{Tightly rigged weights}
\label{sec:tightly-rigged}

In \S\ref{sec:canonical} we were concerned with constructing a rigging
of a weight that was known to be a $\QF$-coalgebra, i.e.\ with
deducing the second part of the definition of $w$-rigged weight from
the first.  We cannot hope to do the reverse in general, but there is
one case in which we can.

\begin{prop}\label{thm:tightly-rigged}
  If $\Phi$ is an \sF-weight such that $\phbar\colon \Lan_J
  \Phi_\tight \to \Phi_\loose$ is pointwise bijective on objects, then
  $\Phi$ is $p$-rigged.
\end{prop}
\begin{proof}
  Since $qJ.p$ is an identity, we have $\phi = qJ.p.\phi$, and hence
  by adjunction the following square commutes in $[\sDl,\Cat]$.
  \[\xymatrix{\Lan_J \Phi_\tight \ar[r]^{\overline{p.\phi}} \ar@{->|}[d]_{\phbar} & Q_\wbar \Phi_\loose \ar@{ >->}[d]^q\\
    \Phi_\loose \ar@{=}[r] & \Phi_\loose }\]
  The left-hand map is bijective on objects (by assumption), and the
  right-hand map is fully faithful, since it is an equivalence in
  $\Ps(\sDl,\Cat)$.  (Here, and the analogous assertion later on, is
  where we use the restriction to $w=p$.)  Therefore, by
  orthogonality, there exists a unique diagonal $s\maps \Phi_\loose
  \to Q_\wbar \Phi_\loose$ satisfying $q.s=1$ and $s \phbar =
  \overline{p.\phi}$.  The former condition says that $s$ is a section
  of $q$; the latter is equivalent to $sJ.\phi = p.\phi$, so that $s$
  is actually a morphism $\Phi \to \QF_\wbar \Phi$ of \sF-weights.

  Thus, for $\Phi$ to be a $\QF_\wbar$-coalgebra, it remains only to
  show that $Qp.s =Qs.s$.  We claim that both $Qp . s$ and $Qs. s$ are
  diagonal fillers in a square of the following form:
  \begin{equation}
    \xymatrix{\Lan_J \Phi_\tight \ar[r]\ar@{->|}[d]_\phbar &
      QQ\Phi_\loose \ar@{ >->}[d]^{Qq}\\
      \Phi_\loose \ar[r] \ar@{.>}[ur] & Q \Phi_\loose.}\label{eq:uniqsq}
  \end{equation}
  This is equivalent to saying that
  \begin{align*}
    Qq . Qp . s &= Qq . Qs. s \qquad\text{and}\\
    Qp . s . \phbar &=  Qs . s . \phbar
  \end{align*}
  The first equation is easy; both sides are equal to $s$ since $q . p =
  1 = q . s$.  For the second, we consider the adjuncts and compute
  \begin{alignat*}{2}
    (Qs)J . sJ . \phi
    &= (Qs)J . p_{\Phi} . \phi
    &\qquad& \text{(definition of $s$)}\\
    &= p_{Q\Phi} . sJ . \phi
    &\qquad& \text{(naturality of $p$)}\\
    &= p_{Q\Phi} . p_\Phi . \phi
    &\qquad& \text{(definition of $s$)}\\
    &= Q p_\Phi . p_\Phi . \phi
    &\qquad& \text{(naturality of $p$)}\\
    &= Q p_\Phi . sJ . \phi
    &\qquad& \text{(definition of $s$)}
  \end{alignat*}
  Thus, there does exist a square~\eqref{eq:uniqsq} in which $Qp.s$
  and $Qs.s$ are both diagonal fillers.  Since $Qq$ is pointwise fully
  faithful and \phbar\ is bijective on objects, by orthogonality any
  such square has a unique diagonal filler; thus $Qp . s = Qs. s$ as
  desired.
\end{proof}

We call an \sF-weight \textbf{tightly rigged} if $\phbar$ is pointwise
bijective on objects.  All of the rigged weights we have encountered
so far are tightly rigged; Propositions \ref{thm:corefleq-rigged} and
\ref{thm:tightly-rigged} provide some reasons why many rigged
weights are tightly rigged.  However, not every rigged weight is
tightly rigged, or can even be made so by changing the tight part.

\begin{eg}\label{eg:not-tr2}
  Let $\sD_\loose$ be the set of natural numbers, regarded as a poset
  (hence a locally discrete 2-category) with the reverse of
  its usual ordering, and define the unique morphism $n\to m$ to be
  tight if either (1) $n=m$ or (2) $n$ is even and $m$ is odd.  Let
  $\Phi_\loose$ be constant at \bbone; then
  \[Q_l \Phi_\loose(n) =
  \begin{cases}
    \bbone & \quad \text{if } n \text{ is odd}\\
    \bbtwo & \quad \text{if } n \text{ is even}
  \end{cases}
  \]
  so we can make $\Phi_\loose$ into a $Q_l$-coalgebra (or a $Q_c$- or
  $Q_p$-coalgebra).

  Now for any $\Phi_\tight$ and $\phi\colon\Phi_\tight \to \Phi_\loose
  J$ such that the resulting \sF-weight $\Phi$ is a $\QF_w$-coalgebra,
 $\Phi_\tight$ can be supported only on the odd numbers (at each of
  which, it may be \bbone or $\emptyset$). Then for any $n$, $\Lan_J
  \Phi_\tight(m)$ is a discrete category with one object for every odd
  $m\ge n$ such that $\Phi_\tight(m)=\bbone$.  Hence, if $\Phi$ is to
  be rigged, then $\Phi_\tight$ must be nonempty at arbitrarily large
  odd numbers; but \phbar cannot be bijective at $n$ if $\Phi_\tight$
  is nonempty at more than one $m\ge n$.  Thus, there are many choices
  of $\Phi_\tight$ for which $\Phi$ is rigged, but none for which it
  is tightly rigged.

  If we modify \bbD by stipulating that $n\to m$ is tight if either
  (1) $n=m$ or (2) $m$ is odd and $n-m$ is congruent to $0$, $1$, or
  $3$ mod $4$, then there are two incompatible choices of
  $\Phi_\tight$ for which $\Phi$ is tightly rigged: we can take
  $\Phi_\tight$ to be nonempty at exactly the numbers that are $1$ mod
  $4$, or at exactly those that are $3$ mod $4$.  Neither of these is
  the maximal choice from \S\ref{sec:canonical}, which would be
  nonempty at all odd numbers; in that case the weight is rigged but
  not tightly rigged.  So although the canonical rigging is often
  tightly rigged, by \autoref{thm:corefleq-rigged}, it is not always
  so, even if a tight rigging exists.
\end{eg}

\subsection{2-categories and PIE-limits}
\label{sec:pie}

We now consider what the theory we have developed has to say about
purely 2-categorical weights.  Here the statements are simpler, since
the distinction between rigged weights and $Q$-coalgebras evaporates.
Specifically, we have the following.

\begin{prop}\label{thm:2cat-char}
  Let $\Phi\maps \sD\to\Cat$ be a \Cat-weight.  Then $\Phi$-weighted
  limits lift along $U_w\maps T\Alg_w \to\sK$, for any 2-monad $T$ on
  a 2-category \sK, if and only if $\Phi$ is a $Q_\wbar$-coalgebra,
  where $Q_\wbar$ is the 2-categorical $\wbar$-transformation
  classifier on $[\sD,\Cat]$.
\end{prop}
\begin{proof}
  If $\Phi$ is a $Q_\wbar$-coalgebra, then by
  \autoref{thm:2catcoalg-rigged} it has a canonical extension to a
  $w$-rigged \sF-weight $\Psi$, so that $\Psi$-weighted limits
  lift to the \sF-category $T\bbAlg_w$
  for any $T$.  Hence, in particular, $\Phi$-weighted limits lift to
  the 2-category $T\Alg_w = (T\bbAlg_w)_\loose$ for any 2-monad $T$.
  We could also prove this by imitating the proof of
  \autoref{thm:coalg-looselift} in the 2-categorical world.

  For the converse, we seem to have no alternative to imitating (the
  easy part of) the proof of \autoref{thm:lifts-rigged}: if
  $\Phi$-weighted limits lift to $T\Alg_w$, then the $\Phi$-weighted
  colimit $\Phi = \Phi * Y$ lifts to $Q_\wbar\Coalg_\wbar$ (the
  diagram $Y$ lying in $Q_\wbar\Coalg_\wbar$ since representables are
  $Q_\wbar$-coalgebras and $Q_\wbar$ is $\wbar$-idempotent).  Thus,
  $\Phi$ is a $Q_\wbar$-coalgebra.
\end{proof}

Note that the characterization is weaker than the \sF-categorical
version: since the 2-categorical $U_w$ is not conservative, lifting of
limits does not imply their creation.

Moreover, when $w=l$ or $c$, this 2-categorical version of the theorem is not
very useful, since in these cases there seem to be few purely
2-categorical $Q_w$-coalgebras aside from the cofree ones.  Furthermore,
this version contains no information about which projections are
strict and detect strictness, a detail which was important
in~\cite{bkp:2dmonads} even when $w=p$.

However, when $w=p$, it does turn out that the $Q_p$-coalgebras are
precisely the class of limits already known to lift to $T\Alg$ for all
2-monads $T$, namely the PIE-limits (those constructible from
Products, Inserters, and Equifiers).  In the rest of this section we
give a proof of this equivalence.

Let \sD\ be a 2-category and $\Phi\maps \sD\to\Cat$ a 2-functor.  We
write $\sD_0$ for the underlying ordinary category of \sD, and
$\ob\Phi_0\maps \sD_0\to\Set$ for the composite $\sD_0
\xrightarrow{\Phi_0} \Cat_0 \xrightarrow{\ob} \Set$.  For any functor
$F\maps \mathbf{C}\to\Set$ we write $\mathrm{el}(F)$ for its
\emph{category of elements} (aka ``Grothendieck construction'').
Recall the following theorem from~\cite{pr:pie-limits}.

\begin{thm}\label{thm:pie-char}
  A weight $\Phi\maps \sD\to\Cat$ is a PIE-weight if and only if
  $\mathrm{el}(\ob\Phi_0)$ is a disjoint union of categories with initial
  objects.
\end{thm}

We require the following easy lemma.  We write $\ob(\mathbf{C})$ for
the set of objects of a category $\mathbf{C}$, regarded as a discrete
category, with the obvious inclusion functor $J\maps
\ob(\mathbf{C})\to \mathbf{C}$.

\begin{lem}\label{thm:pie-lan}
  For any functor $F\maps \mathbf{C}\to\Set$, $\mathrm{el}(F)$ is a
  disjoint union of categories with initial objects if and only if
  there exists a functor $G\maps \ob(\mathbf{C})\to \Set$ 
  and an isomorphism $\Lan_J G\iso F$.
\end{lem}
\begin{proof}
  If $G$ exists, then we have
  \[F(c) = \sum_{c'\in \mathbf{C}\atop x\in G(c')}\mathbf{C}(c',c).\]
  and therefore
  \[\mathrm{el}(F) = \sum_{c'\in \mathbf{C}\atop x\in G(c')} c'/\mathbf{C},
  \]
  and $c'/\mathbf{C}$ certainly has an initial object.  Conversely, if
  $\mathrm{el}(F)$ is a disjoint union of categories with initial
  objects, we choose for each component $\mathbf{D}_i$ an initial
  object $(c_i,x_i)$ and let $G(c) = \{x_i \mid c_i=c\}$.  Then
  \[1 \xrightarrow{(c_i,x_i)} \mathbf{D}_i \to \mathbf{C}\] is a
  factorization of $1\xrightarrow{c_i} \mathbf{C}$ as an initial
  functor followed by a discrete fibration.  Since such factorizations
  are unique (see~\cite{sw:cmprh-fact-funct}), we must have
  $\mathbf{D}_i \iso (c_i/\mathbf{C})$, and therefore $F\iso \Lan_J
  G$.
\end{proof}

\begin{thm}\label{thm:pie-qpcoalg}
  A \Cat-weight $\Phi\colon \sD\to\Cat$ is a $Q_p$-coalgebra if and
  only if it is a PIE-weight.
\end{thm}
\begin{proof}
  In one direction the proof is obvious: PIE-weights are known to lift
  to $T\Alg_p$ for any 2-monad $T$, hence by \autoref{thm:2cat-char}
  they must be $Q_p$-coalgebras.

  Alternatively, by \autoref{thm:pie-char} and \autoref{thm:pie-lan} we
  have $\ob\Phi_0 \cong \Lan_J G$ for some $G$, which equivalently
  means we have a bijective-on-objects map $k\colon \Lan_J G \to
  \Phi$.  Thus, if we define $\psi\colon \Phi_\tight \hookrightarrow
  \Phi J$ to be the full image of the adjunct $\kbar\colon G \to \Phi
  J$, then we have a tightly rigged \sF-weight, whose domain is the
  inchordate \sF-category on \sD.  By \autoref{thm:tightly-rigged}, it
  is a $\QF_p$-coalgebra, so its loose part, namely $\Phi$, is a
  $Q_p$-coalgebra.  Combining this argument with
  \autoref{thm:2cat-char}, we obtain a new proof that PIE-weights lift
  to $T\Alg_p$ for any $T$.

  For the converse, we invoke \autoref{thm:corefleq-rigged} for the
  inchordate \sF-category on \sD, and conclude that any
  $Q_p$-coalgebra $\Phi$ can be made into a $w$-rigged \sF-weight for
  which $\Lan_J \Phi_\tight \to \Phi$ is pointwise bijective on
  objects.  Defining $G \coloneqq \ob (\Phi_\tight)_0$, we have
  $\Lan_{J_0} G \cong \ob \Phi_0$; hence by \autoref{thm:pie-char} and
  \autoref{thm:pie-lan} $\Phi$ is a PIE-weight.
\end{proof}

\subsection{Saturation}
\label{sec:saturation}

The \textbf{saturation} of a class $\mathcal{X}$ of \sV-weights is the class
of weights $\Phi$ such that every $\mathcal{X}$-complete \sV-category
is $\Phi$-complete and every $\mathcal{X}$-continuous \sV-functor is
$\Phi$-continuous.  This notion was introduced in~\cite{ak:closure} under the
name \textbf{closure}, but ``saturation'' is now standard.
The main result of~\cite{ak:closure} is that
$\Phi\maps \sD\to\sV$ lies in the saturation of $\mathcal{X}$ if and only
if it lies in the closure of the representables under
$\mathcal{X}$-colimits in $[\sD,\sV]$.

A class of weights is called \textbf{saturated} if it is its own saturation.

\begin{thm}\label{thm:saturated}
  For any $w$, the class of $w$-rigged weights is saturated.
\end{thm}
\begin{proof}
  Let $\mathscr{R}$ denote the class of $w$-rigged weights, and let $\Phi$
  be an \sF-weight such that every $\mathscr{R}$-complete \sF-category
  is $\Phi$-complete and every $\mathscr{R}$-continuous \sF-functor is
  $\Phi$-continuous.  By \autoref{prop:reduction}, to show $\Phi$ is
  $w$-rigged it suffices to show that it lifts to $T\bbAlg_w$ for any
  \sF-monad $T$ on a \emph{complete} \sF-category \bbK.  But in this
  case, $T\bbAlg_w$ is $\mathscr{R}$-complete and $U_w\colon T\bbAlg_w
  \to\bbK$ is $\mathscr{R}$-continuous. Hence, by assumption,
  $T\bbAlg_w$ is also $\Phi$-complete and $U_w$ is also
  $\Phi$-continuous. But this is just to say that $U_w$  lifts 
  $\Phi$-weighted limits; so by \autoref{thm:lifts-rigged}, $\Phi$ is
  $w$-rigged.
\end{proof}

Recall that essentially by definition, the PIE-weights are the
saturation of the class consisting of products, inserters, and
equifiers.  However, we do not know of any manageable collection of
weights which generates the $w$-rigged weights, even for $w=p$.

\appendix

\section{Alternative proof of the lifting theorem}

Here, as promised, we give an alternative proof of the lifting theorem. 
This could replace all of Section~\ref{sec:lifting} after 
Remark~\ref{rmk:appendix}.   We suppose throughout that $w=l$.

Suppose that rather than an
individual $\QF^\bbD$-coalgebra $\Phi$, we have a functor $\Psi\colon
\bbE\op\to \QF^\bbD\bbCoalg_c$ for some other small \sF-category \bbE.
Then $\Psi$ has an underlying functor $\bbE\op \to [\bbD,\bbF]$, which
is equivalently an \sF-profunctor $\bbD\pto \bbE$, and so
(if we assume, as before, that \bbK is complete) 
we have an induced functor
\[\{\Psi,-\} \colon [\bbD,\bbK] \to [\bbE,\bbK].
\]
where for $M:\bbD\to\bbK$ and $E\in\bbE$ we have 
$\{\Psi,M\}E=\{\Psi E,M\}$. 
We would like to lift this to a functor $\bbOplax(\bbD,T\bbAlg_l)\to
\bbOplax(\bbE,T\bbAlg_l)$, and by the same arguments as  in \S\ref{sec:lifting}, it
suffices to construct a lax monad morphism from $\bbOplax(\bbD,T)$ to
$\bbOplax(\bbE,T)$.

However, just as in \S\ref{sec:lifting}, $\bbOplax(\bbD,T)$ is a
lifting of $[\bbD,T]$ to the Kleisli category $\bbOplax(\bbD,\bbK)$ of
$\RF^\bbD$ induced by a distributive law
\[k^\bbD\colon [\bbD,T] \RF^\bbD \to \RF^\bbD [\bbD,T],
\]
and likewise for $\bbOplax(\bbE,T)$ and $\RF^\bbE$.
Thus, by \autoref{lem:monad-stuff-2}, it suffices to show that
$\{\Psi,-\}$ is a lax monad morphism from $[\bbD,T]$ to $[\bbE,T]$ and
also a colax monad morphism from $\RF^\bbD$ to $\RF^\bbE$, in such a way
that the diagram~\eqref{eq:monad-stuff} from \autoref{lem:monad-stuff}
commutes.

\begin{prop}\label{thm:functoriality-1}
  Let \bbK be complete and \bbD and \bbE small, and let $\Psi\colon
  \bbE\op\to \QF^\bbD_{c}\bbCoalg_{c}$ be an \sF-functor.
  Then the \sF-functor $F=\{\Psi,-\}\colon[\bbD,\bbK]\to[\bbE,\bbK]$
  naturally has the structure of a
  morphism in $\Mnd_l(\Mnd_c(\sF\Cat))$ from $k^\bbD$ to $k^\bbE$,
   and therefore lifts to a functor
  \[\llim{\Psi,-}\colon \bbOplax(\bbD,T\bbAlg_l) \to \bbOplax(\bbE,T\bbAlg_l).\]
\end{prop}
\begin{proof}
  The proof is mostly a straightforward generalization of
  \autoref{prop:alg-structure-on-L}; the one somewhat different thing
  is that we need $F$ to be a colax monad morphism from $\RF^\bbD$ to
  $\RF^\bbE$, rather than merely a $\RF^\bbD$-opalgebra.  For this,
  recall that $\bbLax (\bbE\op, \QF^\bbD \bbCoalg_c)
  \cong \bbLax (\bbE\op, \QF^\bbD) \bbCoalg_c$.  Since $\Psi$ is an
  object of the former, it is equivalently an object of the latter,
  i.e.\ the coalgebra structure maps $s_e\colon \Psi_e \to \QF^\bbD
  \Psi_e$ are (tight) lax \sF-natural in $e$.  Therefore, the
  composite
  \begin{equation}
    (F \RF^\bbD)_e = \{ \Psi_e, \RF^\bbD - \} \cong \{ \QF^\bbD
    \Psi_e, -\} \xrightarrow{s_e^*} \{\Psi_e , -\} = F_e\label{eq:colax-str}
  \end{equation}
  is (tight) oplax \sF-natural in $e$.   (The
  isomorphism is from \autoref{thm:qr-adjt}.)
  Since $\RF^\bbE$ coclassifies
  oplax \sF-natural transformations, we have an induced tight and
  strict transformation
  \[ \chi\colon F \RF^\bbD \to \RF^\bbE F.
  \]
  The unit and associativity axioms for the \QF-coalgebra structure of
  $\Psi$ directly imply that $\chi$ makes $F$ into a colax morphism of
  monads from $\RF^\bbD$ to $\RF^\bbE$.

   The rest of the proof is basically the same as the proof of
  \autoref{prop:alg-structure-on-L}, so we leave it to the reader.
  In particular, the same argument shows that the induced $T$-algebra
  structures are the same as those satisfying~\eqref{eq:defn-l}.
\end{proof}

We can also make this lifting functorial on morphisms between
profunctors of the above sort.

\begin{prop}\label{thm:alg-str-2cell}
  Let $\Psi,\Upsilon\colon \bbE\op\to \QF^\bbD_{c}
\bbCoalg_{c}$
  be as in \autoref{thm:functoriality-1}, and let $\alpha\colon
  \Psi\to\Upsilon$ be a tight strict \sF-natural transformation (whose
  components are thus tight strict \QF-morphisms).  Then $\alpha^* \colon
  \{\Upsilon,-\} \to \{\Psi,-\}$ is a 2-cell in
  $\Mnd_l(\Mnd_c(\sF\Cat))$,  and hence induces a natural
  transformation 
  \[\alpha^*\colon \llim{\Upsilon,-} \to \llim{\Psi,-}.\]
\end{prop}
\begin{proof}
  We must verify that $\alpha$ is a monad 2-cell for both the lax and
  the colax structures.  For the lax monad morphism structures from
  $[\bbD,T]$ to $[\bbE,T]$, this follows from the existence of a
  lifting $\{\alpha,-\}$ at the level of strict algebras.  For the
  colax monad morphism structures from $\RF^\bbD$ to $\RF^\bbE$, this
  follows because they are constructed out of the $\QF$-coalgebra
  structures on $\Psi$ and $\Upsilon$, and $\alpha$ consists of strict
  \QF-morphisms.
\end{proof}

There is one further sort of functoriality we need, which involves
profunctor composition.  Let $\sF\Prof$ denote the bicategory of small
\sF-categories and \sF-profunctors.  (Recall that
composition of profunctors is defined as a coend.)  Then for any
complete \sF-category \bbK, we have a pseudofunctor
\begin{equation}
  [-,\bbK]\colon \sF\Prof\co\to \sF\Cat,\label{eq:kpsfr}
\end{equation}
which sends $\bbD$ to $[\bbD,\bbK]$ and $\Psi\colon \bbD\pto \bbE$ to
$\{\Psi,-\}\colon [\bbD,\bbK] \to [\bbE,\bbK]$.

Our goal is to lift this to a pseudofunctor sending \bbD to
$\bbOplax(\bbD,T\bbAlg_l)$.  However, since we are only considering
profunctors that are ``pointwise \QF-coalgebras,'' we need a
bicategory of such.  This is the purpose of the following sequence of
lemmas, analyzing how \QF interacts with profunctor composition.

\begin{lem}
  For \sF-profunctors $\Upsilon\colon \bbC\pto\bbD$ and $\Psi\colon
  \bbD\pto\bbE$, we have $\QF^\bbC_{c}(\Psi\otimes_\bbD \Upsilon) \cong
  \Psi \otimes_\bbD \QF^\bbC_{c} \Upsilon$.
\end{lem}
\begin{proof}
  We can observe that $\QF^\bbC_{c}$ is constructed using colimits in
  $[\bbC,\bbF]$, and $\Psi \otimes_\bbD -$ is a weighted colimit;
  hence the two commute.  Or, we can verify that both satisfy the same
  universal property.
\end{proof}

\begin{lem}\label{thm:qq-coadjt}
  For \sF-profunctors $\Upsilon\colon \bbC\pto\bbD$ and $\Psi\colon
  \bbD\pto\bbE$, we have $\QF^\bbD_{c } \Psi \otimes_{\bbD}
  \Upsilon \cong \Psi \otimes_\bbD \QF^{\bbD\op}_{l}\Upsilon$.
\end{lem}
\begin{proof}
  Again, we can prove this using the construction of both sides out of
  colimits, or show directly that they have the same universal
  property.   It can also be regarded as a special case of (the
  dual of) \autoref{thm:qr-adjt}. (Note the reversal of sense
  from  $l$ to $c$  on the two sides of the isomorphism.)
\end{proof}

\begin{lem}\label{thm:compose-coalg}
  Given $\Psi\colon \bbE\op\to \QF^\bbD_{c}\bbCoalg_{c}$ and
  $\Upsilon\colon \bbD\op\to \QF^\bbC_{c}\bbCoalg_{c}$, the
  composite $\Psi \otimes_\bbD \Upsilon$ has the structure of a
  functor $\bbE\op \to \QF^\bbC_{c}\bbCoalg_{c}$.
\end{lem}
\begin{proof}
  Since $\bbLax(\bbD\op, \QF^\bbC_c\bbCoalg_c) \cong \bbLax(\bbD\op,
  \QF^\bbC_c)\bbCoalg_c$, the structure map $\Upsilon \to \QF^\bbC_c
  \Upsilon$ is lax natural in $\bbD\op$, so it may equivalently be regarded
  as a map $\QF^{\bbD\op}_l \Upsilon \to \QF^\bbC_c \Upsilon$.  Similarly,
  we have $\QF^{\bbE\op}_l \Psi \to \QF^\bbD_c \Psi$, and the desired
  structure on $\Psi \otimes_\bbD \Upsilon$ ought to be a morphism
  \[\QF^{\bbE\op}_l (\Psi \otimes_\bbD \Upsilon) \to \QF^\bbC_c (\Psi \otimes_\bbD \Upsilon).\]
  We can define this morphism to be the composite
  \begin{multline}
     \QF^{\bbE\op}_l( \Psi \otimes_\bbD \Upsilon)
      \xrightarrow{\cong} \QF^{\bbE\op}_l \Psi \otimes_\bbD \Upsilon
      \xrightarrow{s \otimes 1} \QF^\bbD_c \Psi \otimes_\bbD \Upsilon\\
      \xrightarrow{\cong} \Psi \otimes_\bbD \QF^{{\bbD\op}}_l \Upsilon
      \xrightarrow{1 \otimes s} \Psi \otimes_\bbD \QF^\bbC_c \Upsilon
      \xrightarrow{\cong} \QF^\bbC_c (\Psi \otimes_\bbD \Upsilon).\label{eq:compstr}
  \end{multline}
  The axioms follow straightforwardly from those for $\Psi$ and $\Upsilon$.
\end{proof}

There is another way to obtain a \QF-coalgebra structure on $\Psi
\otimes_\bbD \Upsilon$: we can apply \autoref{thm:functoriality-1} to
$\Psi$ and the monad $(\QF^\bbC_{c})\op$ on $[\bbC,\bbF]\op$.  This gives a
lifting of the functor
\[ \{ \Psi,-\} \colon [\bbD, [\bbC,\bbF]\op] \longrightarrow [\bbE, [\bbC,\bbF]\op] 
\]
(which is just $(-\otimes_{\bbD} \Psi)\op$) to a functor
\[ [\bbD, (\QF^\bbC_{c}\bbCoalg_{c})\op] \longrightarrow [\bbE, (\QF^\bbC_{c}\bbCoalg_{c})\op].
\]
Comparing~\eqref{eq:compstr} to~\eqref{eq:lifted-tstr}, and recalling
that \autoref{thm:qq-coadjt} is the dual of \autoref{thm:qr-adjt}, we
see that these two definitions agree.  

It is also easy to see from the above definition that tight strict
transformations (consisting of strict \QF-morphisms) induce similar
tight strict transformations between composites of profunctors.
We conclude:

\begin{lem}
  There is a bicategory $\QF\Prof$ whose objects are small
  \sF-categories, and whose hom-categories are
  \[ \QF\Prof(\bbD,\bbE) =
  \sF\Cat(\bbE\op, \QF^\bbD_{c}\bbCoalg_{c}).
  \]
\end{lem}

\begin{proof}
   The unit profunctors lie in $\QF\Prof$ by
  \autoref{thm:yon-qcoalg}, while composition is given by \autoref{thm:compose-coalg}. 
  A computation shows that  the associativity and unitality
  isomorphisms in $\sF\Prof$ are strict \QF-morphisms.
\end{proof}

The 2-cells in $\QF\Prof$ are tight strict \QF-morphisms.  There is a
forgetful functor $\QF\Prof \to \sF\Prof$ which is bijective on
objects and faithful on 2-cells; \autoref{thm:rigged-strictmor}
implies that it is full on 2-cells between profunctors that are
pointwise rigged.\footnote{ Although we will not need it, we
  observe that the pointwise-rigged profunctors actually form a
  sub-bicategory of $\QF\Prof$.  This follows from
  \autoref{thm:saturated}, which implies that rigged weights are
  closed under rigged colimits.}
We can now deduce the following strong functoriality statement.

\begin{prop}\label{thm:psfr}
  For a complete \sF-category \bbK and an \sF-monad $T$ on \bbK, the
  pseudofunctor $[-,\bbK]\colon \QF\Prof\co \to \sF\Prof\co\to
  \sF\Cat$ lifts to a pseudofunctor $\QF\Prof\co\to \sF\Cat$ sending
  \bbD to $\bbOplax(\bbD,T\bbAlg_{l})$.
\end{prop}
\begin{proof}
  It suffices to lift $[-,\bbK]$ to a pseudofunctor
  \begin{equation}
    \QF\Prof\co \to \Mnd_l(\Mnd_c(\sF\Cat)),\label{eq:qprof-lifting}
  \end{equation}
  since then we can apply the Kleisli-object-assigning functor
  $\Mnd_c(\sF\Cat)\to \sF\Cat$, followed by the functor
  $\Mnd_l(\sF\Cat)\to \sF\Cat$ which constructs $T\bbAlg_l$ from an
  \sF-monad $T$.

  However, we have already constructed the
  lifting~\eqref{eq:qprof-lifting} on morphisms
  (\autoref{thm:functoriality-1}) and 2-cells
  (\autoref{thm:alg-str-2cell}), so it remains to verify its
  functoriality.  Functoriality on 2-cells is immediate, so we need to
  check that for $\Psi\colon \bbC\pto\bbD$ and $\Upsilon\colon
  \bbD\pto\bbE$ in $\QF\Prof$, the colax and lax monad morphism
  structures on $\{\Upsilon \otimes_\bbD \Psi, -\}$ are the composites
  of those on $\{\Upsilon,-\}$ and $\{\Psi,-\}$.

  For the lax monad morphism structures from $[\bbC,T]$ to $[\bbE,T]$,
  this follows easily since all limits lift, functorially, to
  categories of strict algebras.  And for the colax monad morphism
  structures from $\RF^\bbC$ to $\RF^\bbE$, it follows from the
  construction in \autoref{thm:compose-coalg} of the $\QF$-coalgebra
  structure on $\Upsilon\otimes_\bbD\Psi$.
\end{proof}

Now we need a supply of good profunctors to which to apply this
functoriality.  Let $H\colon \bbD\to\bbE$ be any \sF-functor, and
$H^\bullet\colon \bbE \pto \bbD$ the profunctor defined by
$H^\bullet(d,e) = \bbE(H(d),e)$.  Then $\{H^\bullet,-\}\colon
[\bbE,\bbK] \to [\bbD,\bbK]$ is simply given by precomposition with
$H$.  Moreover, since $H^\bullet\colon \bbD\op \to [\bbE,\bbF]$ is the
composite of $H\op$ with the Yoneda embedding, by
\autoref{thm:yon-qcoalg} it lifts naturally to a morphism in
$\QF\Prof$.

\begin{lem}\label{lem:restriction}
  The lifted functor
  \[\overline{\{H^\bullet,-\}} \colon \bbOplax(\bbE,T\bbAlg_{l})\to
  \bbOplax(\bbD,T\bbAlg_{l})
  \]
  is also given by precomposition with $H$.
\end{lem}
\begin{proof}
  By construction of the \QF-coalgebra structure in \autoref{thm:yon-qcoalg}.
\end{proof}

\begin{cor}\label{thm:prof-restriction}
  For $H\colon \bbD\to\bbE$ and $\Psi\colon \bbE\op \to
  \QF^\bbC_{c}\bbCoalg_{c}$, the composite $H^\bullet \otimes_\bbE
  \Psi$ in $\QF\Prof$ is naturally isomorphic to the composite functor
  \[ \bbD \xrightarrow{H} \bbE \xrightarrow{\Psi} \QF^\bbC_{c}\bbCoalg_{c}.
  \]
\end{cor}
\begin{proof}
  This follows from \autoref{lem:restriction}, together with the
  observation after \autoref{thm:compose-coalg} that composition in
  $\QF\Prof$ can also be described as a lifted limit.
\end{proof}

Similarly, for $H:\bbD\to\bbE$ as above, we have a profunctor $H_\bullet\colon
\bbD\pto\bbE$ defined by $H_\bullet(e,d) = \bbE(e,H(d))$, such that
$\{H_\bullet,-\}$ is right Kan extension along $H$.

\begin{lem}\label{lem:good}
  Suppose that $H\colon \bbD\hookrightarrow \bbE$ is the inclusion of
  a full subcategory such that  for any $E\in \bbE \setminus \bbD$,
  \begin{itemize}
  \item the weight $\bbE(E,H-)\colon \bbD \to \bbF$ is a \QF-coalgebra, and
  \item if there exists a nonidentity tight morphism $E'\to E$ in
    \bbE, then $\bbE(E,H-)$ is rigged.
  \end{itemize}
  Then $H_\bullet$  has an induced structure of a morphism in $\QF\Prof$.
\end{lem}
\begin{proof}
  By \autoref{thm:yon-qcoalg} and the first assumption, $H_\bullet$
  takes each object of $\bbE$ to a \QF-coalgebra.  Since $\QF_\lambda$
  is weakly idempotent, $H_\bullet$ then necessarily takes each loose
  morphism of \bbE to a (loose) weak \QF-morphism.  Finally, by
  \autoref{thm:rigged-strictmor} and the fact that all representables
  are rigged, the second assumption implies that $H_\bullet$ takes
  tight morphisms to (tight) strict \QF-morphisms.
\end{proof}

\begin{cor}\label{thm:lan-coalg}
  If $H\colon \bbD\hookrightarrow \bbE$ satisfies the hypotheses of
  \autoref{lem:good} and $\Phi\colon \bbD \to \bbF$ is a
  $\QF^\bbD$-coalgebra, then $\Lan_H \Phi\colon \bbE\to\bbF$ is
  a $\QF^\bbE$-coalgebra.  Moreover, if $\Phi$ is rigged, so is
  $\Lan_H \Phi$.
\end{cor}
\begin{proof}
  If $\Phi$ is identified with a profunctor in $\QF\Prof$ from $\bbD$
  to the unit \sF-category, then $\Lan_H \Phi$ can be identified with
  $H_\bullet \otimes_\bbD \Phi$, with \QF-coalgebra structure from
  \autoref{thm:compose-coalg}.  The second statement is
  straightforward using the fact that colimits in \sF, including left
  Kan extensions, are obtained as the full-embedding reflections of
  colimits in $\Cat^\bbtwo$.
\end{proof}

\begin{cor}\label{thm:unit-coalg}
  If $H\colon \bbD\hookrightarrow \bbE$ satisfies the hypotheses of
  \autoref{lem:good}, then we have an isomorphism $H^\bullet
  \otimes_\bbE H_\bullet \cong 1_{\bbD}$ in $\QF\Prof$.
\end{cor}
\begin{proof}
  Since $H$ is fully faithful, we have such an isomorphism in
  $\sF\Prof$.  And since the profunctor $1_\bbD$ is pointwise rigged
  (as observed after \autoref{thm:yon-qcoalg}), it has a unique
  \QF-coalgebra structure; thus the isomorphism lies in $\QF\Prof$.
\end{proof}




Finally, recall that given any profunctor $\Psi\colon \bbD\pto\bbE$, its
\emph{collage} is the category $|\Psi|$ whose objects are the disjoint
union of those of \bbD and \bbE, and whose morphisms are
\begin{equation*}
  \begin{array}{lcl}
    |\Psi|(d,d') &=& \bbD(d,d')\\
    |\Psi|(e,e') &=& \bbE(e,e')\\
    |\Psi|(e,d) &=& \Psi(e,d)\\
    |\Psi|(d,e) &=& 0.
  \end{array}
\end{equation*}
We can now prove the loose part of the universal property.

\begin{thm}\label{thm:coalg-looselift-app}
  Let $\Phi\maps \bbD\to\bbF$ be an \sF-weight which is a
  $\QF_{c}$-coalgebra, let $T$ be an \sF-monad on a complete
  \sF-category \bbK, and let $G\maps \bbD\to T\bbAlg_w$ be an
  \sF-functor.  Then the $T$-algebra structure on $\{\Phi,U G\}$
  obtained from \autoref{thm:functoriality-1} gives it the
  universal property of the limit $\{\Phi_\loose, G_\loose\}$ in the
  2-category $(T\bbAlg_w)_\loose$.
\end{thm}

\proof Let $\bL=(L,\ell)$ be the $T$-algebra constructed in 
\autoref{thm:functoriality-1}.
We must exhibit a $\Phi$-weighted cone $\eta:\Phi\to
T\bbAlg_w(\bL,G)$ such that for any loose $\Phi$-weighted cone
$\alpha:\Phi_\loose\to T\bbAlg_w(\bA,G_\loose)$ there is a unique factorization 
$\alpha':\bA\to\bL$ (plus a similar unique factorization of 2-cells).
We do this by first specifying a factorization $\alpha'$
for each $\alpha$, then showing that this is natural in $\bA$, in the sense that 
for any loose map $f:\bB\to \bA$ in $T\bbAlg_w(\bA,G_\loose)$, the assigned 
factorization of the loose $\Phi$-weighted cone
$$\xymatrix{
\Phi_\loose \ar[r]^-{\alpha} & T\bbAlg_w(\bA,G_\loose) \ar[rr]^{T\bbAlg_w(f,G_\loose)} &&
T\bbAlg_w(\bB,G_\loose) }$$
is the composite $\alpha' f:\bB\to\bA$. 
Finally we show that the specified  factorization of the loose $\Phi$-weighted cone 
$\eta_\loose:\Phi_\loose\to T\bbAlg_w(\bL,G_\loose)$ is the identity. This gives 
the uniqueness of the factorization. 

The strategy is to define various auxiliary \sF-categories, \sF-functors out of which describe
the various structures (weak cones, factorizations, etc.) involved in the previous 
paragraph. These are listed below.
\begin{itemize}
\item Let $\bbE$ be the collage of $\Phi$.  An \sF-functor $\bbE\to\bbA$ 
consists of an \sF-functor $G:\bbD\to\bbA$, an object $A\in\bbA$, and a
$\Phi$-weighted cone $\Phi\to\bbA(A,G)$. We write $*$ for the object of $\bbE$
not in \bbD. 
\item Let $\bbD^{l}$ be the collage of the weight $\ls{\Phi}$,
  defined by $\ls{\Phi}_\loose = \Phi_\loose$ and $\ls{\Phi}_\tight =
  0$.
  An \sF-functor $\bbD^{l}\to\bbA$ consists of an \sF-functor
  $G:\bbD\to\bbA$, and object $A\in\bbA$, and a loose $\Phi$-weighted
  cone; that is, a 2-natural $\Phi_\loose\to\sAl(A,G_\loose)$.  We
  write $x$ for the object of $\bbD^{l}$ not in \bbD.
\item Let $\bbE^l$ be the \sF-category obtained from $\bbE$ by adjoining an 
object $x$ and a loose morphism $x\lto *$. An \sF-functor $\bbE^l\to\bbA$ 
consists of an \sF-functor $G:\bbD\to\bbA$, objects $A,B\in\bbA$, a 
$\Phi$-weighted cone $\eta:\Phi\to\bbA(A,G)$, and a loose morphism $f:B\lto A$.
(Of course there is then an induced loose $\Phi$-weighted cone 
$\Phi_\loose\to\sAl(B,G_\loose)$, and which factorizes through $\eta$ by $f$.)
\item Let $\bbD^{ll}$ be the \sF-category obtained from $\bbD^l$ by adjoining
an object $y$ and a loose morphism $y\lto x$. An \sF-functor $\bbD^l\to\bbA$
consists of an \sF-functor $G:\bbD\to\bbA$, objects $A,B\in\bbA$, a 
loose $\Phi$-weighted cone $\Phi_\loose\to\sAl(A,G_\loose)$, and a loose morphism
$B\lto A$. (Once again this induces a second  loose $\Phi$-weighted cone which
factorizes through the first by the morphism $B\lto A$.)
\item Finally let $\bbE^{ll}$ be the \sF-category obtained from $\bbE^l$ by adjoining
an object $y$ and a loose morphism $y\lto x$. An \sF-functor $\bbE^{ll}\to\bbA$
consists of an \sF-functor $G:\bbD\to\bbA$, objects $A,B,C\in\bbA$, 
a $\Phi$-weighted cone $\eta:\Phi\to\bbA(A,G)$ and loose morphisms 
$C\lto B\lto A$.
\end{itemize}
There is a diagram of fully faithful \sF-functors
\[\xymatrix{
  \bbD \ar[r]^{M} \ar[d]_{H} &
  \bbD^{l} \ar@<1mm>[r]^{M^x} \ar@<-1mm>[r]_{M^{y}} \ar[d]_{H^{l}} & 
  \bbD^{ll} \ar[d]^{H^{ll}} \\
  \bbE \ar[r]_{N} &
  \bbE^{l} \ar@<1mm>[r]^{N^x} \ar@<-1mm>[r]_{N^{y}} &
  \bbE^{ll} }
\]
all of which are literal inclusions except for $M^y$ and $N^y$; these each send 
$x$ to $y$ and fix all other objects.
There is also an \sF-functor
$K\colon \bbD^l\to\bbE$ sending $x$ to $*$ and satisfying $K M=H$; and
there is an \sF-functor $P:\bbE^l\to\bbE$ satisfying $P N=1$ and
$P(x)=*$.

We now want to apply \autoref{lem:good} to conclude that the
profunctors $H_\bullet$, $(H^l)_\bullet$, and $(H^{ll})_\bullet$ all
lie in $\QF\Prof$.  In all three cases, the only object we have to
worry about is $*$, and there are no nonidentity tight morphisms with
this target, so the second condition of \autoref{lem:good} is vacuous.

For $H_\bullet$, the weight mentioned in the first condition is just
$\Phi$, which is assumed to be a \QF-coalgebra.  In the other two
cases, the weight in question is the left Kan extension of $\Phi$ to
$\bbD^l$ or $\bbD^{ll}$, respectively.  Thus, by
\autoref{thm:lan-coalg} it suffices to show that the functors $M$ and
$M^x M = M^y M$ satisfy the hypotheses of \autoref{lem:good}.  In both
cases, the second condition is again vacuous, while the weight we have
to check for the first condition is $\ls{\Phi}$, and this is easily
seen to inherit a \QF-structure from $\Phi$.

Thus, $H_\bullet$, $(H^l)_\bullet$, and $(H^{ll})_\bullet$ all lie in
$\QF\Prof$.  In particular, for any diagram $G\colon \bbD\to
T\bbAlg_l$, we have an induced diagram $\llim{H_\bullet,G}\colon
\bbE\to T\bbAlg_l$.  By \autoref{thm:unit-coalg} and
\autoref{lem:restriction}, the restriction of $\llim{H_\bullet,G}$ to
\bbD is $G$; hence it is a $\Phi$-weighted cone over $G$.  We aim to
show that it is a limit cone.

Given any $\ls{\Phi}$-weighted cone over $G$, seen as an
\sF-functor $F:\bbD^l\to T\bbAlg_w$, we have a canonical diagram
\[\llim{(H^l)_\bullet, F} \colon \bbE^l\to T\bbAlg_l.\]
As before, by \autoref{thm:unit-coalg} and \autoref{lem:restriction},
the restriction of $\llim{(H^l)_\bullet, F}$ along $H^l$ gives us $F$
back again; hence 
$\llim{(H^l)_\bullet, F}$ is actually a loose factorization of
$F$ through some $\Phi$-weighted cone.
Now we claim that the following diagram of profunctors commutes (up
to isomorphism) in $\QF\Prof$.
\begin{equation}
  \vcenter{\xymatrix{ \bbD^l \ar[r]|@{|}^{M^\bullet}
    \ar[d]|@{|}_{(H^l)_\bullet} &
    \bbD \ar[d]|@{|}^{H_\bullet} \\
    \bbE^l \ar[r]|@{|}_{N^\bullet} & \bbE }}\label{eq:profsq1}
\end{equation}
If this is so, then since $M^\bullet$ and $N^\bullet$ are given
by restriction, we will be able to conclude that the above
$\Phi$-weighted cone through which $F$ factors is actually the
putative limit cone $\llim{H_\bullet,G}$.

We leave to the reader the proof that~\eqref{eq:profsq1} commutes in
$\sF\Prof$ (and we will do likewise for all future such assertions).
For commutativity in $\QF\Prof$, it remains to check that the
\QF-coalgebra structures coincide.
But by \autoref{thm:prof-restriction}, for any $E\in\bbE$, the
$\QF^{\bbD^l}$-coalgebra structure of $(N^\bullet\otimes
(H^l)_\bullet)(E)$ is that induced by \autoref{lem:good} applied to
$H^l$, restricted to $\bbE$.  When $E\in\bbD$, this is the unique
structure of a representable, while for $E=*$ we took it to be the
left Kan extension of $\Phi$ (according to \autoref{thm:lan-coalg}).
But this left Kan extension is exactly what $H_\bullet \otimes
M^\bullet$ computes.  Hence the \QF-coalgebra structures agree,
and~\eqref{eq:profsq1} commutes in $\QF\Prof$.

We have shown that any $\ls{\Phi}$-weighted cone over $G$
factors through $\llim{H_\bullet,G}$ in a specified way; we next show
the naturality of these specified factorizations.  Let $\alpha\colon
F\lto F'$ be a loose morphism of $\ls{\Phi}$-weighted cones
over $G$.  We can regard this as a diagram $F^{(2)}$ of shape
$\bbD^{ll}$, and then form the $\bbE^{ll}$-diagram
$\llim{(H^{ll})_\bullet,F^{(2)}}$.  As before,
\autoref{thm:unit-coalg} and \autoref{lem:restriction} imply that
restricting this diagram along $H^{ll}$ gives us back $F^{(2)}$, so
that it consists of loose factorizations of $F$ and $F'$ through some
$\Phi$-weighted cone, and moreover these factorizations
commute with $\alpha$.  Now by an argument just like that given
above for~\eqref{eq:profsq1}, we can conclude that the diagrams of
profunctors
\[\vcenter{\xymatrix@C=3pc{
  \bbD^{ll} \ar[r]|@{|}^{(M^x)^\bullet} \ar[d]|@{|}_{(H^{ll})_\bullet} & 
  \bbD^l \ar[d]|@{|}^{(H^l)_\bullet} \\
  \bbE^{ll} \ar[r]|@{|}_{(N^x)^\bullet} &
  \bbE^l }}
\qquad\text{and}\qquad
\vcenter{\xymatrix@C=3pc{
  \bbD^{ll} \ar[r]|@{|}^{(M^{y})^\bullet} \ar[d]|@{|}_{(H^{ll})_\bullet} & 
  \bbD^l \ar[d]|@{|}^{(H^l)_\bullet} \\
  \bbE^{ll} \ar[r]|@{|}_{(N^{y})^\bullet} &
  \bbE^l }}
\]
commute in $\QF\Prof$.  This implies that the $\Phi$-weighted cone
appearing in $\llim{(H^{ll})_\bullet,F^{(2)}}$ must be
$\llim{H_\bullet,G}$, and the factorizations of $F$ and $F'$ through
it must be those produced by $\llim{(H^l)_\bullet,-}$.

Thus, $\llim{(H^l)_\bullet,-}$ gives us a natural
transformation from the identity functor of the 1-category of
$\ls{\Phi}$-weighted cones over $G$ to the functor constant at
$\llim{H_\bullet,G}$.  As is well-known, to conclude from this that
$\llim{H_\bullet,G}$ is a terminal object of this category (and hence
that factorizations through it are unique), it suffices to check that
the component of this transformation at $\llim{H_\bullet,G}$ itself is
the identity.  This will follow if we can show that the following
diagram of profunctors commutes in $\QF\Prof$.
\begin{equation}
  \vcenter{\xymatrix{ \bbD \ar[r]|@{|}^{H_\bullet} \ar[dr]|@{|}_{H_\bullet} &
    \bbE \ar[r]|@{|}^{K^\bullet} &
    \bbD^l \ar[d]|@{|}^{(H^l)_\bullet} \\
    & \bbE \ar[r]|@{|}_{P^\bullet} & \bbE^l }}\label{eq:profsq3}
\end{equation}
To show this, first note that for all $D\in\bbD$,
$(H^l)_\bullet(D) \cong \bbD^l(D,-)$, while $(H^l)_\bullet(x) \cong
\bbD^l(x,-)$.  Since left Kan extension preserves representables, we
also have
\begin{align*}
  ((H^l)_\bullet \otimes K^\bullet)(D) &\cong \bbE(D,-) \hspace{4cm}\text{and}\\
  ((H^l)_\bullet \otimes K^\bullet)(x) &\cong \bbE(K(x),-) = \bbE(*,-).
\end{align*}
But by definition we also have $P^\bullet(D) = \bbE(D,-)$ and $P^\bullet(x)
= \bbE(P(x),-) = \bbE(*,-)$, so that
\[((H^l)_\bullet \otimes K^\bullet)(E) \cong P^\bullet(E)\] as
\QF-coalgebras for all $E \neq *$, and this remains so after composing
with $H_\bullet$.

It remains to deal with $E=*$.  By definition, $(H^l)_\bullet(*)$ is
the left Kan extension of $\Phi$ to $\bbD^l$, with \QF-coalgebra
structure as in \autoref{thm:lan-coalg}; which is to say
$(H^l)_\bullet(*) \cong \Phi\otimes M^\bullet $.  Thus, we also have
\[((H^l)_\bullet \otimes K^\bullet)(*) \cong \Phi\otimes M^\bullet
\otimes K^\bullet \cong \Phi \otimes H^\bullet.
\]
and hence
\[((H^l)_\bullet \otimes K^\bullet \otimes H_\bullet )(*) \cong
\Phi \otimes H^\bullet \otimes H_\bullet \cong \Phi
\]
as \QF-coalgebras (using \autoref{thm:unit-coalg}).  But $P^\bullet(*)
= \bbE(*,-)$, and so also
\[ (P^\bullet \otimes H_\bullet)(*) \cong \bbE(*,H(-)) = \Phi \] as
\QF-coalgebras, since the structure of $H_\bullet$ as a morphism in
$\QF\Prof$ is induced by \autoref{lem:good} from the \QF-coalgebra
structure of $\Phi$.  This shows that~\eqref{eq:profsq3} commutes in
$\QF\Prof$, and hence factorizations of $\ls{\Phi}$-weighted cones
through $\llim{H_\bullet,G}$ are unique.

To complete the loose part of the universal property of
$\llim{H_\bullet,G}$, we need to deal with 2-cells.  Let
$\bbE^l_2$ be the \sF-category obtained from \bbE by adjoining an
object $x$ and two loose morphisms $\xymatrix@C=1.5pc{ x \ar@<1mm>@{~>}[r]
  \ar@<-1mm>@{~>}[r] & \ast}$ with a 2-cell between them.
Thus a diagram $\bbE^l_2 \to \bbK$ consists of $G\colon \bbD\to\bbK$,
a cone $\Phi \to \bbK(L,G)$, and a parallel pair of loose morphisms
$\xymatrix@C=1.5pc{ A \ar@<1mm>@{~>}[r] \ar@<-1mm>@{~>}[r] & L}$
with a 2-cell between them (inducing a
transformation between two loose cones $\Phi_\loose \to \sK(A,G)$).

Let $\bbD^l_2$ be the full subcategory of $\bbE^l_2$ on all the
objects except $*$; thus 
a diagram of shape $\bbD^l_2$ is a 2-cell between two
$\ls{\Phi}$-weighted cones over the same diagram.  We have
another diagram of fully faithful functors:
\[ \xymatrix{ \bbD \ar[r]^{M_2} \ar[d]_{H} & \bbD^l_2 \ar[d]^{H^l_2}\\
  \bbE \ar[r]_{N_2} & \bbE^l_2.}
\]
We claim that $H^l_2$ satisfies the hypotheses of \autoref{lem:good}.
The second condition is again vacuous, and the relevant weight for the
first condition is the left Kan extension of $\Phi$ to $\bbD^l_2$
along $M_2$; thus it suffices for $M_2$ to satisfy the hypotheses of
\autoref{lem:good}.  In this case the weight we have to check is
$\ls{\Phi} \times \bbtwo$, which becomes a \QF-coalgebra by lifting
the colimit $(-\times \bbtwo)$ to the category of \QF-coalgebras (as
we can do for any comonad).  Thus, by \autoref{lem:good},
$(H^l_2)_\bullet$ lies in $\QF\Prof$.

Now applying \autoref{thm:unit-coalg} and
\autoref{lem:restriction} again, $\{(H^l_2)_\bullet,-\}$ shows that any
2-cell between $\ls{\Phi}$-weighted cones factors through some
specified $\Phi$-weighted cone.  This $\Phi$-weighted cone will be
$\llim{H_\bullet,G}$ if we can show that the following diagram of
profunctors commutes in $\QF\Prof$.
\[ \xymatrix@C=3pc{ \bbD^l_2 \ar[d]|@{|}_{(H^l_2)_\bullet} \ar[r]|@{|}^{(M_2)^\bullet} &
  \bbD \ar[d]|@{|}^{H_\bullet} \\
  \bbE^l_2 \ar[r]|@{|}_{(N_2)^\bullet} &
  \bbE }
\]
As before, by restriction and the definition of the \QF-coalgebra
structure of $(H^l_2)_\bullet$, we have that $((N_2)^\bullet \otimes
(H^l_2)_\bullet)(*)$ is the left Kan extension of $\Phi$ along $M_2$
with structure as in \autoref{thm:lan-coalg}; but this is
exactly what $(H_\bullet \otimes (M_2)^\bullet)(*)$ computes.

Thus, any 2-cell between $\ls{\Phi}$-weighted cones factors through
$\llim{H_\bullet,G}$.  The uniqueness of this factorization is
automatic since $T\bbAlg_l \to \bbK$ is faithful on 2-cells.
\endproof

As we saw in \S\ref{sec:lifting}, the hypotheses of
\autoref{thm:coalg-looselift-app} are not strong enough to conclude that
the $T$-algebra structure induced on $\{\Phi,U G\}$ is actually the
$\Phi$-weighted \sF-limit of $G$; for that
we need $\Phi$ to be $w$-rigged.


\begin{thm}\label{thm:rigged-lift-app}
  If $\Phi$ is a $w$-rigged \sF-weight, then for  any \sF-monad $T$ on an
  \sF-category \bbK, the forgetful functor
  $U_w\maps T\bbAlg_w \to \bbK$ creates $\Phi$-weighted limits.
\end{thm}
\begin{proof}
  We must show that the limit projections corresponding to
  $\Phi_\tight$ are tight and detect tightness in $T\bbAlg_l$.  They
  are certainly tight, since $\llim{H_\bullet,G}$ is a $\Phi$-weighted
  cone and not merely a $\ls{\Phi}$-weighted one.  To show
  that they detect tightness, we continue the pattern of argument from
  the proof of \autoref{thm:coalg-looselift-app}.

  Let $\bbE$ be the collage of $\Phi$, as before, with $H\colon
  \bbD\to\bbE$ the inclusion, and let $\bbE^t$ be the \sF-category
  obtained from \bbE by adjoining an object $z$ and a \emph{tight}
  morphism $z\to *$.  Thus an \sF-functor $\bbE^t\to \bbA$ consists of
  a tight factorization of one $\Phi$-weighted cone through another.

  Let $V\colon \bbE\to\bbE^t$ be the inclusion (which in particular
  sends $*$ to $*$), and let $H^t\colon \bbE \to \bbE^t$ be the
  evident functor satisfying $H^t H = V H$ and $H^t(*) = z$.  Then
  $H^t$ satisfies the conditions of \autoref{lem:good}: the one weight
  we have to worry about is the left Kan extension of $\Phi$ to
  $\bbE$, which is rigged by \autoref{thm:lan-coalg}.  Thus,
  $(H^t)_\bullet$ lies in $\QF\Prof$.

  Therefore, using \autoref{thm:unit-coalg} and
  \autoref{lem:restriction} as before, we conclude that
  $\llim{(H^t)_\bullet,-}$ factors any $\Phi$-weighted cone over $G$
  through some other specified $\Phi$-weighted cone.  To show that the
  latter cone is in fact $\llim{H_\bullet,G}$, it suffices to show
  that the following diagram of profunctors commutes in $\QF\Prof$:
  \[\xymatrix{
    \bbE \ar[r]|@{|}^{H^\bullet} \ar[d]|@{|}_{(H^t)_\bullet} & 
    \bbD \ar[d]|@{|}^{H_\bullet} \\
    \bbE^t \ar[r]|@{|}_{V^\bullet} &
    \bbE }
  \]
  But since all the weights in question are now rigged,
  this  is automatic from its commutativity in
  $\sF\Prof$, by \autoref{thm:rigged-strictmor}.
\end{proof}

\bibliographystyle{alpha}

\end{document}